\documentclass[a4paper,10pt]{article} 
\usepackage{fancybox} 
\usepackage{pifont} 

\usepackage{fancybox} 

\usepackage{amsmath} 
\usepackage{amsfonts}
\usepackage{amsthm}
\usepackage{color}
\usepackage{mathrsfs} 

\newcommand{\mysection}{\setcounter{equation}{0} \section}

\renewcommand{\P}{\mathbb{P}}

\newcommand{\E}{\mathbb{E}}

\newcommand{\R}{\mathbb{R}}  
\newcommand{\I}{\mathbb{I}}  
\renewcommand{\b}{\mathbf{b}}

\newtheorem{defi}{Definition}
\newtheorem{THM}{Theorem}   
\newtheorem{REM}{Remark}

\newtheorem{PROP}[THM]{Proposition}
\newtheorem{lem}[THM]{Lemma}

\renewcommand{\S}{{\mathbb S}}

\def\1{\mbox{1\hspace{-0.25em}l}}

\newcommand \A[1]{{\bf (#1)}}

\def\leftB{[\![}
\def\rightB{]\!]}

\def\0{{\mathbf{0}}}

\def\p{{\partial}}

\title{
{\textbf{Schauder estimates for  drifted fractional operators in the supercritical case}}}

\author{\textbf{Paul-\'Eric Chaudru de Raynal}\footnote{Univ. Grenoble Alpes, Univ. Savoie Mont Blanc, CNRS, LAMA, 73000 Chamb\'ery, France. pe.deraynal@univ-smb.fr},
 \textbf{
St\'ephane Menozzi}\footnote{Laboratoire de Mod\'elisation Math\'ematique d'Evry (LaMME), Universit\'e d'Evry Val d'Essonne, 23 Boulevard de France 91037 Evry, France and Laboratory of Stochastic Analysis, HSE,
Shabolovka 31, Moscow, Russian Federation. stephane.menozzi@univ-evry.fr}
\  \textbf{ and Enrico Priola}\footnote{Universit\`a di Torino, Dipartimento di matematica ``Giuseppe Peano'', via Carlo Alberto, 10, 10123 Torino. enrico.priola@unito.it}
}

\begin{document}

\maketitle
\begin{abstract}
We 
 consider a non-local operator $L_{\textcolor{black}{ \alpha}}$ which is the sum of a fractional Laplacian $\triangle^{\alpha/2} $, $\alpha \in (0,1)$,  plus a first order  term which is    measurable   in the time variable
    and locally
$\beta$-H\"older continuous in the space variables. \textcolor{black}{Importantly, the fractional Laplacian $\Delta^{ \alpha/2} $  does not dominate the first order term.}
 We show that  global  parabolic Schauder estimates hold even in this case under the natural condition  $\alpha + \beta >1$. Thus, the constant appearing in the Schauder estimates is 
  in fact   independent of the $L^{\infty}$-norm of the first  order term.
In our approach  we  do not use the so-called extension property and 
 we can replace $\triangle^{\alpha/2}  $ with  other  operators of $\alpha$-stable type 
which are somehow close, \textcolor{black}{including the relativistic $\alpha$-stable operator}. 
Moreover, when $\alpha \in (1/2,1)$, we can prove Schauder estimates for more general $\alpha$-stable type operators like the  singular cylindrical one, i.e., when  $\triangle^{\alpha/2} $ is replaced by a sum of one dimensional fractional Laplacians
  $\sum_{k=1}^d
   (\partial_{x_k x_k}^2 )^{\alpha/2}$. 
\end{abstract}


\mysection{Statement of the problem and main results}

We are interested in establishing global Schauder estimates for the following parabolic  integro-partial differential equation (IPDE):
\begin{eqnarray}\label{Asso_PDE}
\p_t  u(t,x) +  {L}_{\alpha} u(t,x) + F(t,x) \cdot D_x u(t,x) &=& -f(t,x),\quad \text{on }[0,T)\times \R^d, \notag\\
u(T,x) &=& g(x),\quad \text{on }\R^d,
\end{eqnarray}
where $T>0$ is a fixed final time horizon.

The operator $L_{\alpha }$ can be the fractional Laplacian $\triangle^{\alpha/2}$, i.e., for regular functions $\varphi: \R^d \to \R$,
\begin{equation}\label{lap} 
\triangle^{\alpha/2} \varphi(x
) ={\textcolor{black}{ {\rm p.v.}}} \int_{\R^d} [\varphi(x+y)-\varphi(x)] \nu_{\textcolor{black}{\alpha}} (dy),
 \;\; \text{where}\; \,
\nu_{\textcolor{black}{\alpha}} (dy) = C_{\alpha,d}\frac{dy}{|y|^{d+\alpha}},
\end{equation}
 or a more general \textcolor{black}{\textit{symmetric}} non-local $\alpha$-stable  operator with   \textcolor{black}{ symbol comparable to $-|\lambda|^\alpha $ but associated with a wider class of L\'evy measures $\nu_\alpha$  (see Section \ref{SEC_SETTING}  for our  precise assumptions). We can also consider some \textit{non-symmetric}  stable operators like the relativistic stable one with symbol $-\big( |\lambda|^2 + m^{\frac{2}{\alpha}} \big)^{\frac{\alpha}{2}} + m $, $m>0$.
 }

In \eqref{Asso_PDE}, the source 
$f: [0,T]\times \R^d \to \R$ and terminal condition $g:\R^d\to \R$ are assumed to belong to some suitable H\"older spaces and to be bounded.

 The drift term $F$ can  be unbounded. It 
is only assumed to be (locally) $\beta$-H\"older continuous, $\beta \in (0,1)$, i.e.
$F : [0,T] \times \R^d  \to \R^d$ is Borel measurable, locally bounded 
and  there exists ${K_{0}}>0$ such that 
\begin{equation}\label{22}
|F(t,x) - F(t,y)| \le {K_{0}} |x-y|^{\beta},\;\;\; t \in [0,T],\;\; x,y \in \R^d\ s.t. \ |x-y|\le 1.
\end{equation}

In particular,   we concentrate on the so-called \textit{super-critical} case, i.e., $\alpha \in (0,1) $, {\color{black} although our estimates can be extended to the simpler case $\alpha \in [1,2)$
}. 
The difficulty 
is quite clear: \textcolor{black}{in the Fourier space}, $L_{\alpha}$  is of order $\alpha$ and {\it does not dominate}, \textcolor{black}{when $\alpha\in (0,1) $}, the drift  term which is roughly speaking of  order one 
(\textcolor{black}{see also Remark 3.5 in \cite{prio:12}} for related issues).

In order to compensate the \textit{low} smoothing effect of $L_\alpha $, it is natural to ask \textit{more} on the H\"older exponent $\beta $ of the drift. Namely, we need that the gradient $D_x u(t,x)$ in \eqref{Asso_PDE} exists in the classical sense. To this end, {\color{black} since the smoothing effect of $L_\alpha$ on the $\beta$-H\"older source $f$ is \textit{expected} to be of order $(\alpha+\beta)$ in space, it is natural to consider $\alpha + \beta >1$. This condition also appears from a probabilistic viewpoint; it had indeed already been observed in the scalar case by Tanaka \textit{et al.} \cite{tana:tsuc:wata:74} that uniqueness might fail for the corresponding SDE when $\alpha+\beta\le 1 $.}

\textcolor{black}{In the previously described  framework, we obtain estimates like}
\begin{equation}
\label{EST_SCHAU11}
\|u\|_{L^\infty([0,T],C_b^{\alpha+\beta})}\le C(\|g\|_{C_b^{\alpha+\beta}}+\|f\|_{L^\infty([0,T],C_b^{\beta})}),
\end{equation}
\textcolor{black}{with usual notations for H\"older spaces, where $C$ is independent of $u$, $f$ and $g$.}

{\color{black}
An interesting example covered by our assumptions  is the non-local Ornstein-Uhlenbeck operator
\begin{equation}\label{ou1} 
\triangle^{\alpha/2} \varphi (x) + Ax \cdot D_x \varphi(x),
\end{equation}
when $F(t,x) = Ax$ and $A$ is any  $d \times d$ real matrix (in this case assumption \eqref{22}  holds for any $\beta \in (0,1)$). If $\alpha =2$ Schauder estimates where first proved by Da Prato and Lunardi \cite{DL95}. After that paper  the OU operator has been much investigated as  a prototype of operator with unbounded coefficients. \vskip 2mm
}

{\color{black}\noindent\textbf{Related results.}} Schauder estimates in the $\alpha$-stable non-local framework have been addressed by several authors,   always assuming that the drift term is globally bounded and  mainly assuming that $\alpha \ge 1$. 
  In \textcolor{black}{some}  papers,  the L\'evy measure $\nu_{\alpha}$ may also depend on $t$ and $x$. We mention \textcolor{black}{for instance}  the so-called stable-like setting, \textcolor{black}{corresponding to}  time-inhomogeneous operators of the form
\begin{equation}
\label{STAB_LIKE}
L_t\varphi(x)=\int_{\R^d} [\varphi(x+y)-\varphi(x)- \I_{\alpha\in [1,2)}  D_x \varphi(x) \cdot y ] m(t,x,y) \frac{dy}{|y|^{d+\alpha}}\, + \,  F(t,x) \cdot  D_xu(t,x) \I_{\alpha\in [1,2)}, 
\end{equation}
where the \textit{diffusion coefficient} $m $ is bounded from above and below, H\"older continuous in the spatial variable, and even in the $y$ variable for $\alpha=1$. \textcolor{black}{In that framework}, 
Mikulevicius and  Pragarauskas \cite{miku:prag:14} obtained parabolic Schauder type bounds on the whole space
 and derived from those estimates the well-posedness of the corresponding martingale problem. 
Observe that, in \eqref{STAB_LIKE}, in the super-critical case $\alpha\in (0,1) $, the drift \textcolor{black}{term} is set to 0. Again, this is mainly due to the fact that, in that case, the drift cannot be viewed anymore as a lower order perturbation of the fractional \textcolor{black}{operator}.
\vskip 1mm

In the driftless framework, the (elliptic)-Schauder type estimates for stable-like operators were first derived by Bass \cite{bass:09}. We can refer as well to the recent work of Imbert \textit{et al.} \cite{imbe:jin:shvy:16} concerning Schauder estimates for a driftless stable like operator of type \eqref{STAB_LIKE} for $\alpha=1 $ and some non-standard diffusion coefficients $m$ with applications to a non-local Burgers type equation. Eventually, still for $F=0$, in the general
  non-degenerate symmetric $\alpha$-stable setting,
for which $\nu_\alpha $ writes   in polar coordinates $y=\rho s $, $(\rho,s)\in \R_+\times {\mathbb S}^{d-1} $
  as 
\begin{equation}\label{DECOMP_NUo}
\nu_{\alpha}(dy)= \rho^{-1-\alpha} \, d\rho \, \tilde \mu(ds), 
\end{equation}
where $\tilde \mu $ is a non-degenerate symmetric measure on the sphere ${\mathbb S}^{d-1} $,
we can also mention the works of Ros-Oton and Serra \cite{roso:serr:16} for interior and boundary elliptic-regularity and Fernandez-Real and Ros-Oton \cite{fern:roso:17} for parabolic equations. We can also refer to Kim and Kim \cite{kim:kim:15} for results on the whole space involving more general, but rotationally invariant, L\'evy measures, \textcolor{black}{or to Dong and Kim \cite{dong:kim:13} for stable like measures that might be non-symmetric and non-regular w.r.t the jump parameter}.


\vskip 1mm
In  the  elliptic setting,  when $\alpha\in [1,2) $ and  $L_{\alpha }$ is a non-degenerate symmetric $\alpha$-stable operator and for bounded H\"older drifts, global
Schauder estimates were obtained 
 by Priola, see e.g. Section 3 in \cite{prio:12} and \cite{prio:18} with respective applications to the strong well-posedness and Davie's uniqueness for the corresponding SDE. 
 Also, when $\alpha\in [1,2) $,
 elliptic Schauder estimates can be proved for more  general  L\'evy-type generators  invariant for translations, see Section 6 in \cite{prio:18} and  Remark \ref{chissa}.

 \vskip 1mm
%
 For a non trivial, and potentially \textit{rough}, drift, there is a rather large literature concerning the regularity of \eqref{Asso_PDE} when the drift  (possibly depending on the solution) 
 is divergence free, i.e., $\nabla \cdot F(t,x) $=0, in connection with the quasi-geostrophic equation even for $\alpha\in (0,1) $. We can mention the seminal work of Caffarelli and Vasseur \cite{caff:vass:10} and the work of Silvestre \textit{et al.} \cite{silv:vico:zlat:13} which exhibits counter-examples to regularity  of \eqref{Asso_PDE} when the terminal condition lacks good integrability properties or when the condition $\alpha+\beta>1 $ is not met (see Theorem 1.1 therein). \textcolor{black}{Conditions on divergence free drifts $F$ in Morrey-Campanato or Besov spaces giving the H\"older continuity of \eqref{Asso_PDE} are discussed in \cite{cham:meno:16} and \cite{cham:meno:18}}.\\


When, $\alpha \in (0,1) $ and $F$ is H\"older continous and bounded (\textcolor{black}{but not necessarily divergence free}),  Silvestre obtained in \cite{silv:12} sharp Schauder estimates on balls for the 
fractional Laplacian. His approach heavily relies on the so-called extension property, see \cite{caff:silv:07} or \cite{molc:ostr:69} for a more probabilistic approach, and therefore seems rather delicate to extend  to more general operators of stable type or with varying coefficients. Also, it seems that  our result when ${L}_{\alpha} = \triangle^{\alpha/2}$  and   $F$ satisfying \eqref{22}  cannot be obtained from the estimates by Silvestre using a standard  covering argument;
indeed  the Schauder constant in \cite{silv:12}  also depends  on the global boundedness of $F$. We can mention as well the recent work of Zhang and Zhao \cite{zhan:zhao:18} who address through probabilistic arguments the parabolic Dirichlet problem in the super-critical case for stable-like operators of the form \eqref{STAB_LIKE} with a non trivial bounded drift, i.e.,  getting rid of the indicator function for the drift. They \textcolor{black}{also} obtain interior Schauder estimates and some boundary decay estimates (see e.g. Theorem 1.5 therein).
\vskip 2mm

{\color{black}\noindent\textbf{Outline for Schauder estimates through perturbative approach.}} {\color{black} In this work we will establish global Schauder estimates for the solution of \eqref{Asso_PDE} inspired by the perturbative approach first introduced in \cite{chau:hono:meno:18} to derive \textcolor{black}{such} estimates in anisotropic H\"older spaces for degenerate Kolmogorov equations. 

\vskip 1mm
Roughly speaking, the main steps of \textcolor{black}{our} perturbative approach are the following: choose \textcolor{black}{first} a suitable \textcolor{black}{\textit{proxy}} for the main equation (\emph{i.e.,} an \textcolor{black}{integro-partial differential operator} whose associated semi-group and heat kernel are known and close enough to the original one), exhibit \textcolor{black}{then} suitable regularization propert\textcolor{black}{ies} associated with the proxy,  expand \textcolor{black}{consequently} the solution of the IPDE of interest around the proxy (Duhamel \textcolor{black}{type formula} or variation of constants formula) and eventually use such a representation to obtain Schauder estimates. Let us emphasize that the derivation of a robust Duhamel representation for the IPDE is crucial in this approach.

More precisely, the perturbative argument take\textcolor{black}{s} here the following form: first choose \textcolor{black}{a} flow $ \theta_{s,\tau}(\xi)=\xi+\int_{\tau}^sF(v,\theta_{v,\tau}(\xi))dv$ depending on parameters $\xi \in \R^d$ and $\tau \in [0,T]$ to be chosen carefully and introduce then   
 the \textcolor{black}{ time inhomogeneous} drift  $F(t, \theta_{t,\tau}(\xi) )$ \textcolor{black}{\textit{frozen}} \textit{along the considered flow}. \textcolor{black}{Rewrite then \eqref{Asso_PDE} as}  
  \begin{eqnarray} \label{we_INTRO}
\p_t  u(t,x) +  {L}_{\alpha} u(t,x) + F(t, \theta_{t,\tau}(\xi) ) \cdot D_x u(t,x) &=& -f(t,x) + 
 [F(t, \theta_{t,\tau}(\xi) )- F(t,x)]  \cdot D_x u(t,x)
,\, 
\notag\\
u(T,x) &=& g(x),
\quad \text{on }\R^d.
\end{eqnarray}
This system reflects more or less the main ingredients needed for our perturbative approach: the \textcolor{black}{integro-partial differential operator} in the above l.h.s. will be our proxy 
which is hence a \emph{frozen} version \textcolor{black}{of the operator in} \eqref{Asso_PDE}, where the freezing is done along the chosen flow, and the second term in the r.h.s. is precisely the error made when expanding the \textcolor{black}{solution} around the proxy. Roughly speaking, by the  Duhamel principle  we get a representation formula for the solution $u$ and we can perform  estimates  by choosing the proxy parameters $\tau$ and $\xi$. In this respect it is useful to look at the proof of Proposition   \ref{DUHAMEL} and in particular to the derivation of estimates \eqref{INTEGRABILITY_BOUND_WITH_RIGHT_PARAMETERS_GRAD} and \eqref{THE_SWAP_LIMITS}. On the other hand, a more general Duhamel  formula is needed in Section 2.4.2 to complete the proof of Schauder estimates. 

At this stage, let us eventually mention that when dealing with unbounded first order coefficients, the previous associated flow
is a rather natural object to consider in order to establish Schauder estimate\textcolor{black}{s} and was already used by Krylov and Priola  \cite{kryl:prio:10} in the diffusive setting.

\vskip 1mm
In comparison with \cite{chau:hono:meno:18}, where the main difficulties encountered \textcolor{black}{consisted in handling} the degeneracy of the operator and its associated anisotropic behavior (while the derivation of a Duhamel representation as well as the existence of a solution were the easier parts), we here face different \textcolor{black}{problems}, especially when trying to obtain a suitable Duhamel representation or when dealing with the existence part. Such difficulties come from two main features of our framework: the stable operator $L_\alpha $ induces major integrability issues and we consider drift terms that are only locally H\"older continuous (see again \eqref{22}).

To overcome these particularit\textcolor{black}{ies},} we introduce a \textcolor{black}{\textit{localized}} version of \eqref{we_INTRO}. 
 The point is to multiply $u$ by a suitable \textit{localizing} test function $\eta_{\tau, \xi} $ where $(\tau,\xi) $ are freezing parameters and to establish a Duhamel type representation formula for $u \eta_{\tau, \xi}$ (cf.  \textcolor{black}{equation \eqref{we}}). \textcolor{black}{We point out that, in our current setting, this localization is not \textit{simply} motivated by the fact to get weaker assumptions on $F$ (i.e., from \textit{global} to \textit{local} H\"older continuity). Indeed, even when $L_\alpha=\Delta^{\frac \alpha 2} $ for $\alpha\in (0,1/2) $, it is also needed to give a proper meaning to the Duhamel representation of the solution because of the low integrability properties of the underlying heat-kernel (see again Proposition \ref{DUHAMEL} and its proof)}.
Let us also emphasize that, the key to perform our analysis consists in having \textit{good controls} on the heat kernel (\textcolor{black}{or density}) $p_\alpha $ associated with $L_\alpha$ and some of its spatial derivatives (cf.  \A{NDb} in Section \ref{SEC_OP}). 

{\color{black} This will for instance be the case \textcolor{black}{when} the spherical measure 
 $\tilde \mu$ in \eqref{DECOMP_NUo} has a smooth density w.r.t. the Lebesgue measure of ${\mathbb S}^{d-1},$ following the work of Kolkoltsov \cite{kolo:97}. }
Roughly speaking, in \textcolor{black}{that framework}, the heat-kernel {\color{black} $p_{\alpha}$} associated with $L_\alpha$, and its first two derivatives, will behave \textit{similarly} to the rotationally invariant  {\color{black} density of  
 $\Delta^{\frac \alpha 2}$} for which we have precise pointwise controls. In this framework we establish Schauder estimates for any $\alpha\in (0,1) $ and $\beta\in (0,1) $ s.t. $\beta+\alpha>1 $ for the potentially unbounded drift satisfying \eqref{22}.

On the other hand, {\color{black} in the case of  }
more general, and possibly singular, fractional operators {\color{black} of symmetric stable type}, following the approach initiated by Watanabe \cite{wata:07} and also used in Huang \textit{et al.} \cite{huan:meno:prio:19} consisting in treating separately the small and large jumps for the considered characteristic time scale, {\color{black} we have an additional constraint.
 We are only able to derive} that the {\color{black} spatial} derivatives of the heat kernel $p_\alpha(t,z) $ {\color{black} \textcolor{black}{(}which have the expected additional time singularity associated with the derivation order\textcolor{black}{)},} can integrate $ z\mapsto |z|^\beta$, i.e.,  $\int_{\R^d} |z|^\beta |D_z^k p_\alpha(t,z)| dz<\infty,\ k\in \{1,2\}$, provided $\beta<\alpha $, {\color{black} $t>0$.} The constraint $\alpha+\beta>1 $ then gives that we can handle in {\color{black} this} general super-critical case, indexes $\alpha\in (1/2,1) $. The difference between {\color{black} the previous  two cases can be intuitively explained as follows: for the fractional Laplacian the derivation of $p_{\alpha}(t, \cdot)$} induces a concentration gain at infinity, see e.g. Bogdan and Jakubowicz \cite{bogd:jaku:07}, which precisely permits to get rid of the integrability constraints that {\color{black} we have to face;  for operators whose symbol is equivalent to $|\xi|^\alpha $ but whose L\'evy  measure $\nu_{\alpha}$  has a very singular spherical part, we do not have such concentration gain (cf. Remark \ref{dff})}.

Eventually, our approach will also allow to handle {\color{black}  stable fractional truncated operators viewing  the difference between the truncated and the non-truncated operators as a bounded  perturbative term under control}.

\vskip 2mm
{\color{black}\noindent\textbf{Organization of this paper.}} \textcolor{black}{The article is organized as follows. We state our precise framework and give our main results at the end of the current section. Section \ref{PERT} is then dedicated to the perturbative approach which is the central point to derive our estimates. In particular, we obtain therein some Schauder estimates for drifted operators along the inhomogeneous flow as well as the key Duhamel representation for solutions. We establish in Section \ref{ESISTENZA}  existence results. Eventually, Section \ref{SEC_PROOF_PROP_P} is devoted to the derivation in the previously described cases, i.e., stable-like and general stable operator for $\alpha\in (1/2,1) $, of the properties required to obtain our main estimates. The proof of some technical results concerning the stability properties of non-Lipschitz flows are postponed to Appendix \ref{APP_TEC}}. 


\subsection{Setting}\label{SEC_SETTING}

\subsubsection{Operators considered}\label{SEC_OP}


We consider a L\'evy generator   ${L}$ such that  
 for $\phi \in C_0^\infty(\R^d)  $, where $C_0^\infty(\R^d)$ stands for the space of real-valued infinitely differentiable functions with compact support one has:
\begin{equation}
\label{STAB_OPERATOR}
L \phi(x)=\int_{\R^d}\big( \phi(x+y )-\phi(x) \big) \nu(dy),\;\;\; x\in \R^d,
\end{equation}
where  $\nu $ is a Borel measure on $\R^d$ such that 
$\int_{\R^d} (1 \wedge |x|) \nu(dx) < \infty$, and $\nu (\{ 0\})=0$ ($\nu$ is an example of L\'evy measure).  
It is well known (see e.g. Sato \cite{sato:99}) that there exists a convolution Markov semigroup $(P_t)$ associated with $L$: 
\begin{equation}\label{sd3}
P_t h(x) = \int_{\R^d} h(x+ y ) \mu_t (dy),
 \;\;  h \in B_b(\R^d), \,\; t>0,\;\; x \in \R^d,
\end{equation}
$P_0 =I$, where $(\mu_t)$ is a family of Borel probability measures on $\R^d$ and $B_b(\R^d)$ stands for the set of real-valued bounded measurable functions.
 The function $v(t,x)= P_t \phi (x)$ provides the
 classical solution to the Cauchy problem  
\begin{equation}\label{c11}
  \p_t  v(t,x) =  {L} v(t,x) = Lv(t, \cdot)(x),\;\; t>0,\;\;
v(0,x) = \phi(x)
\quad \text{on }\R^d.
 \end{equation}
In probabilistic term,  $\mu_t$  is the distribution at time $t\ge 0$ of 
 a purely jump L\'evy process  $(Z_t)_{t\ge 0}$.

\vskip 1mm  \noindent \A{NDa}
We assume that $\mu_t$ has a $C^2$-density $p(t, \cdot)$, $t>0$, and that there exists $\alpha \in (0,1)$ such that
if $0<\gamma < \alpha$:  
\begin{equation}\label{inv}
\int_{\R^d} |y|^{\gamma}  p(t, y) dy \le c\,  t^{\gamma/ \alpha}, \;\;\; t \in [0,1]. 
\end{equation}
 for some $c= c(\gamma, \alpha) >0$.  

 In the sequel we write $L = L_{\alpha}$, $\nu = \nu_{\alpha}$
 $P_t = P_t^{\alpha}$ and 
 $p= p_{\alpha}$ {\color{black} in order to \textcolor{black}{explicitly emphasize} the dependence of these \textcolor{black}{objects} w.r.t. the parameter $\alpha$}. \\
 
 \noindent \A{NDb} To prove Schauder estimates with H\"older index $\beta \in (0,1)$, beside the condition $\alpha + \beta >1$, we need the following smoothing effect: 
 there exists a constant $c= c( \alpha, \beta) >0$
such that 
\begin{equation*}
\tag{${\mathscr P}_\beta $}
\int_{\R^d} |y|^{\beta}\, | D^{k}_y p_{\alpha}(t, y)| \, dy \le \frac{c}{t^{[k- \beta]/\alpha}},\;\;\; t \in (0,1],\;\;\; k=1,2. 
\end{equation*}
where $D^{1}_y p_{\alpha}(t, y) = D_y p_{\alpha}(t, y)$ and $D^{2}_y p_{\alpha}(t, y)$ denote the first and second derivatives in the $y$-variable. \qed

\vskip 2mm
{\color{black}
 \begin{REM} {\em It is  known that in the case of the fractional Laplacian $L_{\alpha}= \triangle^{\alpha/2}$  assumption 
 $({\mathscr P}_\beta)$  is always verified for any $\beta \in (0,1)$ and $\alpha \in (0,1)$. On the other hand, in the more general class of non-degenerate symmetric stable operators, $({\mathscr P}_\beta)$ holds only if $0< \beta < \alpha$ (see Proposition \ref{CTR_OF_INTEGRABILITY_DER_GEN_STABLE}). This condition together with $\alpha + \beta >1$ imposes $\alpha > 1/2$. \textcolor{black}{We also manage to establish $({\mathscr P}_\beta) $, $\beta\in (0,1) $ for the
 non-symmetric relativistic stable operator}.
}
\end{REM}
}

 \begin{REM} {\em 
 We mention that \A{NDb} is specifically needed to handle the remainder perturbative term in the r.h.s. of \eqref{we_INTRO} and can be viewed as a  \textit{sufficient} condition to cope with the supercritical case.}
 \end{REM}

\subsubsection{ Non-degenerate symmetric stable operators}
We {\color{black} now} introduce a class  of operators $L_{\alpha}$ which verify \A{NDa} and \A{NDb}. {\color{black}These operators} ${L}_{\alpha}$ will be the generator\textcolor{black}{s} of non-degenerate symmetric stable process\textcolor{black}{es}, i.e., $L_{\alpha} $ can be represented by \eqref{STAB_OPERATOR} where  $\nu = \nu_{\alpha} $ is a symmetric stable L\'evy measure of order $\alpha\in (0,1) $. If we now write in polar coordinates $y=\rho s $, $(\rho,s)\in \R_+\times {\mathbb S}^{d-1} $, the previous measure $\nu_{\alpha} $ decomposes as 
\begin{equation}\label{DECOMP_NU}
\nu_{\alpha}(dy)= \frac{d\rho \tilde \mu(ds)}{\rho^{1+\alpha}},
\end{equation}
where $\tilde \mu $ is a symmetric measure on the ${\mathbb S}^{d-1} $ which is a spherical part of $\nu_{\alpha} $. Again, if $\tilde \mu $ is precisely the Lebsegue measure on the sphere, then $L_\alpha={\Delta}^{\frac \alpha 2} $. It is easy to verify that $\int_{\R^d} (1 \wedge |x|) \nu_{\alpha}(dx) < \infty.$

The L\'evy symbol  associated with ${L}_{\alpha}$ is given by the L\'evy-Khintchine formula
\begin{equation}\label{LEVY_KHINTCHINE} 
\Psi(\lambda)= 
 \int_{\R^d} \big(  e^{i \langle  \lambda, y \rangle }  - 1
\big ) 
\nu_{\alpha} (dy),\;\; \lambda \in \R^d,
\end{equation}
\textcolor{black}{where $\langle \cdot ,\cdot \rangle$ denotes the Euclidean scalar product on $\R^d $}  (see, for instance Jacob \cite{jaco:05} or Sato \cite{sato:99}). In the current symmetric setting,  Theorem 14.10 in \cite{sato:99} then yields:
 \begin{equation}\label{LEVY_STABLE}
\Psi(\lambda)=- \int_{{\mathbb S}^{d-1}} |\langle \lambda,s\rangle|^\alpha \mu(ds),
\end{equation}
where $\mu = C_{\alpha,d} \tilde{\mu}$ for a positive constant $C_{\alpha,d}$. The spherical measure $\mu $ is called the \textit{spectral measure} associated with $\nu_{\alpha} $. We suppose that $\mu $ is non-degenerate, i.e., 
there exists $\eta\ge1$  s.t. for all $\lambda\in \R^d $,
\begin{equation}
\label{EQ_ND}
\eta^{-1}|\lambda|^\alpha \le  \int_{{\mathbb S}^{d-1}} |\langle \lambda,s\rangle|^\alpha \mu(ds) \le \eta|\lambda|^\alpha,\;\; \alpha \in (0,1).
\end{equation}
 We carefully point out that condition \eqref{EQ_ND} is fulfilled by many types of spherical measures $\mu $, from measure equivalent to the Lebesgue measure on ${\mathbb S}^{d-1} $ (a \textit{stable-like} case) to very singular ones, like sum of Dirac masses along the canonical directions, which would correspond to the pure cylindrical case (or equivalently to the sum of scalar fractional Laplacians):
 \begin{equation}\label{cil}
 \sum_{k=1}^d
   (\partial_{x_k x_k}^2 )^{\alpha/2}.
\end{equation}
For symmetric stable operators under \eqref{EQ_ND}, it is well known (see e.g. \cite{kolo:97}) that the associated convolution Markov semigroup $(P_t^{\alpha})$ (see \eqref{sd3}) has a $C^{\infty}$-smooth density $p_{\alpha} (t, \cdot)$.   
 Through Fourier inversion, we get for all $t>0, y\in \R^d $:  
\begin{equation}\label{DENSITY_FTI}
p_\alpha(t,y)=\frac{1}{(2\pi)^{d}} \int_{\R^d}\exp\big(-\langle y,\lambda\rangle\big) \exp\Big(-t\int_{{\mathbb S}^{d-1}}|\langle \lambda,s\rangle|^\alpha \mu(ds) \Big) d\lambda.
\end{equation}
From \eqref{DENSITY_FTI} and \eqref{EQ_ND} we derive directly \A{NDa}. Indeed we have the following scaling property:  $p_{\alpha}(t,y) = {t^{-d/ \alpha}} \, p_{\alpha} (1, t^{- 1/\alpha} y)$, $ t>0$, $ y \in \R^d$. 


Moreover, {\color{black} we have}  the following  global upper bound for the derivatives of the heat-kernel{\color{black} :} there exists $C:=C(\eta)$ s.t. for all $k\in \{0,1,2\}$ and for all $t>0,\ y\in \R^d $,
\begin{eqnarray}
\label{BD_GLOB_GRAD}
|D_x^k p_\alpha(t,y)|\le \frac{C}{t^{\frac{d+k}{\alpha}}}, \;\; t>0.
\end{eqnarray}
which in turns yields with the notations of \eqref{sd3}: 
$\forall t>0,\ x\in \R^d,\ |D_x^kP_t^{\alpha} h(x)| \le {\color{black} C t^{-(d+k)/\alpha}}  \|h\|_\infty,\;\; t>0. 
$

The validity of \A{NDb} for  general symmetric non-degenerate stable operators follows by the next result (note that in this case we have the property for any $t>0$ and not only for $t\in (0,1]$). 
\begin{PROP}
\label{CTR_OF_INTEGRABILITY_DER_GEN_STABLE} Let $\alpha \in (0,1)$.
Assume \eqref{EQ_ND} holds, then  for any $0 \le \gamma<\alpha $, there exists $C:=C(\eta,\gamma)$ s.t. for all $\ell\in \{1,2\} $,  
 \begin{gather} \label{wss0}
 \int_{\R^d} |y|^{\gamma}\, | D^{\ell}_y p_{\alpha}(t, y)| \, dy \le \frac{C}{t^{[\ell- \gamma]/\alpha}},\;\; t>0.
\end{gather}
\end{PROP}
This proposition can be proved following the arguments of Lemma 4.2 in \cite{huan:meno:prio:19}. A complete proof is provided in Section \ref{SEC_PROOF_PROP_P}  for the sake of completeness. 
Importantly, it gives for $\beta = \gamma$ with $\alpha+\beta>1 $ 
the constraint 
$\alpha>1/2 $ in Schauder estimates for any non-degenerate symmetric stable operator.

\vskip 1mm For a class of more regular non-degenerate symmetric stable operators   including  the fractional Laplacian we have the following better result.  
\begin{PROP}
\label{CTR_OF_INTEGRABILITY_DER_GEN_STABLE1} Let $\alpha \in (0,1)$.
Assume \eqref{EQ_ND} holds and that  the spectral measure $\mu  $ has a \textit{smooth} density equivalent to the Lebesgue on ${\mathbb S}^{d-1} $. Then for any $\gamma \in [0,1]$, there exists $C:=C(\eta,\gamma, \mu)$ s.t. for all $\ell\in \{1,2\} $,  
\begin{gather} \label{wss}
\int_{\R^d} |y|^{\gamma}\, | D^{\ell}_y p_{\alpha}(t, y)| \, dy \le \frac{C}{t^{[\ell- \gamma]/\alpha}},\;\; t>0.
\end{gather}
\end{PROP}
This result  can be derived using the estimates of Kolokoltsov \cite{kolo:97}. It extends to a wider class of spectral measure what was already known for the fractional Laplacian itself. Namely, since the derivatives of the associated heat-kernel exhibit  a decay improvement at infinity, see e.g. Bogdan  and Jakubowicz \cite{bogd:jaku:07}, then \eqref{wss} holds for any $\gamma \in [0,1] $. We prove this in Section 4.

Thus for $\beta = \gamma$ with $\alpha+\beta>1 $ we can prove 
 Schauder estimates for operators $L_{\alpha}$ as in Proposition \ref{CTR_OF_INTEGRABILITY_DER_GEN_STABLE1} for any $\alpha \in (0,1)$.

 \begin{REM} {\em We will be also able to treat a non-degenerate symmetric stable operator perturbed by a  bounded term after proving Schauder estimates for the non-degenerate symmetric stable operator; this is why we can also consider truncated stable operators. Namely, for a given truncation threshold $K>0$, we can consider as well, with the notations of \eqref{DECOMP_NU}, a L\'evy measure of the form 
\begin{equation}\label{DECOMP_NU_TRONC}
\nu_{\alpha, K}(dy)=\frac{d\rho \tilde \mu(ds)}{\rho^{1+\alpha}} \I_{\rho\in (0,K]}.
\end{equation}
   }
  \end{REM}

\begin{REM} \label{dff}
{\em We believe that formula \eqref{wss} with $\gamma \in [0,1]$ may hold for other class of stable operators. Indeed 
   estimates  on the derivatives of  densities of stable-like operators  which could be useful to establish   \eqref{wss} are already given in \cite{peng:17}.

On the other hand, one can check that \eqref{wss0} does not hold for $\gamma = \alpha$ in the case of a cylindrical fractional Laplacian $\sum_{k=1}^2
   (\partial_{x_k x_k}^2 )^{\alpha/2}.$ Indeed in such  case the density $p_{\alpha} (t,x) = $ $ q_{\alpha} (t,x_1) q_{\alpha} (t,x_2)$ ($q_{\alpha}(t,r)$ is the density of a one-dimensional fractional Laplacian) and so, for $t>0,$
$$
 \int_{\R^2} |y|^{\alpha} |D p_{\alpha}(t,y)| dy \;\;\; \text{is not finite.  }
$$
  }
  \end{REM}

\def\ciao2 { 

The stable-like case, when  for $\alpha\in (0,1) $ and $\alpha+\beta>1 $. 
In whole generality, for given parameters $\alpha,\beta $ as above, we are able to prove our main results provided that $({\mathscr P}_\beta) $ holds. We can prove it in two main cases:
\begin{trivlist}
\item[-] The stable-like case, when  the spectral measure $\mu  $ has a \textit{smooth} density equivalent to the Lebesgue on ${\mathbb S}^{d-1} $, for $\alpha\in (0,1) $ and $\alpha+\beta>1 $. 
\item[-] General stable operator, with potentially singular spectral measures satisfying \A{ND} (typically the cylindrical fractional Laplacian), when $\alpha \in (1/2,1)$ and $\alpha+\beta>1 $.
\end{trivlist}
The second case is, as indicated above, a direct consequence of Proposition \ref{CTR_OF_INTEGRABILITY_DER_GEN_STABLE} whereas \\

We refer to Section \ref{SEC_PROOF_PROP_P} for the proof of Property $({\mathscr P}_\beta)$ in the indicated  cases.

We believe that ????

}

\subsubsection{Relativistic Stable Operators}
Let us consider here $L_{\alpha} $ corresponding to the relativistic stable operator with symbol 
\begin{gather} \label{symb_rel1}
\Psi (\lambda ) = -\big( |\lambda|^2 + m^{\frac{2}{\alpha}} \big)^{\frac{\alpha}{2}} + m,
\end{gather}
for some $m>0$, $\alpha \in (0,1)$, $\lambda \in \R^d$. 
  It appears to be an important object in the study of relativistic Schr\"odinger operators (see  \cite{ryzn:02} and also the references therein).  
  
 The operator  $L_{\alpha} $ can be represented by \eqref{STAB_OPERATOR} with   $\nu = \nu_{\alpha, m} $  which  has   density
$
 { C_{\alpha,d}}{|x|^{-d- \alpha}} \, e^{- m ^{1/\alpha} \, |x|} \, $ $\cdot \, \phi ( m ^{1/\alpha} \, |x|), $ $ x \not = 0,
 $
with $0 \le \phi (s) \le c_{\alpha,d,m} (s^{\frac{d-1+\alpha}{2} } +1)$, $s \ge 0$ (see Lemma 2 in \cite{ryzn:02}). It is  clear that  $\int_{\R^d} (1 \wedge |x|) \nu_{\alpha}(dx) < \infty.$

Let us fix $\alpha \in (0,1)$. The heat-kernel $p_{\alpha,m}$ of such operator is given in formula (7) of \cite{ryzn:02} at page 4 (with the correspondence $2\beta = \alpha$):
\begin{gather} \label{re2}
p_{\alpha,m}(t,x) = e^{mt} \int_0^{\infty} g (u,x) e^{-m^{\frac{2}{\alpha}} u} \;\; \theta_{\alpha} (t,u) du,
\end{gather}
 where $g (u,x) = (4 \pi u)^{- d/2} e^{- |x|^2/ 4u}$ is the Gaussian kernel. Moreover  $\theta_{\alpha} (t,u)$, $u>0$, is the density function of the strictly $\alpha/2$-stable subordinator at time $t$ (see (4) in the indicated reference). It is well know  that $p_m(t, \cdot) \in C^{\infty}(\R^d)$, $t>0.$ In the limit case $m=0$, one gets the density of $\triangle^{\alpha/2}$:
\begin{gather} 
  p_{\alpha}(t,x)=p_{\alpha,0}(t,x) = \int_0^{\infty} g (u,x)  \;\; \theta_{\alpha} (t,u) du, \;\; x \not =0,
\end{gather}
 which is also considered in the proof of Lemma 5 in \cite{bogd:jaku:07}; note  
  that $\theta_{\alpha} (t,u) \le c_1 \, t u^{-1 - \frac{\alpha}{2}}$.
 We obtain  easily that, for $t \in (0,1]$, $x \not =0$,
\begin{equation}\label{rel1}
p_{\alpha,m}(t,x) \le e^{m} p_{\alpha,0}(t,x)
\end{equation}
 and so \A{NDa} holds.  {It will be shown in Section \ref{SEC_P_BETA_REL} that  $({\mathscr P}_\beta) $ holds even in this non-symmetric case.}

\subsubsection{H\"older spaces and smoothness assumptions  }

 In the following, for a fixed terminal time $T$ and a Borel function $\psi:[0,T]\times \R^d \rightarrow \R^\ell,\ \ell \in \{1, \ldots, d\}$ which is $\gamma $-H\"older continuous in the space variable, $\gamma\in (0,1) $, uniformly in $t\in [0,T] $ we denote by
 \begin{equation}
 \label{HOLDER_MODULUS_WITH_TIME}
 [\psi]_{\gamma,T}:=\inf\{K>0: \forall (t,x,x')\in [0,T]\times (\R^{d})^2,\ |\psi(t,x)-\psi(t,x')|\le K|x-x'|^\gamma  \}.
 \end{equation}
We write that $ \psi \in L^\infty\big([0,T],C^\gamma(\R^d,\R^\ell) \big) 
$ 
as soon as $[\psi]_{\gamma,T}<\infty $. 

 If additionally the function $\psi$ is bounded, we write that $ \psi \in L^\infty\big([0,T],C_b^\gamma(\R^d,\R^\ell) \big)$, where the subscript precisely emphasizes the boundedness, and introduce the corresponding \textit{norm}:
 \begin{equation}
 \label{LINF_B_HOLD}
 \|\psi\|_{L^\infty(C_b^\gamma)}:=\|\psi\|_\infty+ [\psi]_{L^\infty(C^\gamma)}=\|\psi\|_\infty+[\psi]_{\gamma,T}, \ \|\psi\|_\infty:=\sup_{(t,x)\in [0,T]\times \R^d}|\psi(t,x)|.
 \end{equation}
Again, $C_b^\gamma(\R^d,\R^\ell) $ is the \textit{usual} H\"older space of index $\gamma $.
 
 We also naturally define, accordingly the spaces $L^\infty\big([0,T],C_b^{1+\gamma}(\R^d,\R^\ell) \big) $ and $ C_b^{1+\gamma}(\R^d,\R^\ell)$ respectively endowed with the norms:
 \begin{eqnarray}
\label{HOLDER_NORM_GE1}
\|\psi\|_{L^\infty([0,T],C_b^{1+\gamma})}
&:= & \|\psi\|_\infty+\|D\psi\|_\infty+ [D\psi ]_{\gamma, T}, 
\;\;\; \psi \in L^\infty([0,T],C_b^{1+\gamma}); \notag\\
\|\varphi\|_{C_b^{1+\gamma}}&:=& \|\varphi\|_\infty+\|D\varphi\|_\infty+[D\varphi]_{\gamma} \notag\\
&:=&\sup_{x\in \R^d}|\varphi(x)|+\sup_{x\in \R^d}|D\varphi(x)|+\sup_{(x,x')\in (\R^d )^2,x\neq x'}\frac{|D\varphi(x)- D\varphi(x')|}{|x-x'|^\gamma},
\end{eqnarray} 
%
 $ \varphi \in C_b^{1+\gamma}(\R^d,\R^\ell).$
With these notations at hand, for a fixed stability index $\alpha \in (0,1) $, and given final horizon $T>0$ our assumptions concerning the smoothness of the coefficients in \eqref{Asso_PDE} are the following:
\begin{trivlist}
\item[\A{S}] There exists $\beta\in (0,1) $ s.t. $\alpha+\beta >1 $ and the source $f\in L^\infty\big([0,T],C_b^\beta(\R^d,\R)\big)$, $g\in C_b^{\alpha+\beta}(\R^d,\R)$. 
\item[\A{D}] The drift/transport term  $F$  verifies \eqref{22} with $\beta$ as in  \A{S}  for some  $K_0 >0$.
\end{trivlist}
 {\sl We will say that assumption \A{A} is in force as soon as the above conditions \A{S}, \A{D}, \A{NDa} and  \A{NDb} hold. In particular under \A{A} we have that 
 property $({\mathscr P}_\beta)$ holds. }


\subsection{Main Results}

The solutions of \eqref{Asso_PDE} will be sought in function spaces which are the natural extension in the current stable framework of those considered in the diffusive setting by Krylov and Priola \cite{kryl:prio:10}. Namely, we introduce
${\mathscr C \, }_{b}^{\alpha+\beta}([0,T]\times\R^d) $ the set of functions $\psi(t,x) $ defined on $ [0,T]\times \R^d$ such that:
\begin{enumerate}
\item[(i)] The function $\psi$ is continuous on $[0,T]\times \R^d $.
\item[(ii)] For any $t\in [0,T]$ the function $\psi(t,\cdot)\in C_b^{\alpha+\beta}(\R^d) $ and the norm $\|\psi(t,\cdot)\|_{C_b^{\alpha + \beta}} $ 
is bounded w.r.t $t\in [0,T] $, i.e., $\psi\in L^\infty\big([0,T],C_b^{\alpha+\beta}(\R^d)\big) $.
\item[(iii)] There exists a function $\varphi_\psi : [0,T]\times \R^d \rightarrow \R$ s.t. for any smooth and compactly supported function $\eta \in C_0^\infty([0,T]\times \R^d) $, the product $(\varphi_\psi \eta)(t,x) $ is bounded and $\beta +\alpha-1$-H\"older continuous in space uniformly in $t\in [0,T] $ and for any $x\in \R^d,\ 0\le t<s\le T $, it holds that:
$$\psi(s,x)-\psi(t,x)=\int_t^s \varphi_\psi(v,x)dv. $$
For $\psi \in {\mathscr C \, }_{b}^{\alpha+\beta}([0,T]\times\R^d) $, we write $\partial_t \psi=\varphi_\psi $ which is actually the generalized derivative w.r.t. the time variable of the function $\psi $.
\end{enumerate}
Accordingly, having a solution to \eqref{Asso_PDE} in ${\mathscr C}_b^{\alpha+\beta}([0,T]\times \R^d) $ is equivalent to say that for  $0\le t<s\le T,\ x\in  \R^d $,
\begin{eqnarray}\label{INTEGRAL_VERSION}
u(t,x)=u(s,x)+\int_t^s dr f(r,x)-\int_t^s dr\Big(L_\alpha +F(r,x)  \cdot D \Big) u(r,x).
\end{eqnarray}


\begin{THM}[Schauder Estimates]\label{THEO_SCHAU_ALPHA}
Let $\alpha\in (0,1) $ be fixed. Under \A{A}
there exists a constant $C:=C(\A{A},T, {K_{0}})$ s.t. for any solution $u\in {\mathscr C}_{b}^{\alpha+\beta}([0,T]\times\R^d)  $ of \eqref{Asso_PDE}, it holds that:
\begin{equation}
\label{EST_SCHAU}
\|u\|_{L^\infty([0,T],C_b^{\alpha+\beta})}\le C(\|g\|_{C_b^{\alpha+\beta}}+\|f\|_{L^\infty([0,T],C_b^{\beta})}).
\end{equation}
\end{THM}

In particular, the previous control provides uniqueness in the considered function space. Associated with an existence result developed in Section \ref{ESISTENZA}, we eventually derive the following theorem.

\begin{THM}[Existence and Uniqueness]\label{EXISTENCE_UNIQUENESS_ALPHA}
Let $\alpha\in (0,1) $ be fixed. Under the assumptions of the previous theorem
there exists a unique solution $u\in {\mathscr C }_{b}^{\alpha+\beta}([0,T]\times\R^d)  $ to \eqref{Asso_PDE} which also satisfies the estimate
\eqref{EST_SCHAU}.
\end{THM}
We finish the section with some  comments  on the case $\alpha\in [1,2).$ 

\begin{REM}\label{chissa} {\em  When $\alpha\in [1,2) $, a quite
 general class of generators of L\'evy processes 
such that elliptic Schauder estimates hold
is the one considered in Section 6 of  \cite{prio:18}.  To introduce such class we  define
 the L\'evy generator 
$$
L \phi (x) =
\int_{{\mathbb R}^d} \big(  \phi (x+y) -  \phi  (x)
   - 1_{ \{  |y| \le 1 \} } \,  D \phi (x) \cdot y   \big)
   \, \nu (dy),\;\; x \in {\mathbb R}^d, \;\;\; \phi \in C_0^{\infty} (\R^d).
$$
We assume that  the Blumenthal-Getoor exponent $\alpha_{0} = \alpha_{0}(\nu)$, 
 $
\alpha_{0} = \inf \big\{ \sigma >0 \;\; : \;\; { \int}_{\{ |x| \le 1  \, \}} |y|^{\sigma} \nu (dy) < \infty \big \}
$
  belongs to $(0,2)$. Moreover, we require that  the associated convolution semigroup $(P_t)$ verifies: $P_t (C_b({\mathbb R}^d)) \, \subset
C^1_b({\mathbb R}^d)$, ${t>0}$,
and, further, there exists 
$c_{\alpha_{0}}= c_{\alpha_0}(\nu) >0$ 
 such that
 \begin{align} \label{grad}
 \sup_{x \in {\mathbb R}^d}| D P_t f(x)| \le {c_{\alpha_{0}}}\,{t^{-\frac{1}{ \alpha_{0} } }} \cdot \sup_{x \in {\mathbb R}^d}| f(x)|,\;\; \;
t \in (0,1], \; f \in C_b ({\mathbb R}^d). 
\end{align}
Non-degenerate symmetric $\alpha$-stable operators, relativistic $\alpha$-stable operators, temperate $\alpha$-stable operators verify the previous two assumptions with $\alpha = \alpha_0$.

According to Theorem 6.7 in \cite{prio:18} we have, for $\alpha_0 \ge 1 $, $\beta \in (0,1)$ and $\beta + \alpha_0 >1$,
 \begin{equation} \label{sch4}
 \begin{array}{l}
 \lambda \| w_{\lambda}\|_{\infty} + [
Dw_{\lambda}]_{C^{{\alpha_0} + \beta - 1}_b (\R^d)}
 \;  \le \, C_0 \| \lambda w -  {\mathcal L} w  - b \cdot Dw \|_{C^{\beta}_b(\R^d)}, \;\;  \lambda \ge 1,
\end{array} 
\end{equation}
 assuming   $b \in C^{\beta}_b(\R^d, \R^d)$.
Such elliptic  estimates could be extended to the parabolic setting without difficulties.
}
\end{REM}

\mysection{Proof of the main results through a perturbative approach}\label{PERT}

\subsection{Frozen semi-group  and associated smoothing effects}
The key idea in our approach consists in considering a suitable \textit{proxy} IPDE, for which we have \textit{good} controls along which to expand a solution $u\in {\mathscr C}^{\alpha+\beta}([0,T]\times \R^d)$ to \eqref{Asso_PDE}. Under \A{A}, which involves potentially unbounded drifts, we will use for the proxy IPDE a non zero first order term which involves \textit{a} flow associated with the drift coefficient $F$ (which in the current setting exists from the Peano theorem). This flow is, for given \textit{freezing} parameters $(\tau,\xi)\in [0,T]\times \R^d$, defined as:

\begin{equation}\label{DEF_FLOW}
\theta_{s,\tau}(\xi)=\xi+\int_{\tau}^sF(v,\theta_{v,\tau}(\xi))dv, \;\; s \ge \tau,
\end{equation} 
 $\theta_{s, \tau}(\xi) = \xi$ for $s < \tau$.
 For $f$ and $g$ as in Theorem \ref{THEO_SCHAU_ALPHA} we then introduce our proxy IPDE:
\begin{eqnarray}
\label{FROZEN_Asso_PDE}
\p_t  \tilde u(t,x) +  L_\alpha \tilde u(t,x) + F(t,\theta_{t,\tau}(\xi)) \cdot D_x \tilde u(t,x) &=& -f(t,x),\quad \text{on }[0,T)\times \R^d, \notag\\
\tilde u(T,x) &=& g(x),\quad \text{on }\R^d.
\end{eqnarray}
Under \A{A}, it is clear that the time-dependent operator $L_{\alpha } +  F(t, \theta_{t,\tau}(\xi) ) \cdot D_x $ generates a family of transition probability (or two parameter transition semi-group)  $\big(\tilde P_{s,t,\alpha}^{(\tau,\xi)})_{0\le t\le s\le T}$. For fixed $0\le t< s\le T $, the associated heat-kernel writes: 
\begin{equation}\label{CORRESP_SHIFTED}
\tilde p_\alpha^{(\tau,\xi)}(t,s,x,y)= p_\alpha\Big(s-t,y-m_{s,t}^{(\tau,\xi)}(x)\Big),\ m_{s,t}^{(\tau,\xi)}(x):=x+\int_t^s F(v,\theta_{v,\tau}(\xi))dv,
\end{equation}
(we remark that, for fixed $(t,x)\in [0,T]\times \R^d $ and $ s\in [t,T]$,  the corresponding process with frozen drift is 
 $\tilde X_s^{\tau,\xi}=x+\int_t^s F(v,\theta_{v,\tau}(\xi))dv+Z_{s-t})$ with generator $L_\alpha $. Observe that from the above definition of the \textit{shift} $m_{s,t}^{(\tau,\xi)}(x) $ we have the important property 
 \begin{equation}
\label{CORRES_LINEARIZED_FLOW_AND_FLOW}
 m_{s,t}^{(\tau,\xi)}(x)|_{(\tau,\xi)=(t,x)}=\theta_{s,t}(x).
 \end{equation}

Let us now state the smoothing effect of the semi-group associated with \eqref{FROZEN_Asso_PDE}.
 \begin{lem}[Smoothing effects of the derivatives of the frozen semi-group]
\label{sd} Assume  \A{NDa} and  \A{NDb} 
 (so property $({\mathscr P}_\beta)$ holds).  There exists $C\ge 1$ s.t. for any $\varphi\in C^{\beta}(\R^d,\R)$, any freezing couple $(\tau,\xi) $, $\ell\in \{1,2\} $ and all $0\le t\le s\le T,\ x\in \R^d$ :
\begin{equation}\label{REG_BETA}
|D_x^\ell\tilde P_{s,t,\alpha}^{(\tau,\xi)}\varphi(x)|\le \frac{C [\varphi]_{\beta}}{(s-t)^{\frac{\ell}{\alpha}-\frac\beta \alpha}}.
\end{equation}
\end{lem}
\begin{proof}
We recall  that, with the notation of \eqref{CORRESP_SHIFTED}, $m^{(\tau,\xi)}_{s,t}  (x):=x+\int_t^s F(v,\theta_{v,\tau}(\xi)) dv $.
Again, if $ (\tau,\xi)=(t,x)$ then $m^{(\tau,\xi)}_{s,t}  (x)=\theta_{s,t}(x)$ but this will not be the only case to be considered. Let us prove \eqref{REG_BETA}. We use a cancellation argument (recall that $\int_{\R^d}dy\tilde p_\alpha^{(\tau,\xi)}(t,s,x,y) =1$) and property $({\mathscr P}_\beta) $ to write:
 \begin{eqnarray*}
 |D_x\tilde P_{s,t,\alpha}^{(\tau,\xi)}\varphi(x)|=\Big|\int_{\R^d} dy D \tilde p_\alpha^{(\tau,\xi)}(t,s,x,y)\big[\varphi(y)-\varphi(m_{s,t}^{(\tau,\xi)}(x))\big]\Big|
\\  
 = 
 \Big |\int_{\R^d} dy D_x  p_\alpha (s-t ,y - m_{s,t}^{(\tau,\xi)}(x))\big[\varphi(y)-\varphi(m_{s,t}^{(\tau,\xi)}(x))\big]
 \Big|
 \\
 \le [\varphi]_{\beta} 
 \int_{\R^d} dy |D_x  p_\alpha (s-t ,y - m_{s,t}^{(\tau,\xi)}(x))| \,  |y-m_{s,t}^{(\tau,\xi)}(x)|^\beta
\\  
\le  [\varphi]_{\beta} 
 \int_{\R^d} d\tilde y | D_x  p_\alpha (s-t ,\tilde y )| \,  |\tilde y|^\beta\underset{(\mathscr P_{\beta})}{\le} \frac{C[\varphi]_\beta}{(s-t)^{\frac 1\alpha-\frac \beta\alpha}}. 
 \end{eqnarray*}
This proves the result for $\ell=1 $. The case $\ell=2 $ is dealt similarly. 
\end{proof}
\textcolor{black}{We point out that, if we restrict to bounded functions $ \varphi \in C_b^\beta(\R^d,\R)$, the above result can also be derived directly from standard interpolation arguments provided \eqref{EQ_ND} holds and without the integrability constraints of $({\mathscr P}_\beta)$ (see e.g. the proof of Theorem 3.3 in \cite{prio:12})}.

{\color{black}
We  now define our candidate to be the solution of the proxy IPDE. This candidate appears to be (if it exists and  is smooth enough) the representation of the solution of \eqref{FROZEN_Asso_PDE} obtained through \textcolor{black}{the} Duhamel principle. \textcolor{black}{Hence, we call it} the Duhamel representation associated with \eqref{FROZEN_Asso_PDE}. As it will appear, such a representation is robust enough to satisf\textcolor{black}{y} Schauder estimates \textcolor{black}{under the assumptions} of Theorem \ref{THEO_SCHAU_ALPHA}. \textcolor{black}{It then indeed provides} a solution to \eqref{FROZEN_Asso_PDE}.
\begin{defi}[Duhamel representation associated with \eqref{FROZEN_Asso_PDE}]
For any freezing couple $(\tau,\xi)\in [0,T]\times \R^d $, we call Duhamel representation associated with \eqref{FROZEN_Asso_PDE}, and we denote it by $\tilde u^{(\tau,\xi)}$ the map
\begin{eqnarray}
\label{FROZEN_SEMI_GROUP}
\tilde u^{(\tau,\xi)} : [0,T]\times \R^d \ni (t,x)\mapsto \tilde P_{T,t,\alpha}^{(\tau,\xi)}g(x)
+\int_t^T ds  \Big( \tilde P_{s,t,\alpha}^{(\tau,\xi)} f(s,\cdot)\Big)(x),\;\;\;  (t,x)\in [0,T]\times \R^d.
\end{eqnarray}
\end{defi}
}

\begin{PROP}[Schauder estimate for the frozen semi-group]\label{SCHAUDER_FROZEN}
Under \A{A}, for any freezing couple $(\tau,\xi)\in [0,T]\times \R^d $, the following Schauder estimate holds for the map $\tilde u^{(\tau,\xi)} $ defined by \eqref{FROZEN_SEMI_GROUP}. There exists a constant $C:=C(\A{A})$ s.t. 
\begin{equation}\label{SCHAUDER_FROZEN_BETA}
\|\tilde u^{(\tau,\xi)}\|_{L^\infty([0,T],C_b^{\alpha+\beta})}\le C\big(\|g\|_{C_b^{\alpha+\beta}}+\|f\|_{L^\infty([0,T],C_b^{\beta})} \big).
\end{equation}
\end{PROP}

\begin{proof} We focus on the proof of \eqref{SCHAUDER_FROZEN_BETA} and proceed in three steps: control of the supremum norm of the function, of its gradient, of the H\"older modulus of the gradient.

\begin{trivlist}
\item[$(i) $] \emph{Supremum norm of the solution.}
 
The control for the supremum norm readily follows from \eqref{FROZEN_SEMI_GROUP}. Indeed, for a fixed frozen couple $(\tau,\xi)\in [0,T]\times\R^d $ and any $(t,x)\in [0,T]\times \R^{d} $:
\begin{eqnarray}
|\tilde u^{(\tau,\xi)}(t,x)|=\Big|\tilde P_{T,t,\alpha}^{(\tau,\xi)}g(x)
+\int_t^T ds  \Big( \tilde P_{s,t,\alpha}^{(\tau,\xi)} f(s,\cdot)\Big)(x)\Big|\le \|g\|_\infty+T \|f\|_\infty.\label{CTR_BD_SUP_FROZEN}
\end{eqnarray}

\item[$(ii) $] \emph{Supremum norm of the gradient.}

Let us now give a global bound for the gradient. Differentiating \eqref{FROZEN_SEMI_GROUP} yields:
\begin{eqnarray}
|D_x \tilde u^{(\tau,\xi)}(t,x)|&=&\Big|D_x \tilde P_{T,t,\alpha}^{(\tau,\xi)}g(x)
+\int_t^T ds  D_x \Big( \tilde P_{s,t,\alpha}^{(\tau,\xi)} f(s,\cdot)\Big)(x)\Big|\notag \\
&\le & \Big|D_x \tilde P_{T,t,\alpha}^{(\tau,\xi)}g(x) \Big|+ \Big|\int_t^T ds  D_x\Big(  \tilde P_{s,t,\alpha}^{(\tau,\xi)} f(s,\cdot)\Big)(x)\Big|.\label{PRELIM_CTR_GRADIENT_SUP_FROZEN}
\end{eqnarray}

From  \eqref{CORRESP_SHIFTED} write for $g\in C^{\alpha+\beta}_b(\R^d)$ recalling that $\alpha+\beta>1 $:
\begin{eqnarray}
D_x\tilde P_{T,t,\alpha}^{(\tau,\xi)}g(x)&=&D_x \int_{\R^d} dy \tilde p_{\alpha}^{(\tau,\xi)}(t,T,x,y) g(y)= D_x\int_{\R^d} dy p_{\alpha}(T-t,y-m_{T,t}^{(\tau,\xi)}(x)) g(y)
\notag
\\
&=&D_x 
\int_{\R^d} dy  p_{\alpha}\Big(T-t,y-\big(x+\int_t^T F(v,\theta_{v,\tau}(\xi) ) dv\big)\Big) g(y)
\notag\\
&=& D_x \int_{\R^d} dz 
p_{\alpha}\big(T-t,z \big) g\Big(z+ \big(x+\int_t^T F(v,\theta_{v,\tau}(\xi) ) dv\big)\Big),\notag\\
|D_x\tilde P_{T,t,\alpha}^{(\tau,\xi)}g(x)|&\le & 
\|{D}g\|_\infty.\label{CTR_DG}
\end{eqnarray}
Let us now turn to control of the source in \eqref{PRELIM_CTR_GRADIENT_SUP_FROZEN}.  From Lemma \ref{sd} we derive
$$|D_x  (\tilde P_{s,t,\alpha}^{(\tau,\xi)} f(s,\cdot))(x)|\le \frac{C[f]_{\beta, T}}{(s-t)^{\frac 1\alpha-\frac\beta \alpha}}.$$
Since $\alpha+\beta>1 $ we thus get that $ 1/\alpha-\beta /\alpha<1 $ so that the above singularity is integrable in time. We therefore derive:
\begin{eqnarray}
\Big|\int_t^T ds  D_x\Big(  \tilde P_{s,t,\alpha}^{(\tau,\xi)} f(s,\cdot)\Big)(x)\Big|\le C[f]_{\beta,T}\int_t^T\frac{ds}{(s-t)^{\frac{1}{\alpha}-\frac{\beta}{\alpha}}}\le C[f]_{\beta,T}(T-t)^{\frac{\alpha+\beta-1}{\alpha}}. \label{CTR_PERT_INT_FROZEN}
\end{eqnarray}
Plugging \eqref{CTR_PERT_INT_FROZEN} and \eqref{CTR_DG} into \eqref{PRELIM_CTR_GRADIENT_SUP_FROZEN} gives the following bound for the gradient:
\begin{equation}
\label{CTR_GRAD_SUP_FROZEN}
|D_x \tilde u^{(\tau,\xi)}(t,x)|\le \|Dg\|_\infty+ C (T-t)^{\frac{\alpha+\beta-1}{\alpha}}[f]_{\beta,T}.
\end{equation}

\item[$(iii) $] \emph{H\"older modulus of the solution.}

It now remains to control the $\alpha+\beta-1 $ H\"older modulus of the gradient of the solution. To this end we introduce for a given time $t\in [0,T] $ and given spatial points $x,x'\in \R^d$, the notion of \textit{diagonal} and \textit{off-diagonal} regime.
  
For the frozen semi-group, we say that the off-diagonal regime (resp. diagonal) holds when $T-t\leq c_0 |x-x'|^\alpha$ (resp. $T-t\ge c_0 |x-x'|^\alpha$). We use here a constant $c_0$, which will be for our further perturbative analysis  meant to be small for circular type arguments to work. Anyhow, for the frozen semi-group this constant  could be arbitrary, for instance 1.

We first investigate the H\"older continuity in space of $x\mapsto D_{x} \tilde P_{T,t,\alpha}^{(\tau,\xi)}g(x)$.\\

$(iii)-(a)$ {\it  Off-diagonal regime.} Let $T-t\leq c_0 |x-x'|^\alpha$. Observe from \eqref{CORRESP_SHIFTED} that $ D_x \tilde p_\alpha^{(\tau,\xi)}(t,T,x,y)=-D_y\tilde p_\alpha^{(\tau,\xi)}(t,T,x,y)$. We therefore write:
\begin{eqnarray}
&&D_{x} \tilde P_{T,t,\alpha}^{(\tau,\xi)}g(x)- D_{x} \tilde P_{T,t,\alpha}^{(\tau,\xi)}g(x')
\nonumber \\
&=&
\bigg [ \int_{\R^{d}}  \tilde p_\alpha^{(\tau,\xi)}(t,T,x,y) D
 g(y) dy- \int_{\R^{d}} \tilde p_\alpha^{(\tau,\xi)}(t,T,x',y) D
 g(y) dy \bigg ]
\nonumber \\
&=&\bigg [ \int_{\R^{d}}  \tilde p_\alpha^{(\tau,\xi)}(t,T,x,y) [D
g(y)-
D
g(m_{t,T}^{(\tau,\xi)}(x))] dy\notag\\
&&+[D
g(m_{t,T}^{(\tau,\xi)}(x))-D
g(m_{t,T}^{(\tau,\xi)}(x'))]- \int_{\R^{d}} \tilde p_\alpha^{\xi'}(t,T,x',y) [D
g(y)-D
g(m_{T,t}^{(\tau,\xi')}(x))]dy\bigg],\notag
\end{eqnarray}
recalling that $\tilde p_\alpha^{(\tau,\xi)}(t,T,x,\cdot),\ \tilde p_\alpha^{(\tau,\xi')}(t,T,x',\cdot) $ are probability densities for the last equality.  Since $Dg$ is $\alpha+\beta-1 $-H\"older continuous and $\alpha+\beta-1<\alpha $, we therefore get:
\begin{eqnarray*}
&&|D_{x} \tilde P_{T,t,\alpha}^{(\tau,\xi)}g(x)- D_{x} \tilde P_{T,t,\alpha}^{(\tau,\xi)}g(x')|\\
& \le& [Dg]_{\alpha+\beta-1}\Big(\int_{\R^d} dy p_\alpha(T-t,y-m_{T,t}^{(\tau,\xi)}(x))|y-m_{T,t}^{(\tau,\xi)}(x)|^{\alpha+\beta-1}\\
&& +|m_{t,T}^{(\tau,\xi)}(x)-m_{t,T}^{(\tau,\xi)}(x')|^{\alpha+\beta-1}+\int_{\R^d} dy p_\alpha(T-t,y-m_{T,t}^{(\tau,\xi)}(x'))|y-m_{T,t}^{(\tau,\xi)}(x')|^{\alpha+\beta-1}\Big)\\
 &\le & [Dg]_{\alpha+\beta-1}\Big( C(T-t)^{\frac{\alpha+\beta-1}\alpha}+|x-x'|^{\alpha+\beta-1}\Big),
\end{eqnarray*} 
using property   \A{NDa}
and recalling that the mapping $x\mapsto m_{T,t}^{(\tau,\xi)}(x) $ is affine for the last inequality. On the considered off-diagonal regime, the previous bound eventually yields:
\begin{eqnarray}
|D_{x} \tilde P_{T,t,\alpha}^{(\tau,\xi)}g(x)- D_{x} \tilde P_{T,t,\alpha}^{(\tau,\xi)}g(x')|\le [Dg]_{\alpha+\beta-1}\big(C c_0^{\frac{\alpha+\beta-1}\alpha}+1 \big) |x-x'|^{\alpha+\beta-1},\label{BD_SG_G_OFF_DIAG}
\end{eqnarray}
which is the expected bound.\\

$(iii)-(b)$ {\it  Diagonal regime.} If $T-t > c_0 |x-x'|^\alpha$,
we directly write for all $(\tau,\xi) \in [0,T]\times \R^d $:
\begin{eqnarray}
\label{control_Holder_D2_P_g}
&&|D_{x} \tilde P_{T,t,\alpha}^{(\tau,\xi)}g(x)- D_{x} \tilde P_{T,t,\alpha}^{(\tau,\xi)}g(x')| 
\nonumber 
\\
&\le &\big |\int_{\R^{d}} [\tilde p_\alpha^{(\tau,\xi)}(t,T,x,y) -\tilde p_\alpha^{(\tau,\xi)}(t,T,x',y) ] D
g(y) dy  \big |
\nonumber \\
 &\leq& 
\big |\int_0^1 d\mu \int_{\R^{d}}  [D_{x} \tilde p_\alpha^{(\tau,\xi)}(t,T,x'+\mu(x-x'),y) \cdot (x-x') ]D
g(y) dy  \big |
\nonumber 
\\
 &=& 
\big |\int_0^1 d\mu \int_{\R^{d}}  [D_{x}  p_\alpha (T-t, y - m^{\xi}_{T,t} (x'+\mu(x-x'))) \cdot (x-x') ]D
g(y) dy  \big |.
\nonumber 
\end{eqnarray}
This contribution is again dealt through usual cancellation techniques recalling that since $\int_{\R^d} \tilde p_\alpha^{\xi}(t,T,x'+\mu(x-x'),y)dy=1$ then for $\ell\in \{1,2\},\ D_x^\ell \int_{\R^d} \tilde p_\alpha^{(\tau,\xi)}(t,T,x'+\mu(x-x'),y)dy=0 $. We get from the above estimate that denoting as well as in \eqref{CORRESP_SHIFTED} by $ m^{(\tau,\xi)}_{T,t} \big (x'+\mu(x-x')\big ):=x'+\mu(x-x')+\int_t^T F(v,\theta_{v,\tau}(\xi)) dv$, we obtain:
\begin{eqnarray} \label{control_Holder_D2_P_3g}
&&|D_{x} \tilde P_{T,t,\alpha}^{(\tau,\xi)}g(x)- D_{x} \tilde P_{T,t,\alpha}^{(\tau,\xi)}g(x')| 
 \\
 &\leq& 
\bigg |\int_0^1 d\mu\int_{\R^{d}} 
 [D_{x}  p_\alpha (T-t, y - m^{(\tau,\xi)}_{T,t} (x') -\mu(x-x'))) \cdot (x-x') ]
 \Big [D g \big (y \big )-D g\big (m^{(\tau,\xi)}_{T,t}  (x') +\mu(x-x') \big )\Big]dy.
 \nonumber
 \\
  &\leq&  [ Dg]_{{\alpha+\beta-1}}
\int_0^1 d\mu\int_{\R^{d}} 
 |D_{x}  p_\alpha (T-t, y - m^{(\tau,\xi)}_{T,t} (x') -\mu(x-x'))) \cdot (x-x') | \,
 | y - m^{(\tau,\xi)}_{T,t}  (x') - \mu(x-x') |^{\alpha + \beta -1}dy,
 \nonumber
\end{eqnarray}
since $g \in C^{\alpha+\beta}_b(\R^{d})$. 
From Lemma \ref{sd}, 
we obtain
 \begin{eqnarray}
&&|D_{x} \tilde P_{T,t,\alpha}^{(\tau,\xi)}g(x)- D_{x} \tilde P_{T,t,\alpha}^{(\tau,\xi)}g(x')| 
\nonumber \\
 &\le& C [Dg]_{C_b^{\alpha+\beta-1}}   (T-t)^{-\frac 1\alpha+\frac{\alpha+\beta-1}\alpha} \,
  |x-x'|^{}
 \le C [Dg]_{\alpha+\beta-1} |x-x'|^{\alpha+\beta-1}\label{control_Holder_D_P_g1}, 
\end{eqnarray}
where the above constant $C$ also depends on $c_0$: since $c_0|x-x'|^\alpha< (T-t) $ and $\alpha+\beta<2 $, we indeed have $(T-t)^{-\frac 1 \alpha+\frac{\alpha+\beta-1}{\alpha}}|x-x'|\le (c_0 |x-x'|^\alpha)^{-\frac 1\alpha+\frac {\alpha+\beta-1}\alpha} |x-x'|\le C|x-x'|^{\alpha+\beta-1} $. Then, from equations \eqref{control_Holder_D_P_g1} and \eqref{BD_SG_G_OFF_DIAG} we get that, for all $x,x' \in \R^d$:
\begin{equation}
\label{CTR_BD_HOLDER_SG_G}
|D_{x} \tilde P_{T,t,\alpha}^{(\tau,\xi)}g(x)- D_{x} \tilde P_{T,t,\alpha}^{(\tau,\xi)}g(x')|\le C[Dg]_{\alpha+\beta-1}|x-x'|^{\alpha+\beta-1},
\end{equation}
which gives the expected control.

Let us now turn to the mapping $x\mapsto \int_t^T ds D_x\Big(\tilde P_{s,t,\alpha}^{(\tau,\xi)} f(s,\cdot)\Big) (x) $. For separate points, i.e.,  $|x-x'|>0$, while integrating in $s\in [t,T] $ we have that, accordingly with the previous terminology if $s\in [t, (t+ c_ 0|x-x'|^\alpha)\wedge T] $ then the \textit{off-diagonal} regime holds for the analysis of 
\begin{equation}\label{DEF_GRAD_SOURCE}
D_x\Big(\tilde P_{s,t,\alpha}^{(\tau,\xi)} f(s,\cdot)\Big) (x)-D_x\Big(\tilde P_{s,t,\alpha}^{(\tau,\xi)} f(s,\cdot)\Big) (x')
\end{equation}
 and similarly, if $s\in [ (t+ c_ 0|x-x'|^\alpha)\wedge T,T] $ then the \textit{diagonal} regime holds.

We now introduce the transition time $t_0$ defined as follows:
\begin{equation}\label{def_t0}
t_0=(t+c_0 |x-x'|^\alpha) \wedge T,
\end{equation}
i.e., $t_0$ precisely corresponds to the critical time at which a change of regime occurs. Observe that, if $t_0=T $ the contribution \eqref{DEF_GRAD_SOURCE} is in the \textit{off-diagonal} regime along the whole time interval.

Introduce the following family of Green kernels.
\begin{eqnarray}\label{GREEN_FLESSIBILE_IN_TEMPO}
\forall 0\le v<r\le T,\;\tilde G_{r,v,\alpha}^{(\tau,\xi)} f(t,x)&:=&\int_v^{r} ds\int_{\R^{d}} dy \tilde p_\alpha^{(\tau,\xi)}(t,s,x,y)f(s,y). 
\end{eqnarray}
The off-diagonal contribution associated with the difference \eqref{DEF_GRAD_SOURCE} now writes:
\begin{eqnarray}
\big | D_{x} \tilde G_{t_0,t,\alpha}^{(\tau,\xi)} f(t,x) \!-\! D_{x} \tilde G_{t_0,t,\alpha}^{(\tau,\xi)} f(t,x') \big |  
& \le&
\Big |\int_t^{t_0} \!\!ds D_{x} \Big(\tilde P_{s,t,\alpha}^{(\tau,\xi)}f(s,\cdot)\Big)(x)\Big | 
\!\!\!+\!\Big |\int_t^{t_0} \!\!ds D_{x} \Big(\tilde P_{s,t,\alpha}^{(\tau,\xi)}f(s,\cdot)\Big)(x')\Big | 
\nonumber \\ \label{CTR_D1}
&\!\!\!\leq& C [f]_{T, \beta}
\int_t^{t_0} ds (s-t)^{-\frac{1}{\alpha}+ \frac{\beta}{\alpha}}
\le  C[f]_{T, \beta}|x-x'|^{\alpha+\beta-1},
\end{eqnarray}
using equation \eqref{REG_BETA} of Lemma \ref{sd} for the last but one inequality.

For the \textit{diagonal regime}, which appears if $t_0<T $ and corresponds in that case to the difference  $$D_{x}\tilde G_{T,t_0,\alpha}^{(\tau,\xi)} f(t,x)-D_{x}\tilde G_{T,t_0,\alpha}^{(\tau, \xi)} f(t,x'),$$ we have to be more subtle and perform a Taylor expansion of $D_{x} (\tilde P_{s,t, \alpha}^{\xi}f(s,\cdot)) $ similarly to what we did in \eqref{control_Holder_D2_P_3g}. Namely: 
\begin{eqnarray*}
 \big | D_{x}\tilde G_{T,t_0,\alpha}^{(\tau,\xi)} f(t,x) \!-\! D_{x} \tilde G_{T,t_0,\alpha}^{(\tau,\xi)} f(t,x') \big | 
&\le& \int_{t_0}^{T } ds \Big |\int_0^1 d\mu   D_{x}^2\Big( \tilde P_{s,t,\alpha}^{(\tau,\xi)}f(s,\cdot)\Big)(x'+\mu(x-x')) \cdot (x-x') \Big | \\
&\le& |x-x'|\int_{t_0}^{T } ds \int_0^1 d\mu \Big|D_{x}^2 \Big(\tilde P_{s,t,\alpha}^{(\tau,\xi)} f(s,\cdot)\Big)\big(x'+\mu (x-x')\big)\Big | \\
&\le& C [f]_{\beta, T}|x-x'|\int_{t_0}^{T } ds (s-t)^{-\frac 2\alpha+\frac \beta \alpha},
\end{eqnarray*}
using again Lemma \ref{sd}, equation \eqref{REG_BETA} for the last inequality. 
This finally yields, recalling that $t_0=\big(t+ c_ 0|x-x'|^\alpha\big)\wedge T$:
\begin{eqnarray}
\label{CTR_D2}
\big | D_{x} \tilde G_{T,t_0,\alpha}^{(\tau,\xi)} f(t,x) \!-\! D_{x} \tilde G_{T,t_0,\alpha}^{(\tau,\xi)} f(t,x') \big | 
&\le& C  [f]_{\beta, T}\, |x-x| (|x-x'|^\alpha)^{1-\frac2\alpha +\frac \beta \alpha}\notag\\
&\le& 
 C  [f]_{\beta, T} \, |x-x'|^{\alpha+\beta-1}. 
\end{eqnarray}
Gathering \eqref{CTR_D1}  and  \eqref{CTR_D2} gives the stated estimate for the H\"older modulus of the gradient of the Green kernel. We eventually derive:
\begin{equation}
|D_{x}\tilde u^{(\tau,\xi)}(t,x) \!-\! D_{x}\tilde u^{(\tau,\xi)}(t,x')|\le C([Dg]_{\alpha+\beta-1}+[f]_{\beta, T}) |x-x'|^{\alpha+\beta-1}). \label{FINAL_BOUND_GR_HOLD_FROZEN}
\end{equation}

The estimate \eqref{SCHAUDER_FROZEN_BETA} of the proposition follows from \eqref{CTR_BD_SUP_FROZEN}, \eqref{CTR_GRAD_SUP_FROZEN} and \eqref{FINAL_BOUND_GR_HOLD_FROZEN}. 
\def\ciao3{
Now, to establish \eqref{SCHAUDER_FROZEN_GAMMA}, if $\alpha+\gamma>1 $ the proof follows exactly the same lines. When $\alpha+\gamma<1 $ one only has to control the H\"older modulus of the solution itself (instead of the gradient above) and the previous arguments can be readily adapted, i.e. there is no need to differentiate the mild formulation \eqref{FROZEN_SEMI_GROUP}. There are no cancellation arguments anymore but since $\tilde p_\alpha^{(\tau,\xi)}(t,s,x,y) $ is a density we can exploit the smoothness of $f,g$ proceeding e.g. like for the above control of the gradient of $\tilde P_{T,t,\alpha}^{(\tau,\xi)} g(x) $.
}
\end{trivlist}
\end{proof}

{\color{black}
\begin{PROP} \label{PROP_FOURIER}
Let $\tilde u$ be the map defined in \eqref{FROZEN_SEMI_GROUP}. Then,
\begin{trivlist}
\item[(i)] if property $({\mathscr P}_\beta) $ holds, for any freezing couple $(\tau,\xi)\in [0,T]\times \R^d $ the function $\tilde u^{(\tau,\xi)}$ defined in \eqref{FROZEN_SEMI_GROUP} belongs to ${\mathscr C}^{\alpha+\beta}([0,T]\times \R^d) $ and solves \eqref{FROZEN_Asso_PDE};
\item[(ii)] ``conversely'', if for any freezing couple $(\tau,\xi)\in [0,T]\times \R^d $, $\tilde v^{(\tau,\xi)}$ is a solution to \eqref{FROZEN_Asso_PDE} in ${\mathscr C}^{\alpha+\beta}([0,T]\times \R^d) $ with bounded support, i.e., there exists a compact set $K\subset \R^d$ such that
\begin{equation}\label{supp}
\text{Supp}(\textcolor{black}{\tilde v^{(\tau,\xi)}}(t, \cdot )) \subset K,\;\;\; t \in [0,T].
\end{equation}
then $\tilde v^{(\tau,\xi)}$ = $\tilde u^{(\tau,\xi)}$ defined by \eqref{FROZEN_SEMI_GROUP}.
\end{trivlist}
\end{PROP}
}

\begin{proof} {\color{black} Assertion $(i)$ is \textcolor{black}{rather direct} in view of Proposition \ref{SCHAUDER_FROZEN}: first, it follows from this result that for all $(\tau, \xi)$ the map $\tilde u^{(\tau,\xi)}$ belongs to $L^{\infty}([0,T],C^{\alpha+\beta}_b)$. Secondly, we have from the proof of the aforementioned proposition that for all $x$ in $\R^d$, \textcolor{black}{the mapping $ s\in (t,T]  \mapsto (D_x + L_\alpha)[\tilde P_{s,t,\alpha}^{(\tau,\xi)}(f(s,\cdot))](x)$} is controlled by an integrable quantity on $(t,T]$. From the very definition of $\tilde p_\alpha^{(\tau,\xi)}$ (fundamental solution of \eqref{FROZEN_Asso_PDE}) one can hence invert the (time) differentiation and integral operator in second term in the r.h.s. of \eqref{FROZEN_SEMI_GROUP}. \textcolor{black}{Similar arguments apply for the first term $\tilde P_{T,t,\alpha}^{(\tau,\xi)}g(\cdot) $ in \eqref{FROZEN_SEMI_GROUP}}. This proves, \textcolor{black}{on} the one hand, that for all $(\tau, \xi)$ the map $\tilde u^{(\tau,\xi)}$ indeed belongs to ${\mathscr C}^{\alpha+\beta}([0,T]\times \R^d)$. On the other hand, using again the fact that the kernel $\tilde p_\alpha^{(\tau,\xi)}$ in $P_{\cdot,\cdot,\alpha}^{(\tau,\xi)}$ is a fundamental solution of \eqref{FROZEN_Asso_PDE} we obtain (inverting again differentiation and integration operator) that $\tilde u^{(\tau,\xi)}$ solves \eqref{FROZEN_Asso_PDE}. \textcolor{black}{The previous arguments are detailed in the diffusive setting in Lemma 3.3 in \cite{kryl:prio:10}}. This concludes the proof of the first point. 
\vskip 1mm

Concerning $(ii)$, }we first fix $\tau$ and $\xi$ and set $b(t) = F(t,\theta_{t,\tau}(\xi)) $, $t \in [0,T]$.  We define 
$$
h(t,x) = \tilde u \big (t, x - \int_t^T b(v)dv \big).
$$
Note that,  a.e. in $t$,  $\partial_t h (t,x) +  L_\alpha h(t,x) =  -f(t, x - \int_t^T b(v)dv ), $  $h(T,x) =g(x),$ $ x \in \R^d.$ Set $l(t,x) = f(t, x - \int_t^T b(v)dv ).$

We can  apply the (partial) Fourier transform in the $x$-variable to
$\p_t  h(t,x) +  L_\alpha  h(t,x)$
 and obtain, a.e. in $t$,
 $$
 \p_t  v(t,\lambda ) + {\cal F}( L_{\alpha} h(t,\cdot ))(\lambda ),
 $$
 where $v(t, \lambda ) = {\cal F} h(t, \cdot)(\lambda )$. Note that, 
  for each $t \in [0,T]$, $L_\alpha h(t,\cdot )=  \int_{|y| \le 1} [h(t, \cdot + \, y)- h(t, \cdot )] \nu_{\alpha} (dy)$ $+  \int_{|y| >1} [h(t, \cdot + \, y)- h(t, \cdot )] \nu_{\alpha} (dy)$ belongs to $L^p(\R^d)$, for any $p \ge 1.$

 Now we use  the symbol $ \Psi(\lambda) =\Psi_{\alpha}(\lambda) $  of $L_{\alpha}$ given in \eqref{LEVY_KHINTCHINE}. We find  (cf. Section 3.3.2 in \cite{appl:09})
\begin {equation}
   v(s,\lambda ) -  v(t,\lambda )  + \Psi (\lambda) \int_t^s v(r,\lambda )dr    = - \int_t^s\hat l(r, \lambda), \;\;\;
 \lambda \in \R^d, \; \, t \le s \le T,
 \end{equation}
where ${\cal F}l(t, \cdot)(\lambda)= \hat l(t, \lambda)$ with the condition 
 $v(T, \lambda) = {\cal F}(g)(\lambda) $ $= \hat g (\lambda)$. The solution is given by 
$$
v (t, \lambda) = e^{(T-t)  \Psi(\lambda)} \, \hat g(\lambda) + 
 \int_t^T e^{(r-t) \Psi(\lambda)}   \hat l (r, \lambda) dr.
$$
 Using the stable convolution semigroup $P_t= P_t^{\alpha}$ associated \textcolor{black}{with} $L_{\alpha}$
and  the anti-Fourier transform we  get 
\begin{gather*}
h(t,x) =P_{T-t} g(x) + \int_t^T P_{r-t} \, l (r, \cdot )(x) dr.
\end{gather*}
It follows that  
$
\tilde u \big (t, x - \int_t^T b(v)dv \big)$ $ = P_{T-t} g(x) + \int_t^T P_{r-t} \, f (r, \cdot )(x - \int_r^T b(v)dv) dr.
$ Since  
$\int_r^T b(v)dv = \int_t^T b(v)dv - \int_t^r b(v)dv $, we arrive at 
\begin{gather*}
\tilde u \big (t, y\big) = P_{T-t} g(y +  \int_t^T b(v)dv) + \int_t^T P_{r-t} \, f (r, \cdot )(y +  \int_t^r b(v)dv ) dr.
\end{gather*}
which gives \eqref{FROZEN_SEMI_GROUP}.
 \end{proof}

\subsection{Duhamel type formulas}
The central point is that we will use the auxiliary \textit{proxy} IPDE \eqref{FROZEN_Asso_PDE} in order to derive appropriate quantitative controls on a solution $ u \in {\mathscr C}^{\alpha+\beta}([0,T]\times \R^d)$ of \eqref{Asso_PDE}. The parameters $(\tau,\xi) $ are set free and will be chosen in function of the control we aim to establish. 

Importantly, we can exploit equation \eqref{FROZEN_Asso_PDE} in the following proposition which gives a Duhamel type representation of the solution of \eqref{Asso_PDE} involving precisely the \textit{proxy} IPDE \eqref{FROZEN_Asso_PDE}. 
\begin{PROP}[A first Duhamel type representation]\label{DUHAMEL}
Let \A{A} 
hold. For a smooth non-negative spatial test function $\rho$ which is equal to  1 on the ball $B(0,1/2)$ and vanishes outside $B(0,1)$, we introduce using the proxy parameters $(\tau,\xi) $ the following cut-off function
\begin{gather*} 
  \eta_{\tau, \xi}(s,y) = \rho ( y - \theta_{s, \tau}(\xi)),\;\; y \in \R^d,\;\; s \in [0,T],
\end{gather*}
which precisely localizes around the frozen flow.

For $u\in {\mathscr C}^{\alpha+\beta}([0,T]\times \R^d) $ solving \eqref{Asso_PDE} the function $v_{\tau,\xi}:=u \eta_{\tau, \xi}$ solves the equation 
\begin{eqnarray} 
\p_t  v_{\tau,\xi}(t,x)  +  {L}_{\alpha} v_{\tau, \xi} (t,x)+ F(t, \theta_{t,\tau}(\xi) ) \cdot D_x v_{\tau, \xi}(t,x)
 &=& -  \Big[(\eta_{\tau, \xi} f)(t,x) + {\mathcal R}_{\tau,\xi}(t,x)\Big], \ (t,x)\in [0,T)\times \R^d,
\, 
\notag
\\
v_{\tau,\xi}(T,x) &=& g(x)  \eta_{\tau, \xi}(T,x),
\quad \text{on }\R^d, \label{we} 
\end{eqnarray}
  where, $\nu = \nu_{\alpha}$,
  \begin{eqnarray}
  \label{DEF_R}
{\mathcal R}_{\tau,\xi}(t,x)&=&  \Big[ \big( [F(t,x)-F(t, \theta_{t,\tau}(\xi) )]  \cdot D_x u(t,x) \big) \eta_{\tau, \xi} (t,x)\Big]\notag \\
&&- \Big[u(t,x)L_\alpha \eta_{\tau,\xi}(t,x)+\int_{\R^d} \big(u(t,x+y)-u(t,x)\big)\big(\eta_{\tau,\xi}(t,x+y)-\eta_{\tau,\xi}(t,x) \big) \nu (dy)\Big]\notag\\
&=:&R_{\tau,\xi}(t,x)+{\mathcal S}_{\tau,\xi}(t,x),
  \end{eqnarray}
and for $u\in {\mathscr C}^{\alpha+\beta}([0,T]\times \R^d) $ solving \eqref{Asso_PDE}, ${\mathcal S}_{\tau,\xi}\in L^\infty([0,T],C_b^{\beta}(\R^d,\R)) $.

Importantly,  the following representations also hold:
\begin{eqnarray}
v_{\tau,\xi}(t,x)&=&\tilde u^{(\tau,\xi)}(t,x)+\int_t^T  ds \tilde P_{s,t,\alpha}^{(\tau,\xi)} \Big( R_{\tau,\xi}(s,\cdot)\Big) (x),\notag\\
D_x v_{\tau,\xi}(t,x)&=&D_x\tilde u^{(\tau,\xi)}(t,x)+ D_x\int_t^T  ds \tilde P_{s,t,\alpha}^{(\tau,\xi)} \Big(R_{\tau,\xi}(s,\cdot)\Big)(x)
\label{DUAHM_1}
\end{eqnarray}
where the function $\tilde u^{(\tau,\xi)} $ solves 
\begin{eqnarray} 
\p_t  \tilde u^{(\tau,\xi)}(t,x)  +  {L}_{\alpha} \tilde u^{(\tau, \xi)} (t,x)+ F(t, \theta_{t,\tau}(\xi) ) \cdot D_x \tilde u^{(\tau, \xi)}(t,x)
 &=& -  \Big(\eta_{\tau, \xi} f + {\mathcal S}_{\tau,\xi}\Big)(t,x), \ (t,x)\in [0,T)\times \R^d,
\, 
\notag\\
\tilde u^{(\tau,\xi)}(T,x) &=& g(x)  \eta_{\tau, \xi}(T,x),
\quad \text{on }\R^d,\label{we_BASE}
\end{eqnarray}
Eventually,
\begin{equation}\label{FINAL_BD_FOR_GRADIENT}
\Big(D_x v_{\tau,\xi}(t,x)\Big)\Big|_{(\tau,\xi)=(t,x)}=D_xu(t,x)=\Big(D_x\tilde u^{(\tau,\xi)}(t,x)\Big)\Big|_{(\tau,\xi)=(t,x)}+ \int_t^T  ds  \Big[D_x \tilde P_{s,t,\alpha}^{(\tau,\xi)} \Big(R_{\tau,\xi}(s,\cdot)\Big)(x)\Big]\Big|_{(\tau,\xi)=(t,x)},
\end{equation}
 \end{PROP}
\begin{REM} {\em 
The above representation formulas \eqref{DUAHM_1} are crucial in the sense that they allow to write any solution $u\in {\mathscr C}^{\alpha+\beta}([0,T]\times \R^d) $ of \eqref{Asso_PDE}  localized with a cut-off along the flow in terms of the solution $\tilde u^{(\tau,\xi)} $ to equation \eqref{FROZEN_Asso_PDE} with  modified source and terminal condition, namely $ -(\eta_{\tau, \xi} f) + {\mathcal S}_{\tau,\xi}$ and $ g  \eta_{\tau, \xi}(T,\cdot) $ respectively, and the remainder term $\int_t^T  ds \tilde P_{s,t,\alpha}^{(\tau,\xi)} \Big( R_{\tau,\xi}(s,\cdot)\Big) $. Roughly speaking, the regularity of $\tilde u^{(\tau,\xi)} $ follows from Proposition \ref{SCHAUDER_FROZEN} whereas the control of the remainder will precisely be the difficult remaining part of the work for which we will also need to specify, in the representations, the appropriate values of the freezing parameters $(\tau,\xi)\in [0,T]\times \R^d $. }
 \end{REM}

\begin{proof}
Let a solution $u\in {\mathscr C}^{\alpha+\beta}([0,T]\times \R^d) $ of \eqref{Asso_PDE} be given and let us prove that $v_{\tau,\xi}=u \eta_{\tau,\xi} $ satisfies \eqref{we}. Observe that:
\begin{eqnarray*}
\partial_t v_{\tau,\xi}(t,x)&=&\partial_t u(t,x) \eta_{\tau,\xi}(t,x)+u(t,x)\partial_t \eta_{\tau,\xi}(t,x)\\
&=&\partial_t u(t,x) \eta_{\tau,\xi}(t,x) - u(t,x) D \rho(x-\theta_{t,\tau}(\xi))\cdot F(t,\theta_{t,\tau}(\xi)),\\
F(t,\theta_{t,\tau}(\xi))\cdot D v_{\tau,\xi}(t,x)&=& \big( F(t,\theta_{t,\tau}(\xi))\cdot D u(t,x)\Big) \eta_{\tau,\xi}(t,x)+u(t,x)F(t,\theta_{t,\tau}(\xi)) \cdot D \rho(x-\theta_{t,\tau}(\xi)),\\
L_\alpha(v_{\tau,\xi}(t,x))&=&\big(L_\alpha u(t,x)\big)\eta_{\tau,\xi}(t,x)
+\big(L_\alpha \eta_{\tau,\xi}(t,x)\big)u(t,x)\\
&&+\int_{\R^d} \big(u(t,x+y)-u(t,x)\big)\big(\eta_{\tau,\xi}(t,x+y)-\eta_{\tau,\xi}(t,x+y) \big) \nu (dy)\Big]\\
&=&\big(L_\alpha u(t,x)\big)\eta_{\tau,\xi}(t,x)-{\mathcal S}_{\tau,\xi}(t,x).
\end{eqnarray*}
Summing the above terms yields:
\begin{eqnarray*}
&&\partial_t v_{\tau,\xi}(t,x)+F(t,\theta_{t,\tau}(\xi))\cdot D v_{\tau,\xi}(t,x)+L_\alpha(v_{\tau,\xi}(t,x))\\
&=& (\partial_t  u+ F(t,\theta_{t,\tau}(\xi))\cdot Du(t,x)+L_\alpha u(t,x)) \eta_{\tau,\xi}(t,x)-{\mathcal S}_{\tau,\xi}(t,x)\\
&=& (\partial_t  u+ F(t,x)\cdot Du(t,x)+L_\alpha u(t,x)) \eta_{\tau,\xi}(t,x)+\Big(\big( F(t,\theta_{t,\tau}(\xi))-F(t,x)\big)\cdot Du(t,x) \big)\Big) \eta_{\tau,\xi}(t,x)-{\mathcal S}_{\tau,\xi}(t,x)\\
&=&-\Big(f(t,x)\eta_{\tau,\xi}(t,x)+R_{\tau,\xi}(t,x)+{\mathcal S}_{\tau,\xi}(t,x)\Big)=-\Big(f(t,x)\eta_{\tau,\xi}+{\mathcal R}_{\tau,\xi}\Big)(t,x),
\end{eqnarray*}
which precisely gives equation \eqref{we}. In the above right-hand side it is clear that $f\eta_{\tau,\xi}\in L^\infty\big([0,T],C_b^\beta (\R^d,\R)\big) $. Also, the mapping $${R}_{\tau,\xi}:(t,x)\in [0,T]\times \R^d\mapsto \Big(\big( F(t,\theta_{t,\tau}(\xi))-F(t,x)\big)\cdot Du(t,x) \big)\Big) \eta_{\tau,\xi}(t,x)$$
belongs to $L^\infty\big([0,T], C_b^{\alpha+\beta-1}(\R^d,\R)\big)$. 

The lower regularity is precisely due to the fact that $ Du\in L^\infty\big([0,T],C_b^{\alpha+\beta-1}(\R^d,\R)\big)$. Anyhow, this is not a problem here since we are, for the moment, simply interested in finding the representation formulas in \eqref{DUAHM_1} which will in turn allow to investigate quantitative bounds related to the smoothness of $u$ (gradient bounds and H\"older moduli). Let us now show similarly to the proof of Theorem 3.4 in \cite{prio:12} that ${\mathcal S}_{\tau,\xi}\in L^\infty([0,T],C_b^\beta(\R^d,\R)) $. In the quoted reference, the previous smoothness property is obtained for $\alpha=1 $.

Introduce the non-local operator ${\mathscr T}_{\tau,\xi}$ defined for $f\in C_b^1(\R^d,\R) $ as:
$${\mathscr T}_{\tau,\xi}f(x):=\int_{\R^d} \big(f(x+y)-f(x)\big)\big(\eta_{\tau,\xi}(x+y)-\eta_{\tau,\xi}(x)\big) \nu(dy).$$
It is direct to check that ${\mathscr T}_{\tau,\xi}$ is continuous from $C_b^1(\R^d,\R) $ to $C_b(\R^d,\R) $. Observe now that for a function $f\in C_b^{\alpha+\beta}(\R^d,\R) $:
\begin{eqnarray*}
|D{\mathscr T}_{\tau,\xi}f(x)|&\le& \Big|\int_{\R^d} \big(Df(x+y)-Df(x)\big)\big(\eta_{\tau,\xi}(x+y)-\eta_{\tau,\xi}(x)\big) \nu(dy)\Big|\\
&&+\Big|\int_{\R^d} \big(f(x+y)-f(x)\big)\big(D\eta_{\tau,\xi}(x+y)-D\eta_{\tau,\xi}(x)\big) \nu(dy)\Big|\\
&\le & \Big(\|Df\|_{C_b^{\alpha+\beta-1}} \|\eta_{\eta,\tau}\|_{C_b^1}\int_{|y|\le 1}|y|^{\alpha+\beta}\nu(dy)+4\|Df\|_\infty\int_{|y|>1}\nu (dy) \\
&&+ \|Df\|_\infty \|\eta_{\tau, \xi}\|_{C_b^2}\int_{|y|\le 1}|y|^2\nu(dy)+4\|f\|_\infty\|D\eta_{\tau,\xi}\|_\infty\int_{|y|>1}\nu (dy)  \Big)\\
&\le& C \|\eta_{\tau,\xi}\|_{C_b^2}\|f\|_{C_b^{\alpha+\beta}}.
\end{eqnarray*}
Hence, ${\mathscr T}$ is also continuous from $C_b^{\alpha+\beta}(\R^d,\R) $ into  $ C_b^1(\R^d,\R)$.
We also recall the following interpolation equality between H\"older spaces:
$$\Big(C_b^1(\R^d,\R), C_b^{\alpha+\beta}(\R^d,\R)\Big)_{\beta,\infty}=C_b^{(1-\beta)+\beta(\alpha+\beta)} (\R^d,\R),$$
see e.g. Chapter 1 in Lunardi \cite{luna:09}.Therefore the operator ${\mathscr T}$ is also continuous from $ C_b^{(1-\beta)+\beta(\alpha+\beta)} (\R^d,\R)$ into $C_b^\beta(\R^d,\R) $ (see Theorem 1.1.6 in \cite{luna:09}). Recall that, since $\alpha+\beta>1 $, we indeed have $ (1-\beta)+\beta(\alpha+\beta)<\alpha+\beta$.  To derive  the stated smoothness of ${\mathcal S}_{\tau,\xi} $ we want to apply ${\mathscr T}_{\tau,\xi} $ to $u(t,\cdot)$ where $u\in  L^\infty\big([0,T],C_b^{\alpha+\beta}(\R^d,\R)\big)$. From the above computations and since $C_b^{\alpha+\beta}(\R^d,\R)\subset C_b^{(1-\beta)+\beta(\alpha+\beta)}(\R^d,\R)$ we readily derive that ${\mathscr T}_{\tau,\xi}u \in L^\infty\big([0,T],C_b^{\beta}(\R^d,\R)\big) $. The other term in ${\mathcal S}_{\tau,\xi} $, namely $ L_\alpha \eta_{\tau,\xi}(t,x) u(t,x)$ is handled without difficulties. This concludes the proof of the statement that ${\mathcal S}_{\tau,\xi} \in L^\infty\big([0,T],C_b^{\beta}(\R^d,\R)\big) $.

\textcolor{black}{It now follows from Proposition \ref{PROP_FOURIER} that }
{
\begin{eqnarray*}
v_{\tau,\xi}(t,x)&=&\tilde P_{T,t,\alpha}^{(\tau,\xi)}\big(g\textcolor{black}{\eta_{\tau,\xi}(T,\cdot)}\big)(x) + \int_t^T  ds \tilde P_{s,t,\alpha}^{(\tau,\xi)} \Big( \eta_{\tau, \xi} f + {\mathcal S}_{\tau,\xi}(s,\cdot)\Big) (x)+\int_t^T  ds \tilde P_{s,t,\alpha}^{(\tau,\xi)} \Big( R_{\tau,\xi}(s,\cdot)\Big) (x)\\
&=&\tilde u^{(\tau,\xi)}(t,x)+\int_t^T ds \Big( \tilde P_{s,t,\alpha}^{(\tau,\xi)}R_{\tau,\xi}(s,\cdot)\Big)(x),
\end{eqnarray*}
where 
\begin{eqnarray*}
 {\color{black}\tilde u}^{(\tau,\xi)}(t,x) = \tilde P_{T,t,\alpha}^{(\tau,\xi)}\big(g\textcolor{black}{\eta_{\tau,\xi}(T,\cdot)}\big)(x) + \int_t^T  ds \tilde P_{s,t,\alpha}^{(\tau,\xi)} \Big( \eta_{\tau, \xi} f + {\mathcal S}_{\tau,\xi}(s,\cdot)\Big) (x).
\end{eqnarray*}
 {\color{black}Also, from Definition \ref{FROZEN_SEMI_GROUP}, Proposition \ref{SCHAUDER_FROZEN} and \ref{PROP_FOURIER}, $\tilde u^{(\tau,\xi)} $ is the unique solution in \textcolor{black}{${\mathscr C}_b^{\alpha+\beta}([0,T]\times \R^d)$ of \eqref{we_BASE}}.}
Therefore the first Duhamel representation formula in \eqref{DUAHM_1} holds.

}
Since both $D\tilde v_{\tau,\xi}(t,x) $ and $D\tilde u^{(\tau,\xi)}(t,x) $ exist (recall indeed that $\tilde v_{\tau,\xi} =u\eta_{\tau,\xi} $), we deduce that 
$$D_x \int_t^T ds \Big( \tilde P_{s,t,\alpha}^{(\tau,\xi)}R_{\tau,\xi}(s,\cdot)\Big)(x)$$ 
is also meaningful. This proves the second assertions of \eqref{DUAHM_1}. Eventually, recall now that under $({\mathscr P}_\beta) $, one has for $(\tau,\xi)=(t,x) $:
\begin{eqnarray}
&&\Big(|D_x\big( \tilde P_{s,t,\alpha}^{(\tau,\xi)}R_{\tau,\xi}(s,\cdot)\big)(x)|\Big)\Big|_{(\tau,\xi)=(t,x)}\notag\\
&\le&\int_{\R^d }dy \Big(|D_x \tilde p_{\alpha}^{(\tau,\xi)}(t,s,x,y)\big( F(s,y)-F(s,\theta_{s,\tau}(\xi)) \big) Du(s,y) \eta_{\tau,\xi}(y)|\Big)\Big|_{(\tau,\xi)=(t,x)} \notag\\ 
&\le & {K_{0}}\|Du\|_\infty \int_{\R^d} dy \Big(|D_x\tilde p_\alpha^{\tau,\xi}(t,s,x,y)| |y-\theta_{s,\tau}(\xi)|^\beta\Big)\Big|_{(\tau,\xi)=(t,x)}\notag\\
&\le & {K_{0}}\|Du\|_\infty \int_{\R^d} dy \Big(|D_xp_\alpha\big(s-t,y-m_{s,t}^{(\tau,\xi)}(x)\big)| |y-\theta_{s,\tau}(\xi)|^\beta\Big)\Big|_{(\tau,\xi)=(t,x)}\notag\\
&\le & {K_{0}}\|Du\|_\infty \int_{\R^d} dy |D_xp_\alpha\big(s-t,z\big)| \Big|_{z=y-\theta_{s,t}(x)} |y-\theta_{s,t}(x)|^\beta\notag\\
&\underset{({\mathscr P}_\beta)}{\le} & \frac{C{K_{0}}\|Du\|_\infty}{(s-t)^{\frac1\alpha-\frac \beta\alpha}},\label{INTEGRABILITY_BOUND_WITH_RIGHT_PARAMETERS_GRAD}
\end{eqnarray}
recalling from \eqref{CORRES_LINEARIZED_FLOW_AND_FLOW} that $m_{s,t}^{(\tau,\xi)}(x)|_{(\tau,\xi)=(t,x)}=\theta_{s,t}(x) $ for the last but one inequality.
Hence, we derive the following important control
\begin{eqnarray}
&&\Big(D_x \int_t^T ds\big( \tilde P_{s,t,\alpha}^{(\tau,\xi)}R_{\tau,\xi}(s,\cdot)\big)(x) \Big)\Big|_{(\tau,\xi)
=(t,x)}\notag\\
&=&\lim_{\varepsilon \rightarrow 0} \varepsilon^{-1}\Big(\int_{t}^T ds\int_{\R^d}dy \Big[\tilde p_\alpha\big(s-t,y-(x+\varepsilon+\int_t^s F(v,\theta_{s,\tau}(\xi)) dv)\big)\notag\\
&&-\tilde p_\alpha\big(s-t,y-(x+\int_t^s F(v,\theta_{s,\tau}(\xi)) dv)\big) \Big]\Big\{F(s,y)-F(s,\theta_{s,\tau}(\xi)) \Big\}Du(s,y)\eta_{\tau,\xi}(s,y)\Big)_{(\tau,\xi)=(t,x)}\notag\\
&=&\lim_{\varepsilon \rightarrow 0} \varepsilon^{-1}\Big(\int_{t}^T ds\int_{\R^d}dy \Big[\tilde p_\alpha\big(s-t,y-(x+\varepsilon+\int_t^s F(v,\theta_{s,t}(x)) dv)\big)\notag\\
&&-\tilde p_\alpha\big(s-t,y-(x+\int_t^s F(v,\theta_{s,t}(x)) dv)\big) \Big] \Big\{F(s,y)-F(s,\theta_{s,t}(x)) \Big\}Du(s,y)\eta_{t,x}(s,y)\Big)\notag\\
&=&\Big(\int_{t}^T ds\int_{\R^d}dy \lim_{\varepsilon \rightarrow 0} \varepsilon^{-1}\Big[\tilde p_\alpha\big(s-t,y-(x+\varepsilon+\int_t^s F(v,\theta_{s,t}(x)) dv)\big)\notag\\
&&-\tilde p_\alpha\big(s-t,y-(x+\int_t^s F(v,\theta_{s,t}(x)) dv)\big) \Big] \Big\{F(s,y)-F(s,\theta_{s,t}(x)) \Big\}Du(s,y)\eta_{t,x}(s,y)\Big)\notag\\
&=&\int_t^T ds\Big(D\big( \tilde P_{s,t,\alpha}^{(\tau,\xi)}R_{\tau,\xi}(s,\cdot)\big)(x)\Big)\Big|_{(\tau,\xi)=(t,x)},\label{THE_SWAP_LIMITS}
\end{eqnarray}
where \eqref{INTEGRABILITY_BOUND_WITH_RIGHT_PARAMETERS_GRAD} allowed to use the bounded convergence theorem in the last but one equality. Equation \eqref{THE_SWAP_LIMITS} in turn yields \eqref{FINAL_BD_FOR_GRADIENT}. Indeed, 
 $D_x v_{\tau,\xi}(t,x)=D_x u(t,x) \eta_{\tau,\xi}(t,x)+u(t,x)D_x \eta_{\tau,x}(x) $ and for $(\tau,\xi=(t,x)) $ one has $ \big(\eta_{\tau,\xi}(t,x)\big)_{(\tau,\xi)=(t,x)}=1, (D_x \eta_{\tau,\xi}(t,x))_{(\tau,\xi)=(t,x)}=0 $.
 
\end{proof}

\begin{REM}[\em About the localization in the Duhamel formula]
{\em We mention that the localization with the cut-off $\eta_{\tau,\xi}$ is precisely needed because we imposed in \eqref{22}  a local H\"older continuity condition. Anyhow, even if we had assumed a global H\"older assumption, such a localization would still be needed to give a meaning to the first identity in \eqref{DUAHM_1} when $\alpha<1/2$ (recall that $({\mathscr P}_\beta) $ involves  the derivatives of the heat-kernel). For $\alpha >1/2 $ and $\beta<\alpha $ such a localization could have been avoided for a globally $\beta $-H\"older continuous drift $F$.}
\end{REM}

\subsection{Derivation of the main \textit{a priori} estimates 
}
From the representations \eqref{DUAHM_1} and \eqref{FINAL_BD_FOR_GRADIENT} in Proposition \ref{DUHAMEL} we see that, since we also know from Proposition \ref{SCHAUDER_FROZEN} that $\tilde u^{(\tau,\xi)} $ is itself \textit{smooth}, the main term which remains to be investigated is the remainder $  \int_t^T ds \big(\tilde P_{s,t,\alpha}^{(\tau,\xi)} R(s,\cdot)\big)(x) $.
\vskip 1mm

{\color{black} In the following, we first give bounds for the  solution $u\in L^\infty\big([0,T],C_b^{\alpha+\beta}(\R^d,\R)\big)$. 
\textcolor{black}{Estimates for the} supremum norm of the solution and its gradient are given in Lemma \ref{LEMMA_SUP_BOUNDS_SOL_AND_GRAD} and the control of the H\"older modulus of the gradient are \textcolor{black}{stated} in Lemma \ref{LEMME_HOLDER_GRAD}. Then, we eventually  prove Theorem  \textcolor{black}{\ref{THEO_SCHAU_ALPHA}} in paragraph \ref{CTR_HOLD_GRAD_SUBSUB}.\\

We emphasize that \textcolor{black}{as the proof of Lemma \ref{LEMME_HOLDER_GRAD}}  requires a thorough analysis (namely a refinement of Proposition \ref{DUHAMEL}, in order to consider different freezing points in function of the diagonal and off-diagonal regimes introduced in the proof of Proposition \ref{SCHAUDER_FROZEN}), \textcolor{black}{it will be postponed to} the next subsection \ref{PROOF_LEM_HOLD_GRAD_SUB}.}

\subsubsection{Control of the supremum norms for the solution and its gradient and associated H\"older modulus}\label{CONTROL_SUP_GRAD_SUBSUB}
As an important corollary of Proposition \ref{DUHAMEL}, we get the next estimates for the supremum norm of $u$ and its gradient.
\begin{lem}[Control of the supremum norm of the solution and the gradient]\label{LEMMA_SUP_BOUNDS_SOL_AND_GRAD}
Assume \A{A} (thus property $(\mathscr P_\beta) $ holds). Let $u \in {\mathscr C}^{\alpha+\beta}([0,T]\times \R^d)$ be a solution of \eqref{Asso_PDE}. There exists a constant $C:=C(\A{A})$ s.t. for all $(t,x)\in [0,T]\times  \R^d$:
 \begin{eqnarray}
|u(t,x)|&\le&  \|g\|_\infty+T\|f\|_\infty+ {K_{0}} \|Du \|_\infty T,\label{BD_SUP}\\
|Du(t,x)|&\le & \|Dg\|_\infty + C(T-t)^{\frac{\alpha+\beta-1}{\alpha}}\big([f]_{\beta,T}+\|u\|_{L^\infty([0,T],C_b^{\alpha+\beta})} +{K_{0}}\|Du\|_\infty\big)\notag\\
&\le & \|Dg\|_\infty + C(T-t)^{\frac{\alpha+\beta-1}{\alpha}}\big([f]_{\beta,T}+\|u\|_{L^\infty([0,T],C_b^{\alpha+\beta})}(1+{K_{0}})\big) .\label{BD_SUP_GRAD}
\end{eqnarray}
\end{lem}
\begin{proof}
Equation \eqref{BD_SUP} readily follows from \eqref{DUAHM_1} taking $ (\tau,\xi)=(t,x)$ and observing that $v_{(\tau,\xi)}(t,x)|_{(\tau,\xi)=(t,x)}=u(t,x)\eta_{t,x}(t,x)=u(t,x) $.
To derive \eqref{BD_SUP_GRAD} we start from \eqref{FINAL_BD_FOR_GRADIENT} to write:
\begin{eqnarray*}
|D_xu(t,x)|\le |D_x \tilde u^{(\tau,\xi)}(t,x)|\Big|_{(\tau,\xi)=(t,x)}+|\int_t^T ds\Big(D\big( \tilde P_{s,t,\alpha}^{(\tau,\xi)}R(s,\cdot)\big)(x)\Big)\Big|_{(\tau,\xi)=(t,x)}|. 
\end{eqnarray*}
From equation \eqref{we_BASE} and the proof of Proposition \ref{SCHAUDER_FROZEN} (see equation \eqref{CTR_GRAD_SUP_FROZEN}) and \eqref{INTEGRABILITY_BOUND_WITH_RIGHT_PARAMETERS_GRAD} we thus derive, recalling that $1/\alpha -\beta/\alpha<1 $:
\begin{eqnarray}
|D_xu(t,x)|&\le& \|Dg\|_\infty+C\Big( (T-t)^{\frac{\alpha+\beta-1}{\alpha}}\big([f\eta_{\tau,\xi}+{\mathcal S}_{\tau,\xi}]_{\beta,T}\big) \Big|_{(\tau,\xi)=(t,x)} +{K_{0}}\|Du\|_\infty\int_t^T \frac{ds}{(s-t)^{\frac 1\alpha-\frac\beta\alpha}}\Big)\notag\\
&\le& \|Dg\|_\infty+C(T-t)^{\frac{\alpha+\beta-1}{\alpha}} \Big( \big([f\eta_{\tau,\xi}+{\mathcal S}_{\tau,\xi}]_{\beta,T}\big) \Big|_{(\tau,\xi)=(t,x)} +{K_{0}}\|Du\|_\infty\Big).\label{PREAL_GRAD_BD}
\end{eqnarray}
It therefore remains to precise the quantity $\big([f\eta_{\tau,\xi}+{\mathcal S}_{\tau,\xi}]_{\beta,T}\big) \Big|_{(\tau,\xi)=(t,x)}\le   \big([f\eta_{\tau,\xi}]_{\beta,T} +[{\mathcal S}_{\tau,\xi}]_{\beta,T}\big) \Big|_{(\tau,\xi)=(t,x)}$. We have,
\begin{eqnarray}
[f\eta_{\tau,\xi}]_{\beta,T}&\le& [f]_{\beta,T}+C\|f\|_\infty\le C \|f\|_{L^\infty([0,T],C_b^\beta)},\notag\\
\ [{\mathcal S}_{\tau,\xi}]_{\beta,T}&\le& [(L_\alpha \eta_{\tau,\xi}) u]_{\beta,T}+[{\mathscr T}_{\tau,\xi}]_{\beta,T} \le C\Big(\|u\|_\infty+\|Du\|_\infty+ \|u\|_{L^\infty([0,T],C_b^{(1-\beta)+\beta(\alpha+\beta)})}\Big)\notag\\
&\le& C_{\beta,\alpha}\big( (1+\varepsilon^{-1})(\|u\|_\infty +\|Du\|_\infty )+\varepsilon \|u\|_{L^\infty([0,T],C_b^{\alpha+\beta})}\big), \label{CTR_HOLDE_MODULI}
\end{eqnarray}
for any $\varepsilon\in (0,1) $ using for the last inequality that for all $t\in [0,T] $, the usual interpolation inequality
$$[u(t,\cdot)]_{(1-\beta)+\beta(\alpha+\beta)}\le [u(t,\cdot)]_{\alpha+\beta}^{s/(\alpha+\beta)} [u(t,\cdot)]_1^{ 1- s/(\alpha+\beta)},
$$ 
for $s=(1-\beta)+\beta(\alpha+\beta)-1$ (see e.g. \cite{kryl:96}, p. 40, (3.3.7)) then yields from the Young inequality $[u(t,\cdot)]_{(1-\beta)+\beta(\alpha+\beta)}\le C_{\beta,\alpha} \big(\varepsilon [u(t,\cdot)]_{\alpha+\beta}+\varepsilon^{-1} [u(t,\cdot)]_1\big)
$. Plugging \eqref{CTR_HOLDE_MODULI} with $\varepsilon=1/2 $ into \eqref{PREAL_GRAD_BD} gives \eqref{BD_SUP_GRAD}.
This completes the proof.
\end{proof}
\begin{REM}[On the $\beta $-H\"older modulus of ${\mathcal S}_{\tau,\xi} $] \label{MOD_HOLDER_S_TAU_XI} {\em 
We eventually mention that equation \eqref{CTR_HOLDE_MODULI} will also be crucial, for a parameter $\varepsilon $ sufficiently small, when investigating the H\"older modulus of the gradient, in order to make the circular argument working.}
\end{REM}

Concerning the H\"older modulus of the gradient of the solution we have the following control whose proof is presented in the next section.
\begin{lem}[H\"older modulus of the gradient]\label{LEMME_HOLDER_GRAD}
Assume \A{A}. Let $u \in {\mathscr C}^{\alpha+\beta}([0,T]\times \R^d)$ be a solution of \eqref{Asso_PDE}. There exists two constant $C_1:=C(\A{A})>0$ and $C_2:=C(\A{A})>0$ such that
   \begin{eqnarray}
 [D u]_{\beta+\alpha-1, T} &\le& C_1\Big\{\Big((1+c_0)\|g\|_{C_b^{\alpha+\beta}} + \Big(c_0^{\frac{\alpha+\beta-1}{\alpha}}[f]_{\beta,T} +  T^{\frac{\alpha+\beta-1}\alpha}\|f\|_{L^\infty([0,T],C_b^\beta)}\Big)\notag\\
 &&+K_0 (c_0^{1+\frac{\beta-2}{\alpha}}+ c_0^{\frac {\alpha+\beta-1}{\alpha}})\|Du\|_{L^\infty } \Big\}\notag\\
 &&+ \Big(\frac 14 + C_2c_0^{\frac{\alpha+\beta-1}{\alpha}}(1+{K_{0}})\Big)\|u\|_{L^\infty([0,T],C_b^{\alpha+\beta})}.\label{CTR_HOLDER_MOD_PROV}
 \end{eqnarray}
\end{lem}


\subsubsection{Final derivation of Theorem \ref{THEO_SCHAU_ALPHA}: Schauder estimate for the solution of \eqref{Asso_PDE}}\label{CTR_HOLD_GRAD_SUBSUB}
Observe carefully from the above Lemmas that the norm of $u$ appears in the r.h.s of the previous controls. However, those contributions are multiplied either by a constant $c_0$, either by a function of $T$ or a small constant. Provided these quantities can be chosen small enough, we can conclude the proof of our main estimates through a circular argument, i.e., the norms in the r.h.s. will be absorbed by those on the l.h.s. When doing so, we end up with Schauder estimates in small time only. To extend it to an arbitrary  fixed horizon $T$, we eventually use the fact that Schauder estimates precisely provide a kind of stability result in the class ${\mathscr C}^{\alpha+\beta}([0,T]\times \R^d) $ so that the final bound follows by inductive application of the estimate in small time.
\vskip 1mm

%
%

{ \color{black}

 Note now that thanks to \eqref{BD_SUP}, \eqref{BD_SUP_GRAD} we have from \eqref{CTR_HOLDER_MOD_PROV} that
  \begin{eqnarray*}
&&\|u\|_{L^{\infty}} + \|Du\|_{L^{\infty}} + [D u]_{\beta+\alpha-1, T} \\
&\le& C_1\Big\{\Big((1+c_0+K_0 (c_0^{1+\frac{\beta-2}{\alpha}}+ c_0^{\frac {\alpha+\beta-1}{\alpha}})+K_0T)\|g\|_{C_b^{\alpha+\beta}}\\
 &&+ \Big(c_0^{\frac{\alpha+\beta-1}{\alpha}}[f]_{\beta,T} +\big(1+K_0 (c_0^{1+\frac{\beta-2}{\alpha}}+ c_0^{\frac {\alpha+\beta-1}{\alpha}})+K_0T\big)(T-t)^{\frac{\alpha+\beta-1}{\alpha}}\|f\|_{L^\infty([0,T],C_b^\beta)}\Big)\Big\}\\
 &&+\Psi(K_0,c_0,\alpha,\beta,\A{A})\|u\|_{L^\infty([0,T],C_b^{\alpha+\beta})}
 \end{eqnarray*}
up to a modification of $C_1$ and where
\begin{eqnarray*}
\Psi(K_0,c_0,\alpha,\beta,\A{A}) =\frac 14 + C_2c_0^{\frac{\alpha+\beta-1}{\alpha}}(1+{K_{0}}) + C(T-t)^{\frac{\alpha+\beta-1}{\alpha}}(1+{K_{0}})\Big(1+K_0T+C_1K_0 (c_0^{1+\frac{\beta-2}{\alpha}}+ c_0^{\frac {\alpha+\beta-1}{\alpha}})\Big).
 \end{eqnarray*}
We can hence choose $c_0$ small enough so that $C_2c_0^{\frac{\alpha+\beta-1}{\alpha}}(1+{K_{0}})\leq 1/4$ and then $T$ small enough so that $C(T-t)^{\frac{\alpha+\beta-1}{\alpha}}(1+{K_{0}})\Big(1+K_0T+C_1K_0 (c_0^{1+\frac{\beta-2}{\alpha}}+ c_0^{\frac {\alpha+\beta-1}{\alpha}})\Big)\leq 1/4$. This} eventually yields that there exists $\tilde C:=\tilde C (\A{A}, c_0)$ s.t.
\begin{eqnarray}
\|u\|_{L^\infty([0,T],C_b^{\alpha+\beta})}&\le& \tilde C \Big( \|g\|_{C_b^{\alpha+\beta}}+\|f\|_{L^\infty([0,T],C_b^\beta)}\Big)  +\frac{3}{4}\|u\|_{L^\infty([0,T],C_b^{\alpha+\beta})}\notag\\
&\le & 4\tilde C \Big( \|g\|_{C_b^{\alpha+\beta}}+\|f\|_{L^\infty([0,T],C_b^\beta)}\Big) .\label{SCHAUDER_SMALL_TIME}
\end{eqnarray}
Equation \eqref{SCHAUDER_SMALL_TIME} provides the control of Theorem \ref{THEO_SCHAU_ALPHA} for $T$ small enough, i.e., for $T\le T_0(\A{A}, ({\mathscr P}_\beta)) $. The result is extended to an arbitrary time $T$ considering $N $ subintervals of $[0,T] $ s.t. $T/N\le T_0 $ and applying inductively \eqref{SCHAUDER_SMALL_TIME} going backwards in time on the time intervals $[(i-1)T/N,iT/N] $, $i\in \leftB 1,N\rightB $ considering as \textit{final} condition on the current time interval the function $g_i(x):=u(iT/n,x) $ which precisely belongs to $C_b^{\alpha+\beta} $ from the previous application of the Schauder estimate \eqref{SCHAUDER_SMALL_TIME} if $i<N $ or by \eqref{Asso_PDE} if $i=N$.  This proves Theorem \ref{THEO_SCHAU_ALPHA} is complete provided Lemma \ref{LEMME_HOLDER_GRAD} holds.\qed 

\subsection{Proof of Lemma \ref{LEMME_HOLDER_GRAD}}\label{PROOF_LEM_HOLD_GRAD_SUB}
Let $ t\in [0,T]$ be fixed. For this part of the analysis, we distinguish two cases, either the given points $(x,x')\in \R^d $ are for a fixed $t\in [0,T) $ in a globally \textit{off-diagonal} regime, i.e., $c_0|x-x'|^\alpha\ge (T-t)$ for a constant $c_0$ to be specified but meant to be small. This means that the spatial distance is larger than the characteristic time-scale up to a prescribed constant which will be useful to equilibrate the computations. In this case, we will \textit{mainly} use the controls of Lemma \ref{LEMMA_SUP_BOUNDS_SOL_AND_GRAD}.


In the \textit{diagonal} case, $c_0|x-x'|^\alpha\le (T-t) $, the spatial points are closer than the typical time-scale magnitude but in the time integration for the source and the perturbative term (see e.g. \eqref{DER_DUHAMEL} below), when $ (s-t)\le c_0|x-x'|^\alpha$ there is again a \textit{local off-diagonal regime}. The key point is that to handle these terms properly it will be useful to be able to change freezing point, i.e., it seems reasonable that, when the spatial points are in a \textit{local diagonal regime}, i.e., $(s-t)\ge c_0|x-x'|^\alpha $,  the auxiliary frozen densities are considered for the same freezing parameter and conversely that in the \textit{locally off-diagonal regime} the densities are frozen along their own spatial argument (similarly to equation \eqref{FINAL_BD_FOR_GRADIENT} in Lemma \ref{LEMMA_SUP_BOUNDS_SOL_AND_GRAD}). We are thus faced with a change of freezing point in the Duhamel formulation. This approach was already used in \cite{chau:hono:meno:18} to obtain Schauder estimates for degenerate local Kolmogorov equations and can be used in the current setting. 

%

\subsubsection{Off-Diagonal Regime}
Let $x,x' \in \R^d$ be s.t. $c_0|x-x'|^\alpha\ge (T-t)  $. In that case, {\color{black} we claim that:
\begin{equation}\label{CTR_HOLDER_HD_LEMMA9}
|Du(t,x)-Du(t,x')|
\le C|x-x'|^{\alpha+\beta-1}\Big([Dg]_{{\beta+\alpha-1}}(1+c_0)+c_0^{\frac{\alpha+\beta-1}{\alpha}}\big([f]_{\beta,T}+\|u\|_{L^\infty([0,T],C_b^{\alpha+\beta})}(1+{K_{0}})\big)\Big).
\end{equation}
}
Indeed, we readily get from \eqref{FINAL_BD_FOR_GRADIENT}, \eqref{BD_SUP_GRAD} and the proof of Lemma \ref{LEMMA_SUP_BOUNDS_SOL_AND_GRAD} that:
\begin{eqnarray}
|Du(t,x)-Du(t,x')|\le C\Bigg(\Big|D\tilde P_{T,t,\alpha}^{(\tau,\xi)} (g\eta_{\tau,\xi}(T,\cdot))(x)\big|_{(\tau,\xi)=(t,x)}-D\tilde P_{T,t,\alpha}^{(\tau,\xi')} (g\eta_{\tau,\xi'}(T,\cdot))(x')\big|_{(\tau,\xi')=(t,x')} \Big|\notag\\
\Bigg|\int_t^{T}ds \int_{\R^d} dy \Big(D_x \tilde p_\alpha^{(\tau,\xi)}(t,s,x,y)\big( \eta_{\tau,\xi}f- {\mathcal S}_{\tau,\xi}(s,y)\big)\Big)\Big|_{(\tau,\xi)=(t,x)} \Bigg|\notag\\+\Bigg|\int_t^{T}ds \int_{\R^d} dy \Big(D_x \tilde p_\alpha^{(\tau,\xi)}(t,s,x',y)\big(\eta_{\tau,\xi'} f- {\mathcal S}_{\tau,\xi'}\big)(s,y)\Big)\Big|_{(\tau,\xi')=(t,x')} \Bigg|\notag\\
+\Bigg|\int_t^{T}ds \int_{\R^d} dy\Big(D_x \tilde p_\alpha^{(\tau,\xi)}(t,s,x,y) \big( F(s,y)-F(s,\theta_{s,\tau}(\xi)))\cdot Du(s,y) \big) \eta_{\tau,\xi}(s,y)\Big)\Big|_{(\tau,\xi)=(t,x)} \Bigg|\notag\\
+\Bigg|\int_t^{T}ds \int_{\R^d} dy\Big(D_x \tilde p_\alpha^{(\tau,\xi')}(t,s,x',y) \big( F(s,y)-F(s,\theta_{s,\tau}(\xi')))\cdot Du(s,y) \big) \eta_{\tau,\xi'}(s,y)\Big)\Big|_{(\tau,\xi')=(t,x')} \Bigg|\Bigg)\notag\\
\le C \Big(\Big|D\tilde P_{T,t,\alpha}^{(\tau,\xi)} (g\eta_{\tau,\xi}(T,\cdot))(x)\big|_{(\tau,\xi)=(t,x)}-D\tilde P_{T,t,\alpha}^{(\tau,\xi')} (g\eta_{\tau,\xi'}(T,\cdot))(x')\big|_{(\tau,\xi')=(t,x')} \Big| \notag\\
+(T-t)^{\frac{\alpha+\beta-1}{\alpha}}\big([f]_{\beta,T}+\|u\|_{L^\infty([0,T],C_b^{\alpha+\beta})}(1+{K_{0}})\big)\Big)\notag\\
C \Big(\Big|D\tilde P_{T,t,\alpha}^{(\tau,\xi)} (g\eta_{\tau,\xi}(T,\cdot))(x)\big|_{(\tau,\xi)=(t,x)}-D\tilde P_{T,t,\alpha}^{(\tau,\xi')} (g\eta_{\tau,\xi'}(T,\cdot))(x')\big|_{(\tau,\xi')=(t,x')} \Big| \notag\\
+c_0^{\frac{\alpha+\beta-1}{\alpha}}|x-x'|^{\alpha+\beta-1}\big([f]_{\beta,T}+\|u\|_{L^\infty([0,T],C_b^{\alpha+\beta})}(1+{K_{0}})\big)\Big).\label{THE_HOLDER_GRAD_OFF_DIAG}
\end{eqnarray}
The last term in the r.h.s. gives a \textit{good} control in the sense that provided $c_0$ is small, the above bound is compatible with the previously indicated circular argument to absorb the norms of $u$ in the r.h.s. Hence to conclude in that case it remains to handle the difference of the frozen semi-groups observed at different freezing points in space. 

 Recall from \eqref{CORRESP_SHIFTED} that $ D_x \tilde p_\alpha^{(\tau,\xi)}(t,T,x,y)=-D_y\tilde p_\alpha^{(\tau,\xi)}(t,T,x,y)$ and  write:
\begin{eqnarray}
&&\Big(D_{x} \tilde P_{T,t,\alpha}^{(\tau,\xi)}\big(g\eta_{\tau,\xi}(T,\cdot)\big)(x)- D_{x} \tilde P_{T,t,\alpha}^{(\tau,\xi')}\big(g \eta_{\tau,\xi'}(T,\cdot)\big)(x')\Big)
\nonumber \\
&=&
\bigg [ \int_{\R^{d}}  \tilde p_\alpha^{(\tau,\xi)}(t,T,x,y) D
\big( g\eta_{\tau,\xi}\big)(y) dy- \int_{\R^{d}} \tilde p_\alpha^{(\tau,\xi')}(t,T,x',y) D
\big( g \eta_{\tau,\xi'}\big)(y) dy \bigg ]
\nonumber \\
&=&\bigg [ \int_{\R^{d}}  \tilde p_\alpha^{(\tau,\xi)}(t,T,x,y) [D
\big(g \eta_{\tau,\xi}\big)(y)-D
\big( g \eta_{\tau,\xi} \big)(\theta_{T,t}(\xi))] dy\notag\\
&&+[D
 \big( g\eta_{\tau,\xi}\big)(\theta_{T,t}(\xi))-D
 \big( g\eta_{\tau,\xi'}\big) (T, \theta_{T,t}(\xi'))]- \int_{\R^{d}} \tilde p_\alpha^{(\tau,\xi')}(t,T,x',y) [D
 \big(g\eta_{\eta,\xi'}\big)(y)-D
 \big(g \eta_{\eta,\xi'} \big)(\theta_{T,t}(\xi'))]dy\bigg],\notag
\end{eqnarray}
recalling that $\tilde p_\alpha^{(\tau,\xi)}(t,T,x,\cdot),\ \tilde p_\alpha^{(\tau,\xi')}(t,T,x',\cdot) $ are probability densities for the last equality.  Taking $\tau=t,\xi=x,\xi'=x' $ we now write from  \eqref{CORRESP_SHIFTED} (recall that $\alpha + \beta -1 < \alpha$) and observing from the definition of the cut-off functions $\eta_{\tau,\xi},\eta_{\tau,\xi'} $ that $\eta_{\tau,\xi}(\theta_{T,\tau}(\xi))\big|_{(\tau,\xi)=(t,x)}=\eta_{\tau,\xi'}(\theta_{T,\tau}(\xi'))\big|_{(\tau,\xi')=(t,x')}=1$, $D \eta_{\tau,\xi}(\theta_{T,\tau}(\xi))\big|_{(\tau,\xi)=(t,x)}=D \eta_{\tau,\xi'}(\theta_{T,\tau}(\xi'))\big|_{(\tau,\xi')=(t,x')}=0 $: 
\begin{eqnarray*}
&&|D_{x} \tilde P_{T,t,\alpha}^{(\tau,\xi)}\big(g\eta_{\tau,\xi}(T,\cdot) \big)(x)- D_{x} \tilde P_{T,t}^{(\tau,\xi')}\big( g\eta_{\tau,\xi'}(T,\cdot)\big)(x')|_{(\tau,\xi,\xi')=(t,x,x')}
\nonumber \\
&\le &C\bigg [ \int_{\R^{d}}   p_\alpha(T-t,y-m_{T,t}^\xi(x)) \big([D
g]_{{\beta+\alpha-1}}+1\big)|y-\theta_{T,t}(\xi)|^{\beta+\alpha-1} dy
+[D
g]_{{\beta+\alpha-1}}|\theta_{T,t}(\xi)-\theta_{T,t}(\xi')|^{\alpha+\beta-1}\notag\\
&&+ \int_{\R^{d}}  p_\alpha(T-t,y-m_{T,t}^{\xi'}(x')) \big([D
g]_{{\beta+\alpha-1}}+1\big)|y-\theta_{T,t}(\xi')|^{\beta+\alpha-1} dy\Bigg]\Big|_{(\xi,\xi')=(x,x')},\notag\\
&\le &C\bigg [ \int_{\R^{d}}   p_\alpha(T-t,y-\theta_{T,t}(x)) \big([D
g]_{{\beta+\alpha-1}}+1\big)|y-\theta_{T,t}(x)|^{\beta+\alpha-1} dy
+[D
g]_{{\beta+\alpha-1}}|\theta_{T,t}(x)-\theta_{T,t}(x')|^{\alpha+\beta-1}\\
&&+ \int_{\R^{d}}  p_\alpha(T-t,y-\theta_{T,t}(x'))\big( [D
g]_{{\beta+\alpha-1}}+1\big)|y-\theta_{T,t}(x')|^{\beta+\alpha-1} dy\Bigg].\notag
\end{eqnarray*}
With the corresponding scaling (cf.  \A{NDa})
we derive:
\begin{eqnarray}
|D_{x} \tilde P_{T,t,\alpha}^{(\tau,\xi)}\big( g\eta_{\tau,\xi}(T,\cdot)\big)(x)- D_{x} \tilde P_{T,t}^{(\tau,\xi')}\big( g\eta_{\tau,\xi'}(T,\cdot) \big)(x')|\Big|_{(\tau,\xi,\xi')=(t,x,x')}\notag\\
\le C([D g]_{{\beta+\alpha-1}}+1)\Big[(T-t)^{\frac{\beta+\alpha-1}{\alpha}}+|\theta_{T,t}(x)-\theta_{T,t}(x')|^{\alpha+\beta-1}\Big].\label{PREAL_GRAD_SG_HD}
\end{eqnarray}
From the spatial regularity of $F$ we have the following key result whose proof is postponed to the Appendix.
\begin{lem}[Controls on the flows] \label{FLOW_LEMMA}
Let $\alpha+\beta>1 $ and $F$ satisfying  \eqref{22}.  Then there exists a constant $C\ge 1$ s.t. for all $0\le t\le s\le T\le 1,\ (x,x')\in (\R^d)^2 $:
\begin{equation}
\label{equiv_flow}
|\theta_{s,t}(x)-\theta_{s,t}(x')|\le C(|x-x'|+(s-t)^{\frac{1}{1-\beta}})\le C(|x-x'|+(s-t)^{\frac{1}{\alpha}}).
\end{equation}
\end{lem}
From \eqref{equiv_flow} with $ s=T$ and \eqref{PREAL_GRAD_SG_HD}, we therefore derive that:
\begin{eqnarray}
&&|D_{x} \tilde P_{T,t,\alpha}^{(\tau,\xi)}\big(g \eta_{\tau,\xi}(T,\cdot)\big)(x)- D_{x} \tilde P_{T,t}^{(\tau,\xi')}\big( g\eta_{\tau,\xi'}\big)(x')| \Big|_{(\tau,\xi,\xi')=(t,x,x')}\notag\\
&\le& C([Dg]_{{\beta+\alpha-1}}+1)\Big[(T-t)^{\frac{\beta+\alpha-1}{\alpha}}+|x-x'|^{\alpha+\beta-1}\Big]\notag\\
&\le &C([Dg]_{{\beta+\alpha-1}}+1)(1+c_0)|x-x'|^{\alpha+\beta-1}.\label{CTR_GRAD_SG_HD}
\end{eqnarray}
{\color{black}Plugging \eqref{CTR_GRAD_SG_HD} into \eqref{THE_HOLDER_GRAD_OFF_DIAG} eventually yields for the off-diagonal regime that \eqref{CTR_HOLDER_HD_LEMMA9} holds.}

\subsubsection{Diagonal-Regime} It is here assumed that for given points $(t,x,x')\in [0,T]\times (\R^d)^2 $,  $c_0|x-x'|^\alpha\le (T-t)  $. \textbf{All the statements of the paragraph tacitly assume this condition holds}. We first need here to consider a Duhamel representation formula for which we change freezing point along the time integration variable.
With the previous notations of Proposition \ref{DUHAMEL} the following expansion of $u$ and its gradient holds.


\begin{PROP}[Duhamel formula with change of freezing points]\label{DUHAMEL_2_THE_COME_BACK}
Let $u\in {\mathscr C}_b^{\alpha+\beta}([0,T]\times \R^d) $ be a solution of \eqref{Asso_PDE}. For fixed 
$(t,x')\in [0,T]\times \R^d $ and any freezing parameters $(\tau,\xi',\tilde \xi') \in [0,T]\times (\R^d)^2$, for all $\tau_0\in [t,T]$:
\begin{eqnarray}
 v_{\tau,\xi'}(t,x')&=&\tilde P_{T,t,\alpha}^{(\tau,\tilde \xi')}(g\eta_{\tau,\tilde \xi'}(T,\cdot))(x')+\tilde G_{\tau_0,t,\alpha}^{(\tau,\xi')} (f\eta_{\tau,\xi'} -{\mathcal S}_{\tau,\xi'})(t,x')+\tilde G_{T,\tau_0,\alpha}^{(\tau,\tilde \xi')} (f\eta_{\tau,\tilde \xi'} -{\mathcal S}_{\tau,\tilde \xi'})(t,x')\notag\\
&&+\tilde P_{\tau_0,t,\alpha}^{(\tau,\xi')} v_{\tau,\xi'}(\tau_0, x')-\tilde P_{\tau_0,t,\alpha}^{(\tau,\tilde \xi')} v_{\tau,\tilde \xi'}(\tau_0, x')\notag\\
&&+\int_t^T ds \int_{\R^{d}}dy\Bigg(\I_{s\le \tau_0}\tilde p_\alpha^{(\tau,\xi')}(t,s,x',y)\Big( \big(F(s,y)-F(s,\theta_{s,\tau}(\xi'))\big)\cdot Du(s,y)  \Big) \eta_{\tau,\xi'}(s,y)\notag\\
&&\quad +\I_{s>\tau_0} \tilde p_\alpha^{(\tau,\tilde\xi')}(t,s,x',y)\Big(\big(F(s,y)-F(s,\theta_{s,\tau}(\tilde \xi'))\big) \cdot D u(s,y) \Big)\eta_{\tau,\tilde \xi'}(s,y)  \Bigg),
\label{INTEGRATED_DIFF_BXI}
\end{eqnarray}
where $ v_{\tau,\xi'}(t,x')= (u  \eta_{\tau, \xi'})(t,x')$ and 
 we recall from \eqref{GREEN_FLESSIBILE_IN_TEMPO} that:
\begin{eqnarray*}
\forall 0\le v<r\le T,\;\tilde G_{r,v,\alpha}^{(\tau,\xi')} f(t,x)&:=&\int_v^{r} ds\int_{\R^{d}} dy \tilde p_\alpha^{(\tau,\xi')}(t,s,x',y)f(s,y). 
\end{eqnarray*}
Let now $x,x'\in \R^d $ be s.t. $c_0|x-x'|^\alpha\le  T-t $. Then we can differentiate the previous expression for suitable freezing parameters. Namely:
\begin{eqnarray}
&&\big( D_x v_{\tau,\xi'}(t,x')\big)|_{(\tau,\xi')=(t,x)}\notag\\
&=&\big(D_x\tilde P_{T,t,\alpha}^{(\tau,\tilde \xi')}(g\eta_{\tau,\tilde \xi'}(T,\cdot))(x')\big)_{(\tau,\tilde \xi')=(t,x)}+\big(D\tilde G_{\tau_0,t,\alpha}^{(\tau,\xi')} (f\eta_{\tau,\xi'}-{\mathcal S}_{\tau,\xi'} )(t,x')\big)_{(\tau_0,\tau,\xi')=(t_0,t,x')}\notag\\
&&+\big(D_x\tilde G_{T,\tau_0,\alpha}^{(\tau,\tilde \xi')} (f\eta_{\tau,\tilde \xi'} -{\mathcal S}_{\tau,\tilde \xi'})(t,x')\big)_{(\tau_0,\tau,\tilde \xi')=(t_0,t,x)}\notag\\
&&+\big(D_x\tilde P_{\tau_0,t,\alpha}^{(\tau,\xi')} v_{\tau,\xi'}(\tau_0, x')\big)_{(\tau_0,\tau,\xi')=(t_0,t,x')}-\big(D\tilde P_{\tau_0,t,\alpha}^{(\tau,\tilde \xi')} v_{\tau,\tilde \xi'}(\tau_0, x')\big)_{(\tau_0,\tau,\tilde \xi')=(t_0,t,x)}\notag\\
&&+\int_t^T ds \int_{\R^{d}}dy\Bigg(\I_{s\le \tau_0}D\tilde p_\alpha^{(\tau,\xi')}(t,s,x',y)\Big( \big(F(s,y)-F(s,\theta_{s,\tau}(\xi'))\big)\cdot Du(s,y)  \Big) \eta_{\tau,\xi'}(s,y)\notag\\
&&\quad +\I_{s>\tau_0} D\tilde p_\alpha^{(\tau,\tilde\xi')}(t,s,x',y)\Big(\big(F(s,y)-F(s,\theta_{s,\tau}(\tilde \xi'))\big) \cdot D u(s,y) \Big)\eta_{\tau,\tilde \xi'}(s,y)  \Bigg)_{(\tau_0,\tau,\xi',\tilde \xi')=(t_0,t,x',x)},\notag
\label{INTEGRATED_DIFF_BXI_GRAD}
\end{eqnarray}
where  $t_0=t+c_0|x-x'|^\alpha$ as in \eqref{def_t0}.
\end{PROP}
The previous proposition thus emphasizes that changing the freezing point according to the current (local) diagonal or off-diagonal regime can actually been done  up to an additional discontinuity term. 
\begin{proof}
Restarting from \eqref{DUAHM_1} we can indeed rewrite for given $(t,x')\in [0,T]\times \R^{d} $
 and any $ r\in (t,T], (\tau,\xi')\in [0,T]\times \R^{d}$:
\begin{eqnarray}
v_{\tau,\xi'}(t,x')&=&\tilde P_{r,t,\alpha}^{(\tau,\xi')} v_{\tau,\xi'}(r, x)+ \tilde G_{r,t,\alpha}^{(\tau,\xi')} (f\eta_{\tau,\xi'}-{\mathcal S}_{\tau,\xi'})(t,x')\notag\\
&&+\int_t^{r} ds \int_{\R^{d}} dy\tilde p_\alpha^{(\tau,\xi')}(t,s,x',y)\big(F(s,y)-F(s,\theta_{s,\tau}(\xi))\big)\cdot D u(s,y) \eta_{\tau,\xi'}(s,y). 
 \label{GREEN_SUR_SEGMENT_EN_TEMPS}
\end{eqnarray}
According to Proposition 5, we obtain that for a.e. $r\in (t,T]$  for any $\xi'\in \R^{nd} $:
\begin{eqnarray}
\label{DER_DUHAMEL}
0=\partial_r [\tilde P_{r,t,\alpha}^{(\tau,\xi')} v_{\tau,\xi'}(r, x')]+\int_{\R^{d} }dy\tilde p_\alpha^{(\tau,\xi')}(t,r,x',y)\big(f\eta_{\tau,\xi'}-{\mathcal S}_{\tau,\xi'}\big)(r,y)\notag\\
+\int_{\R^{d}} dy \tilde p_\alpha^{(\tau,\xi')}(t,r,x',y)\Big(\big(F(r,y)-F(r,\theta_{r,\tau}(\xi))\big)\cdot D u(r,y)\Big) \eta_{\tau,\xi'}(r,y) .
\end{eqnarray}
Integrating \eqref{DER_DUHAMEL} with respect to $r$ between $t $ and $t_0\in (t,T]$ for a first given $\xi' $ and between $t_0 $ and $T$ with a possibly different $\tilde \xi'  $ yields:
\begin{eqnarray}
0&=&\tilde P_{t_0,t,\alpha}^{(\tau,\xi')} v_{\tau,\xi'}(t_0, x')-v_{\tau,\xi'}(t,x')+\int_{t}^{t_0}ds \int_{\R^{d}}dy \tilde p_\alpha^{(\tau,\xi')}(t,s,x',y)(f\eta_{\tau,\xi'}-{\mathcal S}_{\tau,\xi'})(s,y)\notag\\
&&+\int_t^{t_0} ds \int_{\R^{d}}dy \tilde p_\alpha^{(\tau,\xi')}(t,s,x',y)\Big(\big(F(s,y)-F(s,\theta_{s,\tau}(\xi'))\big)\cdot D u(s,y)\Big) \eta_{\tau,\xi'}(s,y)\notag\\
&&+\tilde P_{T,t,\alpha}^{(\tau,\tilde \xi')} v_{\tau,\tilde \xi'}(T, x')-\tilde P_{t_0,t,\alpha}^{(\tau,\tilde \xi')}v_{\tau,\tilde \xi'}(t_0,x')+\int_{t_0}^Tds \int_{\R^{d}}dy \tilde p_\alpha^{(\tau,\tilde \xi')}(t,s,x',y)(f\eta_{\tau,\tilde \xi'} -{\mathcal S}_{\tau,\tilde \xi'})(s,y)\notag\\
&&+\int_{t_0}^T ds \int_{\R^{d}}dy \tilde p_\alpha^{(\tau,\tilde\xi')}(t,s,x',y)\Big(\big(F(s,y)-F(s,\theta_{s,\tau}(\tilde \xi'))\big)\cdot D u(s,y)\Big)\eta_{\tau,\tilde \xi'}(s,y)\notag.
\end{eqnarray}
Since $v_{\tau,\tilde \xi'}(T,x')=(g\eta_{\tau, \tilde \xi'}(T,\cdot))(x') $ (terminal condition), using the notations of \eqref{GREEN_SUR_SEGMENT_EN_TEMPS}, the above equation rewrites:
\begin{eqnarray}
v_{\tau,\xi'}(t,x')&=&\tilde P_{T,t,\alpha}^{(\tau,\tilde \xi')}(g\eta_{\tau,\tilde \xi'}(T,\cdot) )(x')+\tilde G_{t_0,t,\alpha}^{(\tau,\xi')} (f\eta_{\tau,\xi'}-{\mathcal S}_{\tau,\xi'} )(t,x')+\tilde G_{T,t_0,\alpha}^{(\tau,\tilde \xi')} (f\eta_{\tau,\tilde \xi'}-{\mathcal S}_{\tau,\tilde \xi'})(t,x')\notag\\
&&+\tilde P_{t_0,t,\alpha}^{(\tau,\xi')} v_{\tau,\xi'}(t_0, x')-\tilde P_{t_0,t,\alpha}^{(\tau,\tilde \xi')} v_{\tau,\tilde \xi'}(t_0, x')\notag\\
&&+\int_t^T ds \int_{\R^{d}}dy\Bigg(\I_{s\le t_0}\tilde p_\alpha^{(\tau,\xi')}(t,s,x',y)\Big(\big(F(s,y)-F(s,\theta_{s,\tau}(\xi'))\big)\cdot D u(s,y)\Big) \eta_{\tau,\xi'}(s,y)\notag\\
&&\quad +\I_{s>t_0} \tilde p_\alpha^{(\tau,\tilde\xi')}(t,s,x',y)\Big(\big(F(s,y)-F(s,\theta_{s,\tau}(\tilde \xi'))\big) \cdot D u(s,y)  \Big)\eta_{\tau,\tilde \xi'}(s,y) \Bigg).
\label{INTEGRATED_DIFF_BXI_BIS}
\end{eqnarray}
This gives \eqref{INTEGRATED_DIFF_BXI_BIS}.
 Expression \eqref{INTEGRATED_DIFF_BXI} can then, for the indicated freezing parameters, be differentiated in space reproducing the arguments used to derive \eqref{FINAL_BD_FOR_GRADIENT} and noting that when $s>\tau_0>t $ the bounded convergence theorem readily applies for the last contribution. This gives \eqref{INTEGRATED_DIFF_BXI_GRAD}.  
\end{proof}


We can from \eqref{INTEGRATED_DIFF_BXI} express the full expression of the difference to be investigated. Namely, for all $(t,x,x')\in [0,T]\times (\R^d)^2 $, recalling that for $(\tau,\xi,\xi')=(t,x,x') $, $\eta_{\tau,\xi}(t,x)=\eta_{\tau,\xi'}(t,x')=1 $ we get:
\begin{eqnarray}
&&|D_xu(t,x')-D_x u(t,x)|\notag\\
&=&\Big|D_x (u\eta_{\tau,\xi'}) (t,x')\big|_{(\tau,\xi')=(t,\xi')}-D_x (u\eta_{\tau,\xi})(t,x)\big|_{(\tau,\xi)=(t,x)} \Big| \notag\\&\le& \Big|\Big[D_x\tilde P_{T,t,\alpha}^{(\tau,\tilde \xi')}\big(\eta_{\tau,\tilde \xi'}(T,\cdot)g\big)(x')-D_x\tilde P_{T,t,\alpha}^{(\tau, \xi)}\big(\eta_{\tau, \xi}(T,\cdot) g\big)(x)\Big]\Big|\Big|_{(\tau,\tilde \xi',\xi)=(t,x,x)}\notag\\
&&+\big|D_x\tilde P_{\tau_0,t,\alpha}^{(\tau,\xi')} v_{\tau,\xi'}(\tau_0, x')-D_x\tilde P_{\tau_0,t,\alpha}^{(\tau,\tilde \xi')} v_{\tau,\tilde \xi'}(\tau_0, x')\big| \Big|_{(\tau_0,\tau,\xi',\tilde \xi')=(t_0,t,x',x)}\notag\\
&&+|D_x\tilde G_{\tau_0,t,\alpha}^{(\tau,\xi')}\big( f\eta_{\tau,\xi'} -{\mathcal S}_{\tau,\xi'} \big)(t,x')+D_x\tilde G_{T,\tau_0,\alpha}^{(\tau,\tilde \xi')} \big( f\eta_{\tau,\tilde \xi'}-{\mathcal S}_{\tau,\tilde \xi'}\big)(t,x') \notag\\
&&-D_x\tilde G_{T,t,\alpha}^{(\tau, \xi)} (f\eta_{\tau,\xi}-{\mathcal S}_{\tau,\xi})(t,x)| \Big|_{(\tau_0,\tau,\xi',\tilde \xi',\xi)=(t_0,t,x',x,x)}\notag\\
&&+\Big| \int_t^T ds \int_{\R^{d}}dy\Bigg(\I_{s\le \tau_0} D_x\tilde p_\alpha^{(\tau,\xi')}(t,s,x',y)\Big(\big(F(s,y)-F(s,\theta_{s,\tau}(\xi'))\big)\cdot Du(s,y)\Big)\eta_{\tau,\xi'}(s,y)\notag\\
&&\quad +\I_{s>\tau_0} D_x \tilde p_\alpha^{(\tau,\tilde\xi')}(t,s,x',y)\Big(\big(F(s,y)-F(s,\theta_{s,\tau}(\tilde \xi'))\big)\cdot Du(s,y)\Big)\eta_{\tau,\tilde \xi'}(s,y)\notag\\
&&- D_x \tilde p_\alpha^{(\tau,\xi)}(t,s,x,y)\Big(\big(F(s,y)-F(s,\theta_{s,\tau}( \xi))\big)  \cdot D u(s,y)\Big) \eta_{\tau,\xi}(s,y) \Bigg)_{(\tau_0,\tau,\xi',\tilde \xi',\xi)=(t_0,t,x',x,x)}\Big|.\notag
\label{FULL_DIFF}
\end{eqnarray}
Since we have initially assumed $T$ \textit{small} and we are currently in the diagonal-regime, $|x-x'|\le [(T-t)/c_0]^{\frac 1\alpha}$,  we may assume $|x-x'| \le 1$ provided $c_0$ is small enough. In such case, we obtain the following estimate.

\begin{lem}
There exists $C\ge 1$ s.t. for all $(t,x,x') \in [0,T] \times \R^{d} \times \R^{d}$, s.t. $c_0|x-x'|^\alpha\le T-t $,
one has for $g\in C^{\beta+\alpha}(\R^d,\R)$:
\begin{eqnarray*}
&&|D_xu(t,x')-D_x u(t,x)|\\
&\le& |x-x'|^{\alpha+\beta-1}\Bigg\{ C\bigg(\big([ Dg ]_{C^{\beta+\alpha-1}_{b}}+1\big) + \|g\|_{C_b^{\alpha+\beta}}+T^{\frac{\alpha+\beta-1}\alpha}\|f\|_{L^\infty([0,T],C_b^\beta)}\\
&& +c_0^{\frac{\alpha+\beta-1}\alpha} ([D u]_{\beta+\alpha-1, T}+\|u\|_\infty )(1+{K_0})+ {K_0} (c_0^{1+\frac{\beta-2}{\alpha}}+ c_0^{\frac {\alpha+\beta-1}{\alpha}})\|Du\|_{L^\infty } \bigg) +\frac 14\|u\|_{L^\infty([0,T],C_b^{\alpha+\beta})}\Bigg\}
\end{eqnarray*}
\end{lem}
The proof of Lemma \ref{LEMME_HOLDER_GRAD} then follows from above Lemma and control \eqref{CTR_HOLDER_HD_LEMMA9}. The remainder of this part is dedicated to the proof of the above estimate.

\paragraph{Controls of the frozen semi-group.}
Note that the freezing points $\tilde \xi'=\tilde \xi=x $. Hence, the first contribution in the r.h.s. of \eqref{FULL_DIFF} can be handled similarly to the proof of  Proposition \ref{SCHAUDER_FROZEN} (see equations \eqref{BD_SG_G_OFF_DIAG}-\eqref{CTR_BD_HOLDER_SG_G}).
Namely, we get the following lemma.
\begin{lem}\label{lem_Holder_tildePg}
There exists $C\ge 1$ s.t. for all $(t,x,x') \in [0,T] \times \R^{d} \times \R^{d}$, s.t. $c_0|x-x'|^\alpha\le T-t $,
one has for $g\in C^{\beta+\alpha}(\R^d,\R)$:
\begin{eqnarray*}
&&\big | D_{x} \tilde P_{T,t,\alpha}^{(\tau,\xi)}\big(g\eta_{\tau,\xi}(T,\cdot)\big)(x)-D_{x} \tilde P_{T,t,\alpha}^{(\tau,\tilde \xi')}\big(g\eta_{\tau,\tilde \xi'}(T,\cdot)\big)(x') \big | \Big|_{(\tau,\tilde \xi',\xi)=(t,x,x)}\\
&= &\big | D_{x} \tilde P_{T,t,\alpha}^{(\tau,\xi)}\big(g\eta_{\tau,\xi}(T,\cdot)\big)(x)-D_{x} \tilde P_{T,t,\alpha}^{(\tau, \xi)}\big(g\eta_{\tau, \xi}(T,\cdot)\big)(x') \big | \Big|_{(\tau,\xi)=(t,x)}\\
&\leq& C \big([ D(g \eta_{\tau,\xi})(T,\cdot)]_{C^{\beta+\alpha-1}_{b}}\big)\Big|_{(\tau,\xi)=(t,x)} |x-x'|^{\alpha+\beta-1}\le C \big([ Dg ]_{C^{\beta+\alpha-1}_{b}}+1\big) |x-x'|^{\alpha+\beta-1}.
\end{eqnarray*}
\end{lem}

\paragraph{Smoothing effects associated with the Green kernel. 
}
Let us recall that in the proof of the Schauder estimates for the \textit{frozen} operator (Proposition \ref{SCHAUDER_FROZEN}), to control in the \textit{global} diagonal-regime the H\"older norms of the Green kernel, we split into two parts 
the time integrals according to the position of the time integration variable w.r.t. the change of regime time $t_0$ (see \eqref{def_t0})  \textit{a posteriori} chosen to be 
$
t_0:=t+ c_ 0|x-x'|^\alpha.
$ 
This is again the splitting according to the (now local) off-diagonal and diagonal regime. 

\begin{lem}\label{lemme_Holder_Gg} Under \A{A} and for $T$ small enough, there exists a constant $C:=C(\A{A},T)$ s.t for fixed  points $(t,x,x')\in [0,T]\times (\R^{d})^2 $, s.t. $c_0|x-x'|^\alpha\le T-t $ and for all $f\in L^\infty\big([0,T], C^\beta(\R^{d},\R) \big)$:
\begin{eqnarray} 
&&\Big ( |D_{x} \tilde G_{\tau_0,t,\alpha}^{(\tau,\xi)} \big(f\eta_{\tau,\xi}-{\mathcal S}_{\tau,\xi}\big) (t,x)- D_{x}\tilde G_{\tau_0,t,\alpha}^{(\tau,\xi')}\big( f \eta_{\tau,\xi'}-{\mathcal S}_{\tau,\xi'}\big)(t,x')|\Big|_{(\tau_0,\tau,\xi,\xi')=(t_0,t,x,x')}\notag\\
&&+|D_{x} \tilde G_{T,\tau_0,\alpha}^{(\tau, \xi)} \big(f\eta_{\tau,\xi}-{\mathcal S}_{\tau,\xi}\big)(t,x)- D_{x}\tilde G_{T,\tau_0,\alpha}^{(\tau,\tilde \xi')} \big(f\eta_{\tau,\tilde \xi'}-{\mathcal S}_{\tau,\tilde \xi'}\big)(t,x') | \Big|_{(\tau_0,\tau,\xi,\tilde \xi')=(t_0,t,x,x)} \Big)\notag\\
&
\leq & \Big( C \big(\|g\|_{C_b^{\alpha+\beta}}+T^{\frac{\alpha+\beta-1}\alpha}\|f\|_{L^\infty([0,T],C_b^\beta)}\big)+\frac 14\|u\|_{L^\infty([0,T],C_b^{\alpha+\beta})}\Big)|x-x'|^{\alpha+\beta-1}.
\end{eqnarray}
\end{lem}
\begin{proof}[Proof of Lemma \ref{lemme_Holder_Gg}]
Let us first control the H\"older moduli of the arguments of the Green kernel. Observe that:
\begin{equation}\label{HOLD_MOD_WITH_CUT_OFF}
[f\eta_{\tau,\xi}]_{\beta,T} \le [f]_{\beta,T}+C\|f\|_\infty\le C \|f\|_{L^\infty([0,T],C_b^\beta)}.
\end{equation}
On the other hand, from \eqref{CTR_HOLDE_MODULI} and Lemma \ref{LEMMA_SUP_BOUNDS_SOL_AND_GRAD} we derive that for any $\varepsilon\in (0,1) $:
\begin{eqnarray*}
[{\mathcal S}_{\tau,\xi}]_{\beta,T}&\le& C_{\beta,\alpha}\big( (1+\varepsilon^{-1})(\|u\|_\infty +\|Du\|_\infty )+\varepsilon \|u\|_{L^\infty([0,T],C_b^{\alpha+\beta})}\big),\notag\\
&\le &  C_{\beta,\alpha}\big( (1+\varepsilon^{-1})(\|g\|_\infty+T\|f\|_\infty +\|Du\|_\infty(1+{K_{0}}) )+\varepsilon \|u\|_{L^\infty([0,T],C_b^{\alpha+\beta})}\big)\notag\\
&\le & C_{\beta,\alpha}(1+{K_{0}})\Big( (1+\varepsilon^{-1})\big(\|g\|_{\alpha+\beta}+T^{\frac{\alpha+\beta-1}{\alpha}}\|f\|_{L^\infty([0,T],C_b^\beta)}  \big)\notag\\
&&+\big( (1+\varepsilon^{-1})T^{\frac{\alpha+\beta-1}{\alpha}}(1+{K_{0}})+\varepsilon \big) \|u\|_{L^\infty([0,T],C_b^{\alpha+\beta})}\Big). 
\end{eqnarray*}
Now, for $\varepsilon $ small enough and $T$ small enough w.r.t. $\varepsilon $ it is clear from the above equation that 
\begin{eqnarray}
[{\mathcal S}_{\tau,\xi}]_{\beta,T}&\le & C_{\beta,\alpha}(1+{K_{0}})\Big( (1+\varepsilon^{-1})\big(\|g\|_{\alpha+\beta}+T^{\frac{\alpha+\beta-1}{\alpha}}\|f\|_{L^\infty([0,T],C_b^\beta)}  \big)\Big)
+\frac{c_0^{\frac{2-(\alpha+\beta)}\alpha}}8 \|u\|_{L^\infty([0,T],C_b^{\alpha+\beta})}. \notag\\
\label{CTR_HOLDE_MODULI_RAFFINATI}
\end{eqnarray}

Let us then consider for the difference of the Green kernels the \textit{off-diagonal regime}.  We readily get from Lemma \ref{sd} (with $\tau =t$) and the above equations \eqref{HOLD_MOD_WITH_CUT_OFF} and \eqref{CTR_HOLDE_MODULI_RAFFINATI} that
\begin{eqnarray} 
&&\big | D_{x} \tilde G_{\tau_0,t,\alpha}^{(\tau,\xi)}\big( f \eta_{\tau,\xi}-{\mathcal S}_{\tau,\xi}\big)(t,x) \!-\! D_{x} \tilde G_{\tau_0,t,\alpha}^{(\tau,\xi')} (f\eta_{\tau,\xi'}-{\mathcal S}_{\tau,\xi'} )(t,x') \big | \Big|_{(\tau_0,\tau,\xi,\xi')=(t_0,t,x,x')} 
\nonumber \\
& \le&
\Big |\int_t^{t_0} \!\!ds D_{x} \tilde P_{s,t,\alpha}^{(\tau,\xi)}(f\eta_{\tau,\xi}-{\mathcal S}_{\tau,\xi} )(s,x)\Big | \bigg|_{(\tau,\xi)=(t,x)}
\!\!\!+\!\Big |\int_t^{t_0} \!\!ds D_{x} \tilde P_{s,t,\alpha}^{(\tau,\xi')}(f\eta_{\tau,\xi'}-{\mathcal S}_{\tau,\xi'} )(s,x')\Big | \bigg|_{\xi'=x'}
\nonumber \\ 
&\!\!\!\leq& C ([f\eta_{\tau,\xi}]_{ \beta,T}+[{\mathcal S}_{\tau,\xi}]_{\beta,T}+[f\eta_{\tau,\xi'}]_{ \beta,T}+[{\mathcal S}_{\tau,\xi'}]_{\beta,T})\big|_{(\tau,\xi,\xi')=(t,x,x')}
\int_t^{t_0} ds (s-t)^{-\frac{1}{\alpha}+ \frac{\beta}{\alpha}}\notag\\
&\le&  \Big( C \big(\|g\|_{C_b^{\alpha+\beta}}+T^{\frac{\alpha+\beta-1}\alpha}\|f\|_{L^\infty([0,T],C_b^\beta)}\big)+\frac 18\|u\|_{L^\infty([0,T],C_b^{\alpha+\beta})}\Big)|x-x'|^{\alpha+\beta-1},\notag\\
\label{THE_GREEN_WITH_CUT_OFF_AND_COMMUT}
\end{eqnarray}
recalling that $c_0 \le 1$ for the last inequality.

For the \textit{diagonal regime}, we proceed as in the proof of Lemma \ref{SCHAUDER_FROZEN} (see equation \eqref{CTR_D2}) using again the previous controls \eqref{HOLD_MOD_WITH_CUT_OFF} and \eqref{CTR_HOLDE_MODULI_RAFFINATI} to bound $[f\eta_{\tau,\xi}-{\mathcal S}_{\tau,\xi}]_{\beta,T} $. We obtain a control similar to \eqref{THE_GREEN_WITH_CUT_OFF_AND_COMMUT} for:
 $$|D_{x} \tilde G_{T,\tau_0,\alpha}^{(\tau, \xi)} \big(f\eta_{\tau,\xi}-{\mathcal S}_{\tau,\xi}\big)(t,x)- D_{x}\tilde G_{T,\tau_0,\alpha}^{(\tau,\tilde \xi')} \big(f\eta_{\tau,\tilde \xi'}-{\mathcal S}_{\tau,\tilde \xi'}\big)(t,x') | \Big|_{(\tau_0,\tau,\xi,\tilde \xi')=(t_0,t,x,x)},$$ which then gives the statement. 
\end{proof}

\paragraph{Smoothing effects associated with the discontinuity term.}
\label{SEC_DISC}
It now remains to control the contribution arising from the change of freezing point in equation \eqref{FULL_DIFF}. 
The main result of this section is the next lemma.
\begin{lem}[Control of the discontinuity terms]
\label{CTR_TERME_DISC} 
There exists $C:=C(\A{A})$ s.t. for all $(t,x,x') \in [0,T] \times \R^{d}\times \R^{d}$, s.t. $c_0|x-x'|^\alpha\le T-t $, for $t_0=\big(t+ c_ 0|x-x'|^\alpha\big)\le  T$ as in \eqref{def_t0}, 
\begin{eqnarray*}
 &&\big|D_{x}\tilde P_{\tau_0,t,\alpha}^{(\tau,\xi')} v_{\tau,\xi'}(\tau_0, x')- D_{x} \tilde P_{\tau_0,t,\alpha}^{(\tau,\tilde  \xi')} v_{\tau,\tilde \xi'}(\tau_0, x') \big |_{(\tau_0,\xi',\tilde \xi')=(t_0, x',x)} \\
 &\leq&  C c_0^{\frac{\alpha+\beta-1}\alpha} 
 ([D u]_{\beta+\alpha-1, T}+\|u\|_\infty )(1+{K_0}) |x-x'|^{\alpha+\beta-1}.
\end{eqnarray*}
\end{lem}

We can write:
\begin{eqnarray*}
&&\Big(D_{x} \tilde P_{\tau_0,t,\alpha}^{(\tau,\xi')} v_{\tau,\xi'}(\tau_0, x')- D_{x} \tilde P_{\tau_0,t,\alpha}^{(\tau,\tilde \xi')} v_{\tau,\xi'}(t_0, x')\Big)_{(\tau_0,\tau,\xi',\tilde \xi')=(t_0,t,x',x)}  \nonumber \\
&=&\Bigg [ \int_{\R^{d}}  \tilde p_\alpha^{(\tau,\xi')}(t,\tau_0,x',y) Dv_{\tau,\xi'}(\tau_0,y) dy
 -\int_{\R^{d}}  \tilde p_\alpha^{(\tau,\tilde \xi')}(t,\tau_0,x',y) Dv_{\tau,\tilde \xi'}(\tau_0,y) dy \Bigg ]_{(\tau_0,\tau,\xi',\tilde \xi')=(t_0,t,x',x)}\notag\\
&=&
\Bigg [ \int_{\R^{d}}  \tilde p_\alpha^{(\tau,\xi')}(t,\tau_0,x',y) [Dv_{\tau,\xi'}(t_0,y)-Du(\tau_0,m_{\tau_0,t}^{(\tau,\xi')}(x'))] dy
\nonumber \\
&&\quad -\int_{\R^{d}}  \tilde p_\alpha^{(\tau,\tilde \xi')}(t,\tau_0,x',y) [Dv_{\tau,\tilde \xi'}(\tau_0,y)-Dv_{\tau,\tilde\xi'}(\tau_0,m_{\tau_0,t}^{(\tau,\tilde \xi')}(x'))] dy \Bigg ]_{(\tau_0,\tau,\xi',\tilde \xi')=(t_0,t,x',x)}\notag\\
&&+|Dv_{\tau,\tilde  \xi'}(\tau_0,m_{\tau_0,t}^{(\tau,\tilde \xi')}(x'))-Dv_{\tau,\xi'}(\tau_0,m_{\tau_0,t}^{(\tau,\xi')}(x')|_{(\tau_0,\tau,\xi',\tilde \xi')=(t_0,t,x',x)}.
\label{control_hoped_mathbf_R_ineq_trig_BIS}
\end{eqnarray*}
Exploiting now the regularity of $Du$ and the integrability property 
  \A{NDa} 
we then derive:
\begin{eqnarray}
&&|D_{x} \tilde P_{\tau_0,t,\alpha}^{(\tau,\xi')} v_{\tau,\xi'}(\tau_0, x')- D_{x} \tilde P_{\tau_0,t,\alpha}^{(\tau,\tilde \xi')} v_{\tau,\tilde \xi'}(\tau_0, x')|_{(\tau_0,\tau,\xi',\tilde \xi')=(t_0,t,x',x)}\notag\\
&\le& C \big([Du]_{\alpha+\beta-1, T}+\|u\|_\infty\big)\Big[(t_0-t)^{\frac{\alpha+\beta-1}\alpha}+|m_{t_0,t}^{(t,x)}(x')-m_{t_0,t}^{(t,x')}(x')|^{\alpha+\beta-1}\Big]\notag\\
&\le& C \big([Du]_{\alpha+\beta-1, T}+\|u\|_\infty\big)\Big[c_0^{\frac{\alpha+\beta-1}\alpha}|x-x'|^{\alpha+\beta-1}+|m_{t_0,t}^{(t,x)}(x')-m_{t_0,t}^{(t,x')}(x')|^{\alpha+\beta-1}\Big],
\label{PREAL_FIN}
\end{eqnarray}
recalling from \eqref{def_t0} that $t_0=t+c_0 |x-x'|^\alpha
 $ for the last inequality. It therefore remains to control $m_{t_0,t}^{(t,x)}(x')-m_{t_0,t}^{(t,x')}(x')$. Write:
 \begin{eqnarray}
|m_{t_0,t}^{(t,x)}(x')-m_{t_0,t}^{(t,x')}(x')|&=&|[x'+\int_{t}^{t_0}F(v,\theta_{v,t}(x))dv]-[x'+\int_{t}^{t_0}F(v,\theta_{v,t}(x'))dv]|\notag\\
&\le & {K_0}\int_t^{t_0} dv |\theta_{v,t}(x)-\theta_{v,t}(x')|^\beta\le C {K_0}\, (t_0-t)\big(|x-x'|+(t-t_0)^{\frac 1\alpha}\big)^\beta\notag\\
&\le & C{K_0} c_0 |x-x'|^{\alpha+\beta}\le C {K_0} c_0 |x-x'|, \label{CTR_DIFF_MEAN}
 \end{eqnarray}
 using  Lemma \ref{FLOW_LEMMA} for the second inequality. We have also assumed w.l.o.g. that $|x-x'|\le 1  $ and exploited as well that $\alpha+\beta>1 $. Plugging \eqref{CTR_DIFF_MEAN} into \eqref{PREAL_FIN} yields the statement of the lemma.

\paragraph{Smoothing effects associated with the perturbative term.}
This section is dedicated to the investigation of the spatial H\"older continuity of the perturbative term in \eqref{FULL_DIFF}.  We prove the following Lemma which is  the most difficult part of the proof of Theorem \ref{THEO_SCHAU_ALPHA}.

\begin{lem}\label{HOLDER_CTR_PERT_PART} Under \A{A}, there exists a constant $C:=C(\A{A},T)$  s.t. for fixed  $(t,x,x')\in [0,T]\times (\R^{d})^2 $ s.t. $ c_0|x-x'|^\alpha\le T-t$, we have that:
\begin{eqnarray}
\Big| \int_t^T ds \int_{\R^{d}}dy\Big(\I_{s\le t_0} D_x \tilde p_\alpha^{(\tau,\xi')}(t,s,x',y)\Big(\big(F(s,y)-F(s,\theta_{s,\tau}(\xi'))\big)\cdot Du(s,y)\Big)\eta_{\tau,\xi'}(s,y)\notag\\
\quad +\I_{s>t_0} D_x \tilde p_\alpha^{(\tau,\tilde\xi')}(t,s,x',y)\Big(\big(F(s,y)-F(s,\theta_{s,\tau}(\tilde \xi'))\big) \cdot Du(s,y)\Big)\eta_{\tau,\tilde \xi'}(s,y)
\notag\\
- D_x \tilde p_\alpha^{(\tau,\xi)}(t,s,x,y)\Big(\big(F(s,y)-F(s,\theta_{s,\tau}( \xi))\big)  \cdot D u(s,y)\Big) \eta_{\tau,\xi}(s,y)\Big|_{(\tau_0,\tau,\xi,\xi',\tilde \xi')=(t_0,t,x,x',x)}\notag\\
 \leq  C{K_0} (c_0^{1+\frac{\beta-2}{\alpha}}+ c_0^{\frac {\alpha+\beta-1}{\alpha}})\|Du\|_{L^\infty }  |x-x'|^{\alpha+\beta-1}.\label{esti_holder_deriv_sec_perturpart}
\end{eqnarray}
\end{lem}

As already used for  the semi-group and the Green kernel, we split the investigations into two parts: the first one is done when the system is in the \emph{local off-diagonal} regime w.r.t. the current integration time $s$ (i.e. for $s \leq t_0$) and the other one when the system is in the  \emph{local diagonal} regime (i.e., for $s > t_0$). We also recall that the critical time giving the change of regime is \textcolor{black}{(chosen after potential differentiation)} $t_0=t+c_0|x-x'|^\alpha\le  T$ (in the current global diagonal regime).\\ 

$\bullet$  Control of \eqref{esti_holder_deriv_sec_perturpart}: {\it  Local Off-Diagonal case.} 
 Arguing as in  Lemma \ref{lemme_Holder_Gg} we write for the \textit{local off-diagonal} regime:
 \begin{eqnarray}
 \label{SPLIT_LAMBDA_DIAG}
&&|D_{x}\Delta_{{\rm \textbf{off-diag}}}^{\xi,\xi'}(t,x,x')|\big|_{(\xi,\xi')=(x,x')} \notag\\
&:=&\notag\Big| \int_t^{t_0} ds \,  \int_{\R^{d}} \Big(D_x \tilde p_\alpha^{(\tau,\xi')}(t,s,x',y)\Big(\big(F(s,y)-F(s,\theta_{s,\tau}( \xi'))\big)\cdot Du(s,y) \Big) \eta_{\tau,\xi'}(s,y)\notag\\
&&
- D_x p_\alpha^{(\tau,\xi)}(t,s,x,y)\Big(\big(F(s,y)-F(s,\theta_{s,\tau}( \xi))\big)  \cdot D u(s,y) \Big) \eta_{\tau,\xi}(s,y)\Big|_{(\tau,\xi,\xi')=(t,x,x')}\notag\\
 &\leq & C{K_0}
  \| Du \|_{\infty}
 \Bigg[\Big |\int_t^{t_0} \!\! ds \int_{\R^{d}}\!\!\! dy
|D_{x}  \tilde p_\alpha^{(\tau,\xi)} (t,s,x,y)| |y-\theta_{s,t}(\xi)|^\beta  \Big | \bigg |_{(\tau,\xi)=(t,x)}
\notag\\
&&+ \Big | \int_t^{t_0} \!\! ds \int_{\R^{d}}\!\!\! dy
  |D_{x} \tilde p_\alpha^{(\tau,\xi ')} (t,s,x',y)| |y-\theta_{s,t}(\xi')|^\beta   \Big | \bigg |_{(\tau,\xi')=(t,x')}\Bigg].
\end{eqnarray}
We readily get since  $t_0=t+c_0|x-x'|^\alpha\le T$ using the integrability property $(\mathscr P_\beta) $: 
\begin{eqnarray}
|D_{x} \Delta_{{\rm \textbf{off-diag}}}^{\xi,\xi'}(t,x,x')|\big|_{(\xi,\xi')=(x,x')}&\leq&
C{K_0} \int_t^{t_0} \frac{ds}{(s-t)^{\frac 1\alpha-\frac \beta\alpha}} \|Du\|_{L^\infty}\notag\\
&\leq& C{K_0} \|Du\|_{L^\infty} c_0^{\frac{\alpha+\beta-1}{\alpha}}|x-x'|^{\alpha+\beta-1} .\label{LE_LAB_MANQUANT}
 \end{eqnarray}

$\bullet$  Control of \eqref{esti_holder_deriv_sec_perturpart}: {\it Local Diagonal case.} 
Let us now turn to the control of 
\begin{eqnarray}
&&|D_{x} \Delta_{{\rm \textbf{diag}}}^{\xi,\tilde \xi'}(t,T,x,x')|\Big|_{(\xi,\tilde \xi')=(x,x)}
 \label{CTR_PERT_DIAG} \\
&:=&\notag\Big| \int_{t_0}^Tds \int_{\R^{d}} dy \Big(D_x\tilde p_\alpha^{(\tau,\tilde \xi')}(t,s,x',y)\Big(\big(F(s,y)-F(s,\theta_{s,t}(\tilde \xi'))\big)\cdot Du(s,y) \Big)\eta_{\tau,\tilde \xi'}(s,y)\notag\\
&&
-D_x\tilde p_\alpha^{(\tau,\xi)}(t,s,x,y)\Big(\big(F(s,y)-F(s,\theta_{s,t}( \xi))\big)  \cdot D u(s,y)\Big) \eta_{\tau,\xi}(s,y)\Big|_{(\tau, \xi,\tilde \xi')=(t,x,x)}\notag\\
&:=&\notag\Big| \int_{t_0}^T ds \int_{\R^{d}} dy\Big(D_x\tilde p_\alpha^{(\tau,\xi)}(t,s,x',y)-D_x\tilde p_\alpha^{(\tau,\xi)}(t,s,x,y)\Big)\notag\\
&& \times \Big(\big(F(s,y)-F(s,\theta_{s,t}( \xi))\big)\cdot D u(s,y)\Big)\eta_{\tau,\xi}(s,y)\Big|_{ (\tau,\xi)=(t,x)}\notag\\
&\leq& 
\| Du \|_{\infty} |x-x'|  \int_{t_0}^T ds\int_0^1 d\mu\int_{\R^{d}}dy \big| D^2   p_\alpha (s-t, y - \theta_{s,t}(x) +\mu (x-x'))| \, |F(s,y)-F(s,\theta_{s,t}( x)) |\eta_{t,x}(s,y),
 \notag
 \end{eqnarray}
recalling \eqref{CORRESP_SHIFTED} for the last inequality.
We  finally get:
\begin{eqnarray*}
&&|D_{x} \Delta_{{\rm \textbf{diag}}}^{\xi,\tilde \xi'}(t,T,x,x')|\Big|_{(\xi,\tilde \xi')=(x,x)}\\
&\le& \| Du \|_{\infty}{K_{0}} |x-x'|  \int_{t_0}^T ds\int_0^1 d\mu\int_{\R^{d}}dy \big| D^2   p_\alpha (s-t, y - \theta_{s,t}(x) +\mu (x-x'))| \, |y-\theta_{s,t}( x) |^\beta\I_{|y-\theta_{s,t}( x) |\le 2}\\
&\le& \| Du \|_{\infty}{K_{0}} |x-x'|  \int_{t_0}^T ds\int_0^1 d\mu\int_{\R^{d}}dy \big| D^2   p_\alpha (s-t, y  +\mu (x-x'))| \, |y |^\beta
\\
&\le& \| Du \|_{\infty}{K_{0}} |x-x'|  \int_{t_0}^T ds\int_0^1 d\mu \Big[ \int_{\R^{d}}dy \big| D^2   p_\alpha (s-t, y  +\mu (x-x'))| \, |y +\mu (x-x') |^\beta 
\\ &+ &
\int_{\R^{d}}dy \big| D^2   p_\alpha (s-t, y  +\mu (x-x'))| \, |\mu (x-x') |^\beta  \Big]
\end{eqnarray*}

Now   since   $ s \in [t_0,T]$  and $  t_0=t+c_0|x-x'|^\alpha$ we have $s-t \ge c _0|x-x'|^\alpha$ and so 
\begin{equation} \label{wee}
 |x-x'| \le  (c_0)^{-1/\alpha} \, (s-t)^{1/\alpha} \le  K(s-t)^{\frac 1\alpha} 
\end{equation}
  if the threshold $K$ is chosen large enough.
   Applying  property $({\mathscr P}_\beta) $ twice with $\gamma = \beta$ and $\gamma =0$ 
we eventually derive: 
\begin{gather} \notag
|D_{x} \Delta_{{\rm \textbf{diag}}}^{\xi,\tilde \xi'}(t,x,x')|\big|_{(\xi,\tilde \xi')=(x,x')}
\leq C {K_{0}}\| D u\|_\infty |x-x'| \int_{t_0}^T \frac{ds}{(s-t)^{\frac2\alpha-\frac \beta\alpha}} 
\\ \notag
\le  C {K_0}\|Du\|_{L^\infty} \, (t_0-t)^{1+\frac{\beta}{\alpha}-\frac{2}{\alpha}}|x-x'|
\\
\le  C{K_0}\|Du\|_{L^\infty} |x-x'|^{\alpha+\beta-1},
\label{FINAL_CONTROL_PERT_DIAG}
\end{gather}
where  $C= C(c_0,   \A{A} ) >0$.
Equations \eqref{LE_LAB_MANQUANT} and \eqref{FINAL_CONTROL_PERT_DIAG} give the statement of the lemma.

\def\b {
Recalling from the definition in \eqref{CORRESP_SHIFTED} that the mapping $z\mapsto m_{s,t}^\xi(z) $ is affine, we get $m_{s,t}^\xi(x'+\mu(x-x'))=m_{s,t}^\xi(x)+(1-\mu)(x'-x) $. We then rewrite:
\begin{eqnarray}
 &&|D_{x} \Delta_{{\rm \textbf{diag}}}^{\xi,\xi'}(t,x,x')|\big|_{(\xi,\xi')=(x,x')}\notag\\
&\leq&C\|F\|_{L^\infty(C^{\beta})}\| D u\|_\infty |x-x'| \int_{t_0}^T \frac{ds}{(s-t)^{\frac2\alpha-\frac \beta\alpha}}\int_0^1 d\mu\notag\\
&&\times \int_{\R^{d}}dy \bar p_\alpha\Big(s-t,y-\big(\theta_{s,t}(x)+(1-\mu)(x'-x)\big)\Big)\left(\frac{|y-\theta_{s,t}(x)|}{(s-t)^{\frac 1\alpha}} \right)^\beta.  \label{CTR_PERT_DIAG}
 \end{eqnarray}
Recalling that for $s\in [t_0,T], |x-x'|/(s-t)^{\frac 1\alpha}\le c_0^{-\frac 1\alpha} $, we now use the following lemma whose proof is postponed to Appendix \ref{APP_TEC}.
 \begin{lem}[Integration of the derivatives in the diagonal regime]\label{INT_DIAG}
 There exists $C:=C(\A{A})$ s.t. for $c_0|x-x'|^\alpha\le (s-t) $, $c_0\le 1$, $0\le t<s\le T $ and all $\mu\in [0,1] $ :
$$ 
\int_{\R^{d}}dy \bar p_\alpha\Big(s-t,y-\big(\theta_{s,t}(x)+(1-\mu)(x'-x)\big)\Big)\left(\frac{|y-\theta_{s,t}(x)|}{(s-t)^{\frac 1\alpha}} \right)^\beta
\le 
C c_0^{-\frac{(d+\beta)}\alpha}.$$
 \end{lem}
Plugging the above bound into \eqref{CTR_PERT_DIAG}
yields:
\begin{eqnarray}
|D_{x} \Delta_{{\rm \textbf{diag}}}^{\xi,\tilde \xi'}(t,T,x,x')|\Big|_{(\xi,\tilde \xi')=(x,x)}
&\le&  C {K_0}\|Du\|_{L^\infty}c_0^{-\frac{(d+\beta)}\alpha}\int_{t_0}^T \frac{ds}{(s-t)^{\frac{2}{\alpha}-\frac{\beta}\alpha}}|x-x'|\notag\\
&\le & C {K_0}\|Du\|_{L^\infty}c_0^{-\frac{(d+\beta)}\alpha}(t_0-t)^{1+\frac{\beta}{\alpha}-\frac{2}{\alpha}}|x-x'|\notag\\
&\le & C{K_0}\|Du\|_{L^\infty} c_0^{1+\frac{\beta-2}{\alpha}-\frac{(d+\beta)}\alpha}|x-x'|^{\alpha+\beta-1},
\label{FINAL_CONTROL_PERT_DIAG}
\end{eqnarray}
recalling from \eqref{def_t0} that $t_0-t=c_0|x-x'|^\alpha $.\\
}



\mysection{Existence result.}\label{ESISTENZA}
We point out here that the classical continuity method, which is direct from the a priori estimate, and which was successfully used in \cite{prio:12} to establish existence in the elliptic setting, does not work  for $\alpha\in (0,1) $. The key point is that when one tries to write:
$$\partial_t u(t,x)+ 
 L_{\alpha} u +
\delta_0 F(t,x)\cdot D u(t,x)=-f(t,x)+(\delta_0-\delta) F(t,x)\cdot Dv (t,x), $$
where $v\in C_b^{\alpha+\beta}(\R^d,\R)$ then, the product $F(t,x)\cdot Dv (t,x)$ has under \A{A} a H\"older-regularity of order $\beta+\alpha-1<\beta $, since $\alpha\in (0,1) $. Therefore, we cannot in this framework readily apply our a priori estimate in a fixed point perspective.\\


 We will proceed through a vanishing viscosity approach, as it was also for instance considered by Silvestre in \cite{silv:12} or by Zhang and Zhao in \cite{zhan:zhao:18}. 
Namely, we consider for a given parameter $\varepsilon>0 $  the IPDE:
 \begin{eqnarray}
 \p_t  u(t,x) +  {L}_{\alpha} u(t,x) + F_\varepsilon(t,x) \cdot D_x u(t,x) +\varepsilon \Delta^{\frac 12} u(t,x)&=& -f_\varepsilon(t,x),\quad \text{on }[0,T)\times \R^d, \notag\\
u(T,x) &=& g_\varepsilon(x),\quad \text{on }\R^d,  \label{we_EPS}
 \end{eqnarray}
where $f_\varepsilon, g_\varepsilon $ are mollified version of the initial sources and terminal condition $f$ and $g$ in time-space and space respectively which satisfy uniformly  w.r.t. the mollification procedure assumption \A{A}. Also,  $F_\varepsilon $ stands for a mollified truncation of $F$ so that for any fixed $\varepsilon>0 $, $F_\varepsilon $ is smooth, bounded and uniformly $\beta $-H\"older continuous in space and satisfies assumption \A{A}uniformly w.r.t. the mollification procedure  as well.\\


The procedure is the following we aim at showing that, for any fixed $\varepsilon>0 $, there is a unique solution $u:=u^\varepsilon$ to equation \eqref{we_EPS} which belongs to the function space ${\mathscr C}_b^{1+\beta}([0,T]\times \R^d) $ (where for the regularity of the generalized time derivative in point iii) of the corresponding definition at p. 6, the parameter $\alpha+\beta-1$ has to be replaced by $\beta$ because, from the $ \varepsilon \Delta^{\frac 12}$ regularization term, we go back to the sub-critical case). The next step then consists in rewriting \eqref{we_EPS} as:
 \begin{eqnarray}
 \p_t  u(t,x) +  {L}_{\alpha} u(t,x) + F_\varepsilon(t,x) \cdot D_x u(t,x) &=& -f_\varepsilon(t,x)-\varepsilon \Delta^{\frac 12} u(t,x),\quad \text{on }[0,T)\times \R^d, \notag\\
u(T,x) &=& g_\varepsilon(x),\quad \text{on }\R^d,  \label{we_EPS_TO_BE_DEVELOP_ALONG_TILDE_P_ALPHA}
 \end{eqnarray}
and to establish that $\varepsilon \|\Delta^{\frac 12} u\|_{L^\infty([0,T],C_b^{\beta})}$ 
is controlled uniformly in $\varepsilon $   allowing thus to expand $u $ along the frozen semi-groups $\big(\tilde P_{s,t}^{(\tau,\xi)}\big)_{0\le s,t\le T} $ as in Section \ref{PERT} to establish that the solution satisfies the Schauder estimates uniformly in $\varepsilon$. The existence of a solution in ${\mathscr C}_b^{\alpha+\beta}([0,T]\times \R^d) $ then follows from a standard compactness argument letting $\varepsilon $ go to $0 $. \textcolor{black}{We \textcolor{black}{also} point out that, although inefficient in our case, the continuity method described above will be used \textcolor{black}{several times} for the analysis of the equation \eqref{we_EPS}.} 


\subsection{A priori controls for the regularized equation}

\subsubsection{Estimates for a generic source in $L^\infty([0,T],C_b^\beta) $}
 We focus in this section on an equation of the form
 \begin{eqnarray}
 \p_t  v(t,x)  + F_\varepsilon(t,x) \cdot D_x v(t,x) +\varepsilon \Delta^{\frac 12} v(t,x)&=& -\bar f_\varepsilon(t,x),\quad \text{on }[0,T)\times \R^d, \notag\\
v(T,x) &=& g_\varepsilon(x),\quad \text{on }\R^d,  \label{we_EPS_SUB_CRITICAL}
 \end{eqnarray}
where $\bar f_\varepsilon,F_\varepsilon $ is in $C_b^\beta([0,T]\times \R^d) $ (i.e., $f_\varepsilon$, $F_\varepsilon $ are $\beta $-H\"older continuous in both time and space) and $g_\varepsilon \in C^{1+\beta}(\R^d) $.
 
 In this framework, \textcolor{black}{it follows from \cite{miku:prag:14}} 
 that there exists a unique solution $v:=v^\varepsilon$ to \eqref{we_EPS_SUB_CRITICAL} in ${\mathscr C}_b^{1+\beta}([0,T]\times \R^d) $ which satisfies:
 \begin{equation}\label{EXPL_SCHAUDER}
 \|v\|_{L^\infty([0,T],C_b^{1+\beta})}\le C \big(\Theta_1(\varepsilon) \|g_\varepsilon\|_{C^{1+\beta}_b}+\Theta_2(\varepsilon)\|\bar f_\varepsilon\|_{L^{\infty}([0,T],C_b^{\beta})} \big).  
 \end{equation}
 With respect to the previously described procedure, we actually need to precisely quantify how the $\big(\Theta_i(\varepsilon)\big)_{i\in \{1,2\}}$ behave when $\varepsilon $ goes to $0$. 
 \textcolor{black}{This behavior actually depends {\color{black}on} the smoothing effects associated with $P_{s,1}^\varepsilon $ which denotes the semi-group associated with
 $\varepsilon \Delta^{\frac 12}$}. We can therefore appeal to the results of Section \ref{PERT} considering appropriate scaling arguments. Namely,
 let us consider $\textcolor{black}{w:=w^{\varepsilon}}  $ \textcolor{black}{solving}:
 \begin{eqnarray*}
 \p_t  w(t,x)   +\varepsilon \Delta^{\frac 12} w(t,x)&=& -\bar f_\varepsilon(t,x),\quad \text{on }[0,T)\times \R^d, \notag\\
w(T,x) &=& g_\varepsilon(x),\quad \text{on }\R^d.  \label{we_EPS_SUB_CRITICAL_NO_DRIFT}
 \end{eqnarray*}
Setting  $w(t,x) =: \textcolor{black}{\bar w} (t, x /\varepsilon)$ then 
\begin{eqnarray*}
 \p_t  \bar w(t,y)   +  \Delta^{\frac 12} \bar w(t,y)&=& -\bar f_\varepsilon(t,\varepsilon y),\quad \text{on }[0,T)\times \R^d, \notag\\
\bar w(T,y) &=& g_\varepsilon(\varepsilon y),\quad \text{on }\R^d,  \label{we_EPS_SUB_CRITICAL_NO_DRIFT_RESCALED}
 \end{eqnarray*}
\textcolor{black}{From the proof of point \textit{iii)} of Proposition \ref{SCHAUDER_FROZEN} (see also Lemma \ref{sd}), we get:} 
\begin{gather*} 
\varepsilon^{1+ \beta} [Dw(t, \cdot )]_{\beta} = [D\bar w (t, \cdot  )]_{\beta} \le C ( \varepsilon^{1 + \beta} [ Dg]_{\beta }  +  \varepsilon^{\beta} [f]_{\beta,T}).
\end{gather*}
We thus derive that there exists $C:=C(\A{A})$ idependent of $\varepsilon $ s.t. :
\begin{equation}
\label{SCHAUDER_DRIFTLESS_EPS}
\|w\|_{L^\infty([0,T],C_b^{1+\beta})}\le C\Big(\|g_\varepsilon\|_{C_b^{1+\beta}}+\varepsilon^{-1}\|\bar f_\varepsilon\|_{L^\infty([0,T],C_b^\beta)} \Big).
\end{equation}
We can then follow the arguments of  \cite{kryl:prio:10} to establish from the continuity method that a similar bound will also hold for the unique solution $v$ in ${\mathscr C}_b^{1+\beta}([0,T]\times \R^d) $ of the drifted equation \eqref{we_EPS_SUB_CRITICAL}. Namely,
\begin{equation}
\label{SCHAUDER_DRIFTED_EPS}
\|v\|_{L^\infty([0,T],C_b^{1+\beta})}\le C\Big(\|g_\varepsilon\|_{C_b^{1+\beta}}+\varepsilon^{-1}\|\bar f_\varepsilon\|_{L^\infty([0,T],C_b^\beta)} \Big),
\end{equation}
where the above $C$ also depends on $[F]_{\beta,T} $. Equation \eqref{SCHAUDER_DRIFTED_EPS} specifies that the functions  $\big(\Theta_i(\varepsilon)\big)_{i\in \{1,2\}} $ in \eqref{EXPL_SCHAUDER} actually write $\Theta_1(\varepsilon)=1,\Theta_2(\varepsilon)=\varepsilon^{-1} $.

 \subsubsection{Solvability of equation \eqref{we_EPS}}
In order to prove existence we would like to benefit from the results of the previous section. From equation \eqref{we_EPS} we introduce for a parameter $\lambda\in [0,1] $:
\begin{eqnarray}
 \p_t  u(t,x) +   F_\varepsilon(t,x) \cdot D_x u(t,x) +\varepsilon \Delta^{\frac 12} u(t,x)&=& -f_\varepsilon(t,x)- \lambda {L}_{\alpha} u(t,x),\quad \text{on }[0,T)\times \R^d, \notag\\
u(T,x) &=& g_\varepsilon(x),\quad \text{on }\R^d,  \label{we_EPS_lambda}
 \end{eqnarray}
which can be viewed as a particular case of \eqref{we_EPS_SUB_CRITICAL} with $\bar f_\varepsilon=f_\varepsilon+\lambda L_\alpha u $.

We now recall a useful inequality. For $\theta\in (0,1] $ consider an operator $L_\theta$ with symbol of the form \eqref{LEVY_KHINTCHINE} satisfying \A{ND} and  $\gamma\in (0,1)$ s.t $\theta+\gamma>1 $.
There exists $C_{\theta,\gamma}$ s.t. for a function $\varphi \in C_b^{\gamma+\theta}$,  it holds that:
\begin{equation}
\label{CTR_FRAC_OP_HOLDER}
\|L_\theta \varphi \|_{C_b^\gamma}\le C_{\theta,\gamma} \|\varphi\|_{C_b^{\gamma+\theta}}.
\end{equation}
Recalling that H\"older spaces can be viewed as Besov spaces, this inequality is as a direct consequence of norm equivalences on Besov spaces  (see e.g. Triebel \cite{trie:83}). We also provide a direct proof of \eqref{CTR_FRAC_OP_HOLDER} in Appendix \ref{SEC_PROOF_EST_OP} for a self-contained presentation.

    From \eqref{CTR_FRAC_OP_HOLDER}, we derive that for $v\in {\mathscr C}_b^{1+\beta}([0,T]\times \R^d), L_\alpha v \in L^\infty([0,T],C_b^{1+\beta-\alpha})  $. Hence, the continuity method will also give from the previous estimates that for any fixed $\varepsilon>0 $ there exists a unique solution $\textcolor{black}{u^\varepsilon=u}\in{\mathscr C}_b^{1+\beta}([0,T]\times \R^d)$ to \eqref{we_EPS_lambda} for $\lambda \in [0,1] $ and therefore to  \eqref{we_EPS} corresponding to $\lambda=1 $. Also, for the unique solution of \eqref{we_EPS} in ${\mathscr C}_b^{1+\beta}([0,T]\times \R^d) $ it holds that: 
\begin{equation}\label{FULL_SCHAUDER_VISCOUS}
\|u\|_{L^\infty([0,T],C_b^{1+\beta})}\le C \big(\|g_\varepsilon\|_{C_b^{1+\beta}}+\varepsilon^{-1}\|f_\varepsilon\|_{L^\infty([0,T],C_b^\beta)} \big) , C:=C(\A{A},
\end{equation}
and for all $0\le t<s\le T,\ x\in  \R^d $,
\begin{eqnarray}\label{INTEGRAL_VERSION_EPS}
u(t,x)=u(s,x)+\int_t^s dr f_\varepsilon(r,x)-\int_t^s dr\Big(\varepsilon \Delta^{\frac 12}+L_\alpha +F_\varepsilon(r,x)  \cdot D \Big) u(r,x).
\end{eqnarray}


\subsection{Viscosity viewed as a source and compactness arguments} 
We now rewrite equation \eqref{we_EPS} viewing the viscous perturbation as a source. Namely, as in \eqref{we_EPS_TO_BE_DEVELOP_ALONG_TILDE_P_ALPHA}. We now observe from \eqref{FULL_SCHAUDER_VISCOUS} and \eqref{CTR_FRAC_OP_HOLDER} that:
\begin{eqnarray*}
\varepsilon\|\Delta^{\frac 12} u\|_{L^\infty([0,T],C_b^{\beta})}&\le& C_{1,\beta}\varepsilon \|u\|_{L^\infty([0,T],C_b^{1+\beta})}\le C_{1,\beta} C(\varepsilon\|g_\varepsilon\|_{C_b^{1+\beta}}+\|f_\varepsilon\|_{L^\infty([0,T],C_b^\beta)})\\
&\le & C_{1,\beta} C(\varepsilon\|g_\varepsilon\|_{C_b^{1+\beta}}+\|f\|_{L^\infty([0,T],C_b^\beta)}).
\end{eqnarray*}
Hence, reproducing the previous expansion of Section \ref{PERT} we derive that $u:=u^\varepsilon $ solving \eqref{we_EPS} satisfy the estimate:
\begin{eqnarray}
\|u^\varepsilon\|_{L^\infty([0,T],C_b^{\alpha+\beta})}&\le& C\Big(\|g\|_{C_b^{\alpha+\beta}}  +\|f\|_{L^\infty([0,T],C_b^\beta)} +\varepsilon \|g_\varepsilon\|_{C_b^{1+\beta}}\Big)\notag\\
&\le & C\Big(\|g\|_{C_b^{\alpha+\beta}}  +\|f\|_{L^\infty([0,T],C_b^\beta)}+h(\varepsilon) \|g\|_{C_b^{\alpha+\beta}}\Big)\notag\\
&\le & C\Big(\|g\|_{C_b^{\alpha+\beta}}  +\|f\|_{L^\infty([0,T],C_b^\beta)} \Big),\label{SCHAUDER_PRIMA_DI_COMPATTEZZA}
\end{eqnarray}
where $h(\varepsilon)\underset{\varepsilon \rightarrow 0}{\rightarrow} 0 $, considering a suitable regularization of $g $, i.e., s.t. $\varepsilon\|g_\varepsilon\|_{C_b^{1+\beta}}\le h(\varepsilon) \|g\|_{C_b^{\alpha+\beta}} $, for the second inequality.

  Note that proceeding as in \textcolor{black}{Corollary 4.2} of \cite{kryl:prio:10}, we also have local $\gamma$-H\"older continuity in $[0,T] \times K $, for some small $\gamma>0,$ for $u^{\varepsilon}$ and $Du^{\varepsilon}$ for any compact set $K \subset \R^d$ (with a control of the  H\"older norm independent of $\varepsilon$).

 From the  Ascoli-Arzel\`a theorem,    we deduce that there exists a subsequence $\varepsilon_n \underset{n}{\rightarrow} 0$  and  a continuous function $u$ on ${[0,T] \times \R^d}$   having bounded and continuous derivatives
with respect to $x$
  such that
   $$
   (u^{\varepsilon_n}, D u^{\varepsilon_n} ) \to(
u, Du)
$$
as $n \to \infty$,  uniformly on  bounded subsets of
$[0,T ] \times \R^d$. We also have that 
  $u\in L^\infty([0,T],C_b^{\alpha+\beta}) $ satisfies the last inequality of \eqref{SCHAUDER_PRIMA_DI_COMPATTEZZA}.
  Rewrite now \eqref{INTEGRAL_VERSION_EPS} along the considered subsequence:
\begin{eqnarray}\label{INTEGRAL_VERSION_EPS_SUBSEQUENCE}
u^{\varepsilon_n}(t,x)=u^{\varepsilon_n}(s,x)+\int_t^s dr f_{\varepsilon_n}(r,x)-\int_t^s dr\Big(\varepsilon_n \Delta^{\frac 12}+L_\alpha +F_{\varepsilon_n}(r,x)  \cdot D \Big) u^{\varepsilon_n}(r,x).
\end{eqnarray}
It is readily seen that, in order to pass to the limit in the previous equation, the only \textit{delicate} term to analyze is $\varepsilon_n \Delta^{\frac 12} u^{\varepsilon_n}(t,x)  $.
Observe that:
\begin{eqnarray*}
|\Delta^{\frac 12} u^{\varepsilon_n}(t,x)|&\le& \Big|\int_{|z|\le 1} \Big( u^{\varepsilon_n}(t,x+z)-u^{\varepsilon_n}(t,x)-D u^{\varepsilon_n}(t,x)\cdot z \Big)\frac{dz}{|z|^2}\Big|+\int_{|z|\ge 1} \frac{2\|u^{\varepsilon_n}\|_\infty}{|z|^{2}}dz\notag\\
&\le &C\big( [Du^{\varepsilon_n}]_{\alpha+\beta-1,T} +\|u^{\varepsilon_n}\|_\infty\big)\le C\big(\|g\|_{C_b^{\alpha+\beta}}  +\|f\|_{L^\infty([0,T],C_b^\beta)} \big).
\end{eqnarray*}
We can then pass to the limit in \eqref{INTEGRAL_VERSION_EPS_SUBSEQUENCE} and so $u$ satisfies:
\begin{equation}
u(t,x)=u(s,x)+\int_t^s dr f(r,x)-\int_t^s dr\Big(L_\alpha +F(r,x)  \cdot D \Big) u(r,x).
\end{equation}
\\
The Schauder estimates of Theorem \ref{THEO_SCHAU_ALPHA} then gives uniqueness. This also proves Theorem \ref{EXISTENCE_UNIQUENESS_ALPHA}.

\section{Proof of the Property (${\mathscr P}_\beta $) in the indicated case}
\label{SEC_PROOF_PROP_P}
\subsection{ Proof of Proposition \ref{CTR_OF_INTEGRABILITY_DER_GEN_STABLE1} on stable like operators close to $\triangle^{\alpha/2}$, $\alpha \in (0,1)$}\label{SEC_KOLO_P_BETA}

In other words, the L\'evy measure $\nu$ in \eqref{STAB_OPERATOR} rewrites:
$$\nu(dy) =\nu_{\alpha}(dy) = f\Big(\frac{y}{|y|}\Big)\frac{dy}{|y|^{d+\alpha}}.$$

We have to prove \eqref{wss}   for all $\alpha\in (0,1) $, and $\gamma \in [0,1] $; this   will rely on global estimates on the derivatives of $p_{\alpha}(t, \cdot )$, $t>0,$ which  can be deduced from the work of  Kolokoltsov \cite{kolo:97}. First by \cite[formula (2.38) in Proposition 2.6]{kolo:97} we know that
 there exists $c= c(\alpha, \eta )$ (where $\eta $ denotes the non degeneracy constant associated with the spectral measure in \eqref{LEVY_STABLE}) such that
\begin{equation} \label{df}
|D_y p_{\alpha} (t,y)| \le c \big (\frac{1}{t^{1/\alpha}} \wedge \frac{1}{|y|}\big ) \, p_{\alpha} (t,y),\;\;\; y \in \R^d, \, t>0.
\end{equation}
We need the following result for the second derivatives. 
\begin{lem}\label{222} For any positive $K>0$, there exists $c= c(\alpha, \eta , K)>0$ such that, for $|y| \le K t^{1/\alpha}$ we have 
\begin{equation} \label{df1}
|D_y^2 p_{\alpha} (t,y)| \le \frac{c}{t^{2/\alpha}}  \, p_{\alpha} (t,y),\;\;\;  \, t>0,
\end{equation}
and for  $|y| > K t^{1/\alpha}$ we have
\begin{equation} \label{df2}
|D_y^2 p_{\alpha} (t,y)| \le  \frac{c}{t^{}} \cdot  \frac{1}{|y|^{2- \alpha}}    \, p_{\alpha} (t,y),\;\;\; \, t>0.
\end{equation}
\end{lem}
\begin{proof}
As in formula (2.23) of \cite{kolo:97}
we consider with $b=2$, $\sigma =t >0$,
\begin{gather*}
\phi_{2, \theta}(x, \alpha, t \mu) = \frac{1}{(2 \pi)^d} \int_{\R^d} 
\, [ \int_{{\mathbb S}^{d-1}} \langle p ,s\rangle^2 \theta(ds)]
 \, \exp \Big (- t\int_{{\mathbb S}^{d-1}} |\langle p,s\rangle|^\alpha \mu(ds) \Big) e^{-i \langle p ,  x \rangle } dp.
\end{gather*}
where $\theta$ is  a finite positive measure on ${\mathbb S}^{d-1}$ such that $\theta ({\mathbb S}^{d-1}) \le M$
for some $M>0$. 

Then by \cite[Proposition 2.5]{kolo:97} (see in particular estimates (2.29) and (2.30))  for any positive $K>0$, there exists $c= c(\alpha, \eta , K, M)>0$ such that, for $|x| \le K t^{1/\alpha}$ we have 
\begin{equation} \label{df1s}
|\phi_{2, \nu}(x, \alpha, t \mu)| \le \frac{c}{t^{2/\alpha}}  \, p_{\alpha} (t,x),\;\;\; x \in \R^d, \, t>0,
\end{equation}
and for  $|x| > K t^{1/\alpha}$ we have
\begin{equation} \label{df2s}
|\phi_{2, \nu}(x, \alpha, t \mu)| \le  \frac{c}{t^{}} \cdot  \frac{1}{|x|^{2- \alpha}}    \, p_{\alpha} (t,x),\;\;\; x \in \R^d, \, t>0
\end{equation}
(recalling  that in \cite{kolo:97} $p_{\alpha}(t,x)$ is denoted by  $ S(x, \alpha, t \mu)$).  Let us fix $M>0$.
Let $h \in \R^d$, $h \not = 0$, with $|h| \le \sqrt{M}$. Choosing  the measure $\theta = |h|^2\, \delta_{ \frac{ h} {|h|}}$ we find  that 
\begin{gather*}
\phi_{2, \theta}(x, \alpha, t \mu) = \frac{1}{(2 \pi)^d} \int_{\R^d} 
\,  \langle p ,h\rangle^2 
 \, \exp \Big (- t\int_{{\mathbb S}^{d-1}} |\langle p,s\rangle|^\alpha \mu(ds) \Big) e^{-i \langle p ,  x \rangle } dp = - D^2_{hh} p_{\alpha}(t,x),
\end{gather*}
i.e., we are considering the second derivative of $p_{\alpha}(t,\cdot)$ in the direction $h$. Assertions \eqref{df1} and \eqref{df2} now readily follow from \eqref{df1s}, \eqref{df2s}. 
\end{proof}

\def\a{
It is well known, see e.g. Kolokoltsov \cite{kolo:97} that there exists $C\ge 1$ s.t. for all $(t,z)\in \R_+^*\times\R^d $:
\begin{equation}
\label{CTR_DENS_ROT_INV}
C^{-1} \check p_\alpha(t,z) \le p_\alpha(t,z)\le C \check p_\alpha(t,z),\ \check p_\alpha(t,z):=\frac{1}{t^{\frac d\alpha}}\frac{\check c}{\Big(1+\frac{|z|}{t^{\frac 1 \alpha}}\Big)^{d+\alpha}}, \end{equation}
where $\check c $ is chosen s.t. $\int_{\R^d}\check p_\alpha(t,z) dz=1 $. In particular, \eqref{CORRESP_SHIFTED} and \eqref{CTR_DENS_ROT_INV} yield:
\begin{equation}
\label{BOUND_FROZEN_DENSITY}
C^{-1}\check p_\alpha\Big(s-t,y-m_{s,t}^{(\tau,\xi)}(x)\Big) \le \tilde p_\alpha^{(\tau,\xi)}(s,t,x,y)\le  C \check p_\alpha\Big(s-t,y-m_{s,t}^{(\tau,\xi)}(x)\Big).
\end{equation}

We will also thoroughly use the following important controls on the derivatives of $p_\alpha$ in the considered  rotationally invariant case. It
follows  e.g. from Lemma 5 and Remark 6 in Bogdan and Jakubowicz \cite{bogd:jaku:07} that:
\begin{equation}
\label{DER_HK_ROT_INV}
| D_xp_\alpha (t,z)|\le \frac{C}{t^{\frac 1\alpha}}\frac{1}{t^{\frac{d}{\alpha}}}\frac{1}{(1+\frac{|z|}{t^{\frac 1\alpha}})^{d+\alpha+1}}=:\frac{C}{t^{\frac 1\alpha}}\bar p_\alpha(t,z).
\end{equation}
In other words, the derivative induces a concentration gain. The same arguments as in \cite{bogd:jaku:07}, see also the decompositions in \cite{wata:07}, also yield that for $\ell\in \{1,2\} $:
\begin{equation}
\label{DER_HK_ROT_INV_HIGH}
| D_x^\ell p_\alpha (t,z)|\le \frac{C}{t^{\frac \ell\alpha}}\frac{1}{t^{\frac{d}{\alpha}}}\frac{1}{(1+\frac{|z|}{t^{\frac 1\alpha}})^{d+\alpha+\ell}}\le \frac{C}{t^{\frac \ell\alpha}}\bar p_\alpha(t,z).
\end{equation}

 \begin{REM}[About the driving symmetric stable processes]  
We point out that for a more general non-degenerate symmetric stable process, which has spherical measure equivalent to the Lebesgue measure of ${\mathbb S}^{d-1}$, denoting by   $\hat p_\alpha $ its density, we have the following weaker control (see e.g. Kolokoltsov \cite{kolo:97}). There exists $C:=C(c,d)$ s.t. for all $(t,z)\in \R_+^*\times \R^d$, $\ell \in \leftB 0,2\rightB $:
\begin{eqnarray}
 |\partial_t^\ell  \hat p_\alpha(t,z)|&\le& \frac{C}{t^\ell}\check  p_\alpha(t,z),\ 
 |D_x^\ell  \hat p_\alpha(t,z)|\le \frac{C}{t^{\frac \ell \alpha}}\check  p_\alpha(t,z),\label{CTR_DER_HK_STAB}
\end{eqnarray} 
with $\check p_\alpha $ as in \eqref{CTR_DENS_ROT_INV}.

In particular, in this more general setting, there is no concentration gain on the derivatives, meaning that, in that case, the tails of $|D_x^\ell  \hat  p_\alpha(t,z)|$ can only integrate function growing as $C|z|^\eta,\ \eta<\alpha $ when $|z|$ goes to infinity.
\end{REM}
}
\noindent {\bf Proof of Proposition \ref{CTR_OF_INTEGRABILITY_DER_GEN_STABLE1}.}
 Let $l=1$. Using \eqref{df} we find 
\begin{gather*}
\int_{|y| \le t^{1/\alpha}} |y|^{\gamma}\, |D_y p_{\alpha}(t, y)| dy \le 
 \int_{|y| \le t^{ 1/\alpha}} t^{\gamma/\alpha}\, | D_y p_{\alpha}(t, y) | dy
 \\
 \le  c
 \int_{|y| \le t^{ 1/\alpha}} t^{\gamma/\alpha}\,   t^{- 1/\alpha} p_{\alpha}(t, y) dy \le  c  t^{[\gamma-1]/\alpha}.
\end{gather*}
On the other hand
\begin{gather*}
\int_{|y| > t^{1/\alpha}} |y|^{\gamma}\, |D_y p_{\alpha}(t, y)| dy \le 
 c \int_{|y| > t^{ 1/\alpha}}  
  |y|^{\gamma-1}
  p_{\alpha}(t, y)  dy
 \\
 \le  c
 \int_{|y| > t^{ 1/\alpha}}   t^{[\gamma-  1]/\alpha} p_{\alpha}(t, y) dy \le  c  t^{[\gamma-1]/\alpha}.
\end{gather*}
 Let $k =2$. Using \eqref{df1} and \eqref{df2} we find 
 \begin{gather*}
\int_{|y| \le t^{1/\alpha}} |y|^{\gamma}\, |D_y^2 p_{\alpha}(t, y)| dy  
 \le  c
 \int_{|y| \le t^{ 1/\alpha}} t^{\gamma/\alpha}\,   t^{- 2/\alpha} p_{\alpha}(t, y) dy \le  c  t^{[\gamma-2]/\alpha}.
\end{gather*}
Since $\alpha + \gamma <2$, we find
 \begin{gather*}
\int_{|y| > t^{1/\alpha}} |y|^{\gamma}\, |D_y^2 p_{\alpha}(t, y)| dy \le 
 c \int_{|y| > t^{ 1/\alpha}}  \frac{1}{t}
  |y|^{\gamma-2 + \alpha}
  p_{\alpha}(t, y)  dy
 \\
 \le  c
 \int_{|y| > t^{ 1/\alpha}}  \frac{1}{t}  t^{[\gamma-2 + \alpha]/ \alpha}
  p_{\alpha}(t, y) dy \le  c  t^{[\gamma-2]/\alpha}.
 \end{gather*}
 and the assertion follow.
 \qed

\vskip 1mm We can mention that, in the specific case of the rotationally invariant heat kernel, corresponding to the fractional Laplacian $\Delta^{\alpha/2} $, the previous computations could have been shortened exploiting explicitly the concentration gain for the derivatives of the heat kernel. Namely, it is known  from Lemma 5 and Remark 6 in Bogdan and Jakubowicz \cite{bogd:jaku:07} that in that case:
\begin{equation}
\label{DER_HK_ROT_INV}
| D_xp_\alpha (t,z)|\le \frac{C}{t^{\frac 1\alpha}}\frac{1}{t^{\frac{d}{\alpha}}}\frac{1}{(1+\frac{|z|}{t^{\frac 1\alpha}})^{d+\alpha+1}}.=:\frac{C}{t^{\frac 1\alpha}}\bar p_\alpha(t,z).
\end{equation}
In other words, the derivative induces a concentration gain. The same arguments as in \cite{bogd:jaku:07}, see also the decompositions in \cite{wata:07}, also yield the corresponding result for  $|D_x^2 p_\alpha(t,z)|\le Ct^{-2/\alpha}\bar p_\alpha(t,z) $.

\subsection{Proof of Proposition \ref{CTR_OF_INTEGRABILITY_DER_GEN_STABLE} on   symmetric  stable operators with  $\alpha\in (1/2,1) $}

Here we are considering  symmetric  stable operators such that  \eqref{EQ_ND} holds. We want to prove 
 \eqref{wss}   for all $\alpha\in (0,1) $, and $\gamma \in [0,\alpha) $.
 The result  follows easily from   the following key lemma.
\begin{lem}[Bounds and Sensitivities of the Stable Heat Kernel]
\label{SENS_SING_STAB}
There exists $C:=C(\A{A})$ s.t. for all $\ell \in \{1,2\} $, $t>0$, and $ y\in \R^d $:
\begin{equation} 
\label{CTR_BOUND_SING_KER}
|D_y^\ell p_\alpha(t,y)|\le \frac{C}{t^{\ell/\alpha}} q(t,y),
\end{equation}
where $\big( q(t,\cdot)\big)_{t>0}$ is a  family of probability densities on $\R^d $ such that $q(t,y) = {t^{-d/ \alpha}} \, q (1, t^{- 1/\alpha} y)$, $ t>0$, $ \in \R^d$ and  for all $\gamma \in [0,\alpha) $, there exists a constant $c:=c(\alpha,\eta,\gamma)$ s.t. 
\begin{equation}
\label{INT_Q}
\int_{\R^N}q(t,y)|y|^\gamma dy
\le C_{\gamma}t^{\frac{\gamma}{\alpha}},\;\;\; t>0,
\end{equation}

\end{lem}
 \begin{REM}
\label{NOTATION_Q} 
{
From now on, for the family of stable densities $\big(q(t,\cdot)\big)_{t>0} $, 
we also use the notation  $q(\cdot):=q(1,\cdot) $,  i.e., without any specified argument $q(\cdot)$ stands for the density $q$ at time $1$.}
\end{REM} 
\begin{proof}
We denote by $(S_t)_{t\ge 0} $ a stable process defined on some probability space $(\Omega,{\mathcal F},({\mathcal F}_t)_{\ge 0},\P) $ whose L\'evy exponent is given by \eqref{LEVY_KHINTCHINE}. For $t>0 $, the heat kernel $p_\alpha(t,\cdot) $ associated with $L_\alpha$ given in \eqref{DENSITY_FTI} is then precisely the density of $S_t$.

 Let us recall that, for a given fixed $t>0$, we can use an It\^o-L\'evy  decomposition
 at the associated  characteristic stable time scale (i.e., the truncation is performed at the threshold $t^{\frac {1} {\alpha}} $) 
to write $S_t:=M_t+N_t$
where $M_t$ and $N_t $ are independent random variables. 
More precisely, 
 \begin{equation} \label{dec}
 N_s = \int_0^s \int_{ |x| > t^{\frac {1} {\alpha}} }
\; x  P(du,dx), \;\;\; \; M_s = S_s - N_s, \;\; s \ge 0,
 \end{equation} 
where $P$ is the  Poisson random measure associated with the process $S$; for the considered fixed $t>0$,
 $M_t$ and $N_t$ correspond to
 the \textit{small jumps part } and
\textit{large jumps part} respectively. 
A similar decomposition has been already used in
 \cite{wata:07},        \cite{szto:10} 
 and \cite{huan:meno:15}, \cite{huan:meno:prio:19} (see in particular Lemma 4.3 therein). It is useful to note that the cutting threshold in \eqref{dec} precisely yields for the considered $t>0$ that:
\begin{equation} \label{ind}
N_t  \overset{({\rm law})}{=} t^{\frac 1\alpha} N_1 \;\; \text{and} \;\;
M_t  \overset{({\rm law})}{=} t^{\frac 1\alpha} M_1.
\end{equation}  
To check the assertion about $N$ we start with 
$$
\E [e^{i \langle p , N_t \rangle}] = 
\exp \Big(  t
\int_{\S^{d-1}} \int_{t^{\frac 1\alpha}}^{\infty}
 \Big(\cos (\langle p, r\xi \rangle)  - 1  \Big) \, \frac{dr}{r^{1+\alpha}}\tilde \mu_{S}(d\xi) \Big), \;\; p \in \R^d
$$
(see  \eqref{LEVY_KHINTCHINE} and \cite{sato:99}). Changing variable $\frac{r}{t^{\frac 1\alpha}} =s$
we get that $\E [e^{i \langle p , N_t \rangle}]$ $= \E [e^{i \langle p , t^{\frac 1\alpha} N_1 \rangle}]$ for any $p \in \R^d$ and this shows the assertion (similarly we get
the statement for $M$).
The density of $S_t$ then writes
\begin{equation}
\label{DECOMP_G_P}
p_S(t,x)=\int_{\R^d} p_{M}(t,x-\xi)P_{N_t}(d\xi),
\end{equation}
where $p_M(t,\cdot)$ corresponds to the density of $M_t$ and $P_{N_t}$ stands for the law of $N_t$. 
{From Lemma A.2 in \cite{huan:meno:prio:19} (see as well
Lemma B.1  }
in \cite{huan:meno:15}),  $p_M(t,\cdot)$  belongs to the Schwartz class ${\mathscr S}(\R^N) $ and satisfies that for all $m\ge 1 $ and all $\ell \in \{0,1,2\} $, there exist constants  $\bar C_m,\ C_{m}$ s.t. for all $t>0,\ x\in  \R^d  $:
\begin{equation}
\label{CTR_DER_M}
|D_x^\ell p_M(t,x)|\le \frac{\bar C_{m}}{t^{\frac{\ell}{\alpha} }} \, p_{\bar M}(t,x),\;\; \text{where} \;\; p_{\bar M}(t,x)
:=
\frac{C_{m}}{t^{\frac{d}{\alpha}}} \left( 1+ \frac{|x|}{t^{\frac 1\alpha}}\right)^{-m}
\end{equation}
where $C_m$ is chosen in order that {\it $p_{\bar M}(t,\cdot ) $ be a probability density.}


We carefully point out that, to establish the indicated results, since we are led to consider potentially singular spherical measures,  we only focus on integrability properties similarly to \cite{huan:meno:prio:19} and not on pointwise density estimates as for instance in \cite{huan:meno:15}. The main idea thus consists in exploiting 
{\eqref{dec},}  \eqref{DECOMP_G_P} and \eqref{CTR_DER_M}.
The derivatives on which we want to obtain quantitative bounds  will be expressed through derivatives of $p_M(t,\cdot)$, which also give the corresponding time singularities. However, as for general stable processes, the integrability restrictions come from the large jumps (here $N_t $) and only depend on its index $\alpha$.
A crucial point then consists in observing  that the convolution $\int_{\R^d}p_{\bar M}(t,x-\xi)P_{N_t}(d\xi) $ actually corresponds to the density of the random variable 
\begin {equation} \label{we2}
\bar S_t:=\bar M_t+N_t,\;\; t>0 
\end{equation}
 (where $\bar M_t $ has density $p_{\bar M}(t,.)$ and is independent of $N_t $; 
 {to have such decomposition one can define each $\bar S_t$ on a product probability space}). Then, the integrability properties of $\bar M_t+N_t $, and more generally of all random variables appearing below, come from those of $\bar M_t $ and $N_t$. 

One can easily check that $p_{\bar M}(t,x) = {t^{-\frac d\alpha}} \, p_{\bar M} (1, t^{-\frac 1\alpha} x),$ $ t>0, \, $ $x \in \R^d.$  Hence 
$$
\bar M_t  \overset{({\rm law})}{=} t^{\frac 1\alpha} \bar M_1,\;\;\; N_t  \overset{({\rm law})}{=} t^{\frac 1\alpha}  N_1.
$$
By independence of $\bar M_t$ and $N_t$, using the Fourier transform, one can easily prove that 
\begin{equation} \label{ser1}
\bar S_t  \overset{({\rm law})}{=} t^{\frac 1\alpha} \bar S_1.
\end{equation} 
Moreover, 
$
\E[|\bar S_t|^\gamma]=\E[|\bar M_t+N_t|^\gamma]\le C_\gamma t^{\frac\gamma \alpha}(\E[|\bar M_1|^\gamma]+\E[| N_1|^\gamma])\le C_\delta t^{\frac\gamma \alpha}, \; \gamma \in (0,\alpha).
$ 
This shows that the density of $\bar S_t$ verifies \eqref{INT_Q}. The controls on the derivatives are derived similarly using 
\ref{CTR_DER_M} for $\ell\in \{1,2\} $ and the same previous argument.
\end{proof}

\subsection{Proof of Property $({\mathscr P}_\beta) $ for the relativistic stable operator}
\label{SEC_P_BETA_REL}
%
%
%
%

In order to prove the required integrability properties we first have to differentiate the density. 

Starting from \eqref{re2} and similarly to  the proof of Lemma 5 in \cite{bogd:jaku:07} we  can differentiate under the integral sign and obtain, for $x \not =0$,
\begin{eqnarray}
|D_x p_{\alpha,m}(t,x)| &=& \Big | e^{mt} \int_0^{\infty} D_x g (u,x) e^{-m^{\frac{2}{\alpha}} u} \;\; \theta_{\alpha} (t,u) du    \Big|
= \Big | e^{mt} \; \frac{-x}{2} \int_0^{\infty}  \frac{g (u,x)}{u} e^{-m^{\frac{2}{\alpha}} u} \;\; \theta_{\alpha} (t,u) du    \Big|
\nonumber\\  
&\le& 
  e^{mt} \; \frac{|x|}{2} \int_0^{\infty}  \frac{g (u,x)}{u}  \; \theta_{\alpha} (t,u) du   = e^{mt} |D_x p_{\alpha,0} (t,x)|.\label{rel2}
\end{eqnarray}
The spatial concentration gain induced by the differentiation of $ p_{\alpha,0}(t,x) $ then provides the required integrability property for the first derivative.
 Similarly, 
\begin{eqnarray*}
D_x^2 p_{\alpha,m}(t,x) &=&  e^{mt} \int_0^{\infty} D_x^2 g (u,x) e^{-m^{\frac{2}{\alpha}} u} \;\; \theta_{\alpha} (t,u) du    
\\
&=&  e^{mt} \; \frac{-1}{2} \int_0^{\infty}  \frac{g (u,x)}{u} e^{-m^{\frac{2}{\alpha}} u} \;\; \theta_{\alpha} (t,u) du   + e^{mt}
\frac{|x|^2}{4} \int_0^{\infty}  \frac{g (u,x)}{u^2} e^{-m^{\frac{2}{\alpha}} u} \;\; \theta_{\alpha} (t,u) du.
\end{eqnarray*}
Since 
$$
 | e^{mt} \; \frac{-1}{2} \int_0^{\infty}  \frac{g (u,x)}{u} e^{-m^{\frac{2}{\alpha}} u} \;\; \theta_{\alpha} (t,u) du | \le 
 \frac{e^{m} |D_x p_{\alpha,0} (t,x)|}{|x|}, \;\;\; x \not =0,
$$
and 
\begin{eqnarray*}
e^{mt}
\frac{|x|^2}{4} \int_0^{\infty}  \frac{g (u,x)}{u^2} e^{-m^{\frac{2}{\alpha}} u} \;\; \theta_{\alpha} (t,u) du
 \le e^{m}
\frac{|x|^2}{4} \int_0^{\infty}  \frac{g (u,x)}{u^2}  \;\; \theta_{\alpha} (t,u) du 
\le e^m |D_x^2 p_{\alpha,0}(t,x)|  +  \frac{e^{m} |D_x p_{\alpha,0} (t,x)|}{|x|},
\end{eqnarray*}
to prove \A{NDb} for $k=2$, we concentrate on  $ \frac{|D_x p_{\alpha,0} (t,x)|}{|x|}$.
The estimate
$$
\frac{1}{|z|} \, | D_x p_{\alpha,0} (t,z)|\le \frac{C}{t^{\frac 1\alpha}}\frac{1}{t^{\frac{d}{\alpha}}}\frac{1}{(1+\frac{|z|}{t^{\frac 1\alpha}})^{d+\alpha+1}} \;
\frac{1}{\frac{|z|}{t^{\frac 1\alpha}} t^{\frac 1\alpha} },
$$
easily yields {$({\mathscr P}_\beta) $} for $k=2$.

\appendix
\section{Proof of some technical results}
\label{APP_TEC}
\subsection{Proof of the flow lemma \ref{FLOW_LEMMA}} 
\textcolor{black}{We first assume for the proof that the control \eqref{22} holds globally, i.e., the function $F$ is globally $\beta $-H\"older continuous with constant $K_0$}.
The point is \textcolor{black}{then} to write for $0\le t\le s\le T,\ (x,x')\in (\R^d)^2$:
\begin{eqnarray}
\theta_{s,t}(x)-\theta_{s,t}(x')&=&x-x'+\int_{t}^s [F(u,\theta_{u,t}(x))-F(u,\theta_{u,t}(x'))]du\notag\\
&=&x-x'+\int_{t}^s [F(u,\theta_{u,t}(x))-F_\delta(u,\theta_{u,t}(x))]du \notag\\
&&+\int_{t}^s [F_\delta(u,\theta_{u,t}(x))-F_\delta(u,\theta_{u,t}(x'))]du+\int_{t}^s [F(u,\theta_{u,t}(x'))-F_\delta(u,\theta_{u,t}(x'))]du,\notag
\end{eqnarray}
where for all $y\in \R^d $, and $\delta>0 $, $F_\delta(y):=\int_{\R^d} F(y-z)\phi_\delta(z)dz $ where $\phi_\delta(z)=\frac{1}{\delta^d}\phi(\frac{z}{\delta}) $, where $\phi :\R^d\in \R^+$ is a standard mollifier, i.e., a smooth non negative function s.t. $\int_{\R^d } \phi(z)dz=1$. It is then clear from the fact that $F\in L^\infty(C^{\beta}) $ that:
\begin{eqnarray}
\forall (u,z)\in [0,T]\times \R^d, \ |F_\delta(u,z)-F(u,z)|&\le& {K_0}\delta^\beta,\notag\\
|DF_\delta|_\infty&\le& {K_0}\delta^{\beta-1}.\label{CTR_MOLL}
\end{eqnarray}
We therefore get:
\begin{eqnarray*}
|\theta_{s,t}(x)-\theta_{s,t}(x')|&\le& |x-x'|+2{K_0}\delta^\beta (s-t)+{K_0}\delta^{-1+\beta}\int_{t}^s |\theta_{u,t}(x)-\theta_{u,t}(x')|\\
&\le &\Big(|x-x'|+2{K_0}\delta^\beta (s-t)\Big)\exp({K_0}\delta^{-1+\beta}(s-t)),
\end{eqnarray*}
using the Gronwall lemma for the last inequality. Choose now $\delta^{-1+\beta}(s-t)=1 \iff (s-t)=\delta^{1-\beta} $ to equilibrate the previous contributions. We eventually derive:
\begin{eqnarray*}
|\theta_{s,t}(x)-\theta_{s,t}(x')|&\le& \Big(|x-x'|+2{K_0} (s-t)^{1+\frac{\beta}{1-\beta}}\Big)\exp({K_0}),\\
&\le &C(|x-x'|+(s-t)^{\frac{1}{1-\beta}})\le C(|x-x'|+(s-t)^{\frac{1}{\alpha}}),
\end{eqnarray*}
recalling for the last inequality that $(s-t)\le T\le 1 $ and $\frac{1}{\alpha}<\frac{1}{1-\beta} $ since $\alpha+\beta>1 $. This gives the result when $F$ is globally $\beta $-H\"older continuous. 

\textcolor{black}{Recalling now that we appeal to this result when $|x-x'|\le \big((T-t)/c_0\big)^{1/\alpha} $ where the r.h.s. is small, provided that $T$ is, it is plain to localize the above computations and to observe that the previous global results are actually also local when  the final time horizon $T$ is small enough and the initial points are close w.r.t. the time scale.
This concludes the proof of the Lemma}.

\subsection{Proof of equation \eqref{CTR_FRAC_OP_HOLDER}}
\label{SEC_PROOF_EST_OP}
Let $\varphi\in C_b^{\gamma+\theta}(\R^d)$ where we recall $\gamma+\theta>1 $ with $\theta\in (0,1],\ \gamma\in (0,1) $.
We aim at proving: 
$$\|L_\theta \varphi \|_{C_b^\gamma}\le C_{\theta,\gamma} \|\varphi\|_{C_b^{\gamma+\theta}}.$$
Write first, for all $x\in \R^d$:
\begin{eqnarray*}
|L_\theta \varphi(x)|&\le& \Big|\int_{|z|\le 1} \Big( \varphi(x+z)-\varphi(x)-D \varphi(x)\cdot z \I_{\theta=1} \Big)\nu_\theta(dz)\Big|+\int_{|z|\ge 1} \|\varphi\|_\infty\nu_\theta(dz)\notag\\
&\le &\big( \int_0^1 d\lambda \int_{{\mathbb S}^{d-1}}\tilde \mu(d\xi) \int_{\rho\in (0,1]}\frac{d\rho}{\rho^{1+\theta}}[ D\varphi(x+\lambda \xi \rho)-D\varphi(x)\I_{\theta=1}] \cdot \xi \rho+C_\theta\|\varphi\|_\infty\big)\\
&\le & C_{\theta,\gamma} (\|D\varphi\|_{\infty}\I_{\theta\in (0,1)}+[D\varphi]_\gamma\I_{\theta=1}+\|\varphi\|_\infty) \le C_{\theta,\gamma}\|\varphi\|_{C_b^{\theta+\gamma}}.
\end{eqnarray*}
This gives the control for the supremum norm. Let us now turn to the H\"older modulus.
Fix, $ x,x'\in \R^d,\ x\neq x'$. We first consider the case $\theta\in (0,1) $ for simplicity. Write:
\begin{eqnarray*}
&&|L_\theta \varphi(x)-L_\theta \varphi(x')|\\
&\le& \Big|\int_0^1d\lambda  \int_{|z|\le |x-x'|} \Big(D\varphi(x+\lambda z) -D\varphi(x'+\lambda z) \Big) \cdot z\nu_\theta(dz) \Big|\\
&&+\Big|\int_{0}^1 d\lambda \int_{|z|\ge |x-x'|} \Big( D \varphi(x'+z+\lambda(x-x'))-D\varphi(x'+\lambda (x-x'))\Big)\cdot (x-x') \nu_\theta(dz)\Big|\\
&\le & C\Big( \int_{\rho\in (0,|x-x'|]} \frac{d\rho}{\rho^{1+\theta}} [D\varphi]_{\theta+\gamma-1}|x-x'|^{\theta+\gamma-1}\rho
+\int_{\rho\ge |x-x'| }\frac{d\rho}{\rho^{1+\theta}}  [D\varphi]_{\alpha+\gamma-1} \rho^{\theta+\gamma-1}|x-x'|\Big) \\
&\le & C_{\theta,\gamma} [D\varphi]_{\theta+\gamma-1} |x-x'|^\gamma.
\end{eqnarray*}
The only modifications needed for $\theta=1 $ concern the \textit{small jumps}. Indeed, we can introduce the compensator only up to the threshold $|x-x'|$. We are simply led to analyze:
\begin{eqnarray*}
&&\Big|\int_0^1d\lambda  \int_{|z|\le |x-x'|} \Big([D\varphi(x+\lambda z) -D\varphi(x) ]-[D\varphi(x'+\lambda z)-D\varphi(x')] \Big) \cdot z\nu_1(dz) \Big|\\
&\le& C\int_{\rho\in (0,|x-x'|]}\frac {d\rho}{\rho^2}[D\varphi]_\gamma \rho^{(1+\gamma-1)+1}\le C_{1,\gamma}[D\varphi]_\gamma|x-x'|^\gamma.
\end{eqnarray*}
The other contribution can be handled as above. 
We have therefore established for all $\theta\in (0,1] $, $\gamma\in (0,1) $ s.t. $\theta+\gamma>1 $:
$$|L_\theta \varphi(x)-L_\theta \varphi(x')|\le C_{\theta,\gamma }[D \varphi]_{\theta+\gamma-1}|x-x'|^\gamma .$$ 
This completes the proof of inequality \eqref{CTR_FRAC_OP_HOLDER}.

\def\c{ 
\subsection{Proof of Lemma \ref{INT_DIAG}}
We split $\R^{d}=\{ y\in \R^d: |y-\theta_{s,t}(x)|\le 2c_0^{-\frac 1\alpha }(s-t)^{\frac 1\alpha}\}\cup \{ y\in \R^d: |y-\theta_{s,t}(x)|> 2c_0^{-\frac 1\alpha }(s-t)^{\frac 1\alpha}\}=:{\mathscr D}_{s,t,c_0}\cup {\mathscr {OD}}_{s,t,c_0} $, which respectively correspond to diagonal and off-diagonal sets when integrating \eqref{CTR_PERT_DIAG}. Write for all $\mu\in [0,1] $:
\begin{eqnarray}
&&\int_{\R^{d}} dy\bar p_\alpha\Big(s-t,y-\big(\theta_{s,t}(x)+(1-\mu)(x'-x) \big)\Big) \left( \frac{|y-\theta_{s,t}(x)|}{(s-t)^{\frac 1\alpha}}\right)^\beta\notag\\
&\le& C\Big[\int_{{\mathscr D}_{s,t,c_0}}\frac{dy}{(s-t)^{\frac d\alpha}}\left( \frac{|y-\theta_{s,t}(x)|}{(s-t)^{\frac 1\alpha}}\right)^\beta \notag\\
& &+\int_{{\mathscr {OD}}_{s,t,c_0}}\frac{dy}{(s-t)^{\frac d\alpha}}\left( \frac{|y-\theta_{s,t}(x)|}{(s-t)^{\frac 1\alpha}}\right)^\beta\frac{1}{\Big(1- \frac{|x-x'|}{(s-t)^{\frac 1\alpha}}+\frac{|y-\theta_{s,t}(x)|}{(s-t)^{\frac 1\alpha}} \Big)^{d+\alpha+1}}\Big].\notag
\end{eqnarray}
Recalling that $|x-x'|/(s-t)^{\frac{1}{\alpha}}\le c_0^{-\frac 1\alpha} $, we obtain:
\begin{eqnarray}
&&\int_{\R^{d}}dy\bar p_\alpha\Big(s-t,y-\big(\theta_{s,t}(x)+(1-\mu)(x'-x) \big)\Big) \left( \frac{|y-\theta_{s,t}(x)|}{(s-t)^{\frac 1\alpha}}\right)^\beta\notag\\
&\le&C \Big[ (2c_0^{-\frac 1\alpha})^{d+\beta}+\int_{{\mathscr {OD}}_{s,t,c_0}}\frac{dy}{(s-t)^{\frac d\alpha}}\left( \frac{|y-\theta_{s,t}(x)|}{(s-t)^{\frac 1\alpha}}\right)^\beta\frac{1}{\Big(1+\frac{|y-\theta_{s,t}(x)|}{2(s-t)^{\frac 1\alpha}} \Big)^{d+\alpha+1}}\Big]\notag\\
&\le&C \Big[ (2c_0^{-\frac 1\alpha})^{d+\beta}+\int_{r\ge 2 c_0^{-\frac 1\alpha}} dr r^{\beta}r^{d-1}r^{-(d+\alpha+1)}\Big]\le C c_0^{-\frac{d+\beta}{\alpha}},
\end{eqnarray}
up to a modification of $C$ and recalling that $c_0\le 1$ for the last inequality. This concludes the proof of the lemma.
}

\bibliographystyle{alpha}
\bibliography{bibli}

\end{document}